\documentclass[a4paper,10pt]{amsart}
\usepackage[utf8]{inputenc}
\usepackage{amsmath,amsfonts,amsthm,amssymb}
\usepackage[dvipsnames]{xcolor}
\usepackage{enumerate}
\usepackage{delarray}
\usepackage{float}
\usepackage{pgf,tikz}
\usetikzlibrary{math}
\usetikzlibrary {arrows.meta}
\usetikzlibrary{cd}
\usepackage{enumerate}
\usepackage{dynkin-diagrams}
\usepackage[pagebackref,bookmarksopen]{hyperref}
\hypersetup{
	colorlinks=true,
	linkcolor=blue,
	citecolor=magenta,
	urlcolor=cyan,
}
\newtheorem{theorem}{Theorem}[section]
\newtheorem{lemma}[theorem]{Lemma}
\newtheorem{proposition}[theorem]{Proposition}

\theoremstyle{definition}

\newtheorem{remark}[theorem]{Remark}

\newcommand{\inpr}[3][]{\ensuremath{\langle#2,#3\rangle_{#1}}}
\DeclareMathOperator{\tr}{trace}
\DeclareMathOperator{\Ab}{Ab}

\newcommand{\MQCnq}[2]{
\begin{tikzpicture}[scale=.36,>=Stealth,auto=right]
\tikzmath{
\n=#1; 
\s=\n-1;
\q=#2;  
\d=gcd(\n,\q);
}
\ifnum \d=1 
    \ifnum \q=1 
        \foreach \i in {0,1,...,\s} {
            \node[circle,fill=blue!20,inner sep=1pt,font=\tiny,draw] (a\i) at ({3*cos(360/\n*\i)},{3*sin(360/\n*\i)}) {$\i$};};
        \foreach \i [evaluate=\i as \j using {int(mod(\i+\q,\n))}] in {0,...,\s} \draw[->,red, thick] (a\i) -- (a\j) node[pos=0.5,black] {$2$};
        \node at (0,-4.5) {$Q_{\rho_{\mathrm{Nat}}}(\mathbb{C}_{\n,\q})$}; 
    \else
        \ifnum \q=\s  
            \foreach \i in {0,1,...,\s} {
            \node[circle,fill=blue!20,inner sep=1pt,font=\tiny,draw] (a\i) at ({3*cos(360/\n*\i)},{3*sin(360/\n*\i)}) {$\i$};};
            \foreach \i [evaluate=\i as \j using {int(mod(\i+\q,\n))}] in {0,...,\s} \draw[red, thick] (a\i) -- (a\j);
             \node at (0,-4.5) {$Q_{\rho_{\mathrm{Nat}}}(\mathbb{C}_{\n,\q})$}; 
        \else   
                \foreach \i in {0,1,...,\s} {
                \node[circle,fill=blue!20,inner sep=1pt,font=\tiny,draw] (a\i) at ({3*cos(360/\n*\i)},{3*sin(360/\n*\i)}) {$\i$};};
                \foreach \i [evaluate=\i as \j using {int(mod(\i+1,\n))}] in {0,...,\s} \draw[->,red, thick] (a\i) -- (a\j);
                \foreach \i [evaluate=\i as \j using {int(mod(\i+\q,\n))}] in {0,...,\s} \draw[->,blue, thick] (a\i) -- (a\j);
                \node at (0,-4.5) {$Q_{\rho_{\mathrm{Nat}}}(\mathbb{C}_{\n,\q})$}; 
        \fi
    \fi
\else
\node[red] at (0,0) {$gcd(n,q)\neq1$}; 
\fi
\end{tikzpicture}
}

\newcommand{\MQCn}[1]{
\begin{tikzpicture}[scale=.36,>=Stealth,auto=right]
\tikzmath{
\n=#1; 
\s=\n-1;
\q=1;  
}
        \foreach \i in {0,1,...,\s} {
            \node[circle,fill=blue!20,inner sep=1pt,font=\tiny,draw] (a\i) at ({3*cos(360/\n*\i)},{3*sin(360/\n*\i)}) {$\i$};};
        \foreach \i [evaluate=\i as \j using {int(mod(\i+\q,\n))}] in {0,...,\s} \draw[->,red, thick] (a\i) -- (a\j);
        \node at (0,-4.5) {$Q_{\beta_{1}}(\mathbf{C}_{\n})$}; 
\end{tikzpicture}
}

\title{McKay quivers of small finite subgroups of $GL(2,\mathbb{C})$}

\author{J.~L.~Cisneros-Molina}
\address{Unidad Cuernavaca, Instituto de Matem\'aticas, Universidad Nacional Aut\'onoma de M\'exico. Av.~Universidad s/n, Lomas de Chamilpa, Cuernavaca, Mexico}
\email{jlcisneros@im.unam.mx}

\author{M.~Tosun}
\address{Department of Mathematics, Galatasaray University\\  Ortak{\"o}y 34357,  Istanbul, Turkey}
\email{mtosun@gsu.edu.tr}

\thanks{The first author was partially supported by Tubitak 2221 grant. The second author was partially supported by UNAM-DGAPA-PREI scholarship.}

\begin{document}

\begin{abstract}
We explicitly compute the McKay quivers of small finite subgroups of $GL(2,\mathbb{C})$ relative to the natural representation, using character theory and the McKay quivers of
finite subgroups of $SU(2)$.
We present examples that shows the rich symmetry and combinatorial structure of these quivers.
We compare our results with the MacKay quivers computed by Auslander and Reiten in \cite{Auslander-Reiten:MQEDD}.
\end{abstract}

\maketitle

\section{Introduction}

McKay quivers connect several areas such as algebraic geometry, representation theory and mathematical physics.
In algebraic geometry they appear in the McKay correspondence \cite{McKay:GSFG,Steinberg:FSGSU2DDACE,Gonzalez-Sprinberg-Verdier:McKay,Artin-Verdier:RMORDP,Esnault-Knorrer:RMORDP,Knorrer:GRRRDP,Kapranov-Vasserot:KSDCHA,Crawley-Holland:NCDKS,Ito-Nakamura:HSSS,Ito-Nakamura:MCHS} and its generalizations \cite{Slodowy:SSSAG,Esnault:RMQSS,Wunram:RMQSS,Khan:RMMES,FernandezdB-Romano:RMNGSSDM,Bridgeland-etal:MCEC,Ito-Reid:MCFGSL3C,Reid:CM}.
In representation theory of algebras they appear in the study of the Auslander-Reiten quivers of Cohen-Macaulay modules \cite{Auslander:RSASS,Auslander-Reiten:MQEDD}.
In physics, they have been used  to understand string compactifications \cite{Lawrence-etal:CFTFD,Hanany-He:NAFGT,Feng-etal:DTNAOSM,Feng-etal:DTCGQD,Aspinwall-Plesser:DBDTMC} through the McKay correspondence.

McKay quivers were introduced by McKay in \cite{McKay:GSFG}, where he presented what it is known as the McKay correspondence: a relation between McKay quivers of finite subgroups $\Gamma$ of $SU(2)$ relative to the natural representation and the dual graph of the minimal resolution of the corresponding Kleinian singularity $S_\Gamma=\mathbb{C}^2/\Gamma$.
In this case, the underlying graphs of the McKay quivers are the extended Dynkin diagrams of types $\tilde{A}$, $\tilde{D}$ and $\tilde{E}$.

The goal of the present article is to construct the McKay quivers of small finite subgroups $G$ of $GL(2,\mathbb{C})$, for which the quotient $S_G=\mathbb{C}^2/G$ is a surface singularity.
In \cite{Auslander:RSASS} Auslander proved that the McKay quiver of such a group is isomorphic to the Auslander-Reiten quiver of the reflexive modules of the associated quotient singularity $S_G$, and in \cite{Auslander-Reiten:MQEDD}  to illustrate the main theorem, Auslander and Reiten give the combinatorial structure of these McKay quivers, using the description of the small finite subgroups of $GL(2,\mathbb{C})$ given by Brieskorn in \cite{Brieskorn:RSKF}, the proof is concise and no examples of the quivers were included.
In \cite{Arciniega-etal:CRMQSAPS} it was proved that the list of small finite subgroups of $GL(2,\mathbb{C})$ coincides with the list of non-trivial finite subgroups of $SO(4)$ acting freely and isometrically on the sphere $\mathbb{S}^3$. This latter list consists of the finite subgroups of $SU(2)$, two families of groups denoted $D_{2^k(2r+1)}$ and $P'_{8\cdot 3^k}$, and the direct product of any of these groups with a cyclic group of relatively prime order. These groups are given by presentations, this allowed the authors to write explicitly the irreducible representations of the groups $D_{2^k(2r+1)}$ and $P'_{8\cdot 3^k}$.
Here we construct the McKay quivers of the small finite subgroups of $GL(2,\mathbb{C})$ as follows:
\begin{enumerate}
 \item Using the irreducible representations of the groups $D_{2^k(2r+1)}$ and $P'_{8\cdot 3^k}$ given in \cite{Arciniega-etal:CRMQSAPS}, we compute their character tables and then the McKay quivers using character theory.
 \item We compute the McKay quivers of the rest of the small finite subgroups of $GL(2,\mathbb{C})$ using Theorem~\ref{thm:MQ.prod} which gives the McKay quiver of a direct product of groups from the McKay quivers of the factor groups.
\end{enumerate}
We illustrate examples of the different types of these McKay quivers.

In Section~\ref{sec:MQ} we define McKay quivers, we explain how they can be computed using character theory. We also state a result on connectivity and describe how to compute the McKay quiver of a direct product of groups from the McKay quivers of the factor groups.
In Section~\ref{sec:MC} we list the finite subgroups of $SU(2)$ and their McKay quivers.
In Section~\ref{sec:QS} we give the classification of the small finite subgroups of $GL(2,\mathbb{C})$.
In Section~\ref{sec:CT} we compute the conjugacy classes and the character tables of the families of groups $D_{2^k(2r+1)}$ and $P'_{8\cdot 3^k}$.
In Section~\ref{sec:MQ.all} we explicitly compute the McKay quivers of the small finite subgroups of $GL(2,\mathbb{C})$. Finally in Section~\ref{sec:AR} we compare our results with the McKay quivers given in \cite{Auslander-Reiten:MQEDD}.

\section{McKay quivers}\label{sec:MQ}

A finite group $G$ has a finite number of complex irreducible representations which is equal to the number of conjugacy classes of $G$ \cite[Theorem~15.3]{James-Liebeck:RCG}.
Let $\mathrm{Irr}(G)=\{\rho_0,\rho_1,\dots,\rho_r\}$ be the set of complex \textit{irreducible representations}\index{representation!irreducible}
of $G$, where $\rho_0$ denotes the \textit{trivial representation}\index{representation!trivial}.
The relation between the dimensions $n_i$ of the irreducible representations $\rho_i\colon G\to GL(n_i,\mathbb{C})$ and the order $|G|$ of the group $G$ is given by
\cite[Theorem~11.12]{James-Liebeck:RCG}
\begin{equation}\label{eq:irr.dim.ord}
\sum_{i=0}^r n_i^2=|G|.
\end{equation}

Let $\rho$ be a (possibly reducible) representation of $G$.
Consider the tensor products $\rho\otimes\rho_i$ for $i=0,\dots,r$,  by Maschke's theorem \cite[Theorems~8.1 \& 8.7]{James-Liebeck:RCG} they decompose as direct sum of irreducible representations\index{representation!irreducible}
\begin{equation}\label{eq:rho.rhoi}
\rho\otimes\rho_i=\bigoplus_{j=0}^r a_{ij}\rho_j,\qquad j=0,\dots,r,
\end{equation}
where $a_{ij}\in\mathbb{N}$ is the multiplicity of $\rho_j$ in $\rho\otimes\rho_i$. The \textit{McKay matrix} of $G$ relative to $\rho$ is defined by $A_\rho(G)=\{a_{ij}\}_{i,j=0}^r$.
With the information given by $A_\rho(G)$ we construct the \textit{McKay quiver of $G$ relative to $\rho$}\index{McKay quiver}, denoted by $Q_\rho(G)$, as follows:
associate a vertex to each irreducible representation $\rho_i$, and join the $i$-th vertex to the $j$-th vertex by $a_{ij}$ arrows. We take the convention that an undirected edge between two vertices,
represents a pair of arrows between those vertices pointing in opposite directions.

\subsection{Computing McKay quivers using character theory}\label{ssec:MQ.CT}

Recall that the \textit{character} $\chi_\rho\colon G\to\mathbb{C}$ of a representation $\rho\colon G\to GL(n,\mathbb{C})$ of $G$ is given by $\chi_\rho(g)=\tr(\rho(g))$.
For simplicity we denote the characters of the irreducible representations $\rho_i$ by $\chi_i$ and we call them \textit{irreducible characters}.
Representations of $G$ are characterized by their characters.
Let $\rho$ and $\sigma$ be representations of $G$ and $\chi_\rho$ and $\chi_\sigma$ their corresponding characters. Then $\rho$ and $\sigma$ are isomorphic if and only if $\chi_\rho=\chi_\sigma$
\cite[Theorem~14.21]{James-Liebeck:RCG}.
One can take the direct sum $\rho\oplus\sigma$ and tensor product $\rho\otimes\sigma$ of the representations $\rho$ and $\sigma$ and their characters are given, respectively, by the sum and product of the corresponding characters $\chi_\rho$ and $\chi_\sigma$, that is \cite[(7.10) \& Proposition~19.6]{James-Liebeck:RCG}
\begin{equation}\label{eq:sum.prod}
\chi_{\rho\oplus\sigma}(g)=\chi_\rho(g)+\chi_\sigma(g),\quad
\chi_{\rho\otimes\sigma}(g)=\chi_\rho(g)\chi_\sigma(g),\quad\text{for all $g\in G$.}
\end{equation}

There is an inner product of characters given by
\begin{equation}\label{eq:inprod}
\inpr[G]{\chi_\rho}{\chi_\sigma}=\frac{1}{|G|}\sum_{g\in G}\chi_\rho(g)\overline{\chi_\sigma(g)}.
\end{equation}
The characters of the irreducible representations form an orthonormal set with respect to the inner product \eqref{eq:inprod}, that is \cite[Theorem~14.12]{James-Liebeck:RCG}
\begin{equation}\label{eq:ortho}
\inpr[G]{\chi_i}{\chi_j}=\begin{cases}
                       1, & i=j,\\
                       0, & i\neq j.
                      \end{cases}
\end{equation}
Moreover, we have that a representation $\rho$ with character $\chi_\rho$ is irreducible if and only if $\inpr[G]{\chi_\rho}{\chi_\rho}=1$ \cite[Theorem~14.20]{James-Liebeck:RCG}.

Taking the character of the representation given in \eqref{eq:rho.rhoi} and using \eqref{eq:sum.prod} we get
\begin{equation*}
\chi_\rho\chi_i=\sum_{i=0}^r a_{ij}\chi_j,
\end{equation*}
and taking the inner product with the character $\chi_j$ by the orthogonality relations \eqref{eq:ortho} we obtain
\begin{equation}\label{eq:aij}
 a_{ij}=\inpr[G]{\chi_\rho\chi_i}{\chi_j}.
\end{equation}
Hence, in order to compute the McKay quiver $Q_\rho(G)$ of $G$ relative to the representation $\rho$, we only need the character $\chi_\rho$ of $\rho$ and the character table of $G$.
We associate a vertex to each irreducible character $\chi_i$, and join the $i$-th vertex to the $j$-th vertex by $a_{ij}$ arrows with $a_{ij}$ given by \eqref{eq:aij}.

\subsection{Some results of McKay quivers}

Here we present some results of McKay quivers that we shall use in the sequel.

\subsubsection{Connectivity}
The first result is about connectivity of the McKay quiver (see also \cite[Proposition~3.3 \& Proposition~3.10]{Browne:CPMQ}).

\begin{proposition}[{\cite[Proposition~1]{McKay:GSFG}}]\label{prop:connect}
The McKay quiver $Q_\rho(G)$ is connected if and only of $\rho$ is a faithful representation.
\end{proposition}

\subsubsection{McKay quivers of direct products of groups}

Let $G$ and $H$ be finite groups. Let $V$ be a representation of $G$ and $W$ be a representation of $H$. Consider the tensor product $V\otimes W$,
the action of $G\times H$ on $V\otimes W$ given by
\begin{equation*}
 (g,h)(v\otimes w)=gv\otimes hw,\quad (g,h)\in G\times H,\ v\in V,\  w\in W,
\end{equation*}
makes $V\otimes W$ a representation of $G\times H$ \cite[page~206]{James-Liebeck:RCG}.
Let $\chi_V$ be the character of $V$ and $\chi_W$ be the character of $W$, the character of $V\otimes W$ is $\chi_V \times\chi_W$ \cite[page~206 \& Proposition~19.6]{James-Liebeck:RCG} with
\begin{equation}
 (\chi_V \times\chi_W)(g,h)=\chi_V(g)\chi_W(h),\quad g\in G,\ h\in H.
\end{equation}

\begin{theorem}[{\cite[Theorem~19.18]{James-Liebeck:RCG}}]\label{thm:CT.prod}
Let $\chi_0,\dots,\chi_r$ be the distinct irreducible characters of $G$ and let $\psi_0,\dots,\psi_s$ be the distinct irreducible characters of $H$.
Then $G\times H$ has precisely $(r+1)(s+1)$ distinct irreducible characters, these are
\begin{equation}
 \chi_i\times\psi_k,\quad 0\leq i\leq r,\ 0\leq k\leq s.
\end{equation}
\end{theorem}

The following theorem allows us to compute the McKay quiver of a direct product of groups relative to the tensor product of representations of each factor group.

\begin{theorem}\label{thm:MQ.prod}
Let $\chi_0,\dots,\chi_r$ be the distinct irreducible characters of $G$ and let $\psi_0,\dots,\psi_s$ be the distinct irreducible characters of $H$.
Let $\rho$ be a representation of $G$ with character $\chi_\rho$ and $\sigma$ be a representation of $H$ with character $\psi_\sigma$.
Then in the McKay quiver $Q_{\rho\otimes\sigma}(G\times H)$ of $G\times H$ relative to the representation $\rho\otimes\sigma$,
there are $a_{ij}b_{kl}$ arrows from $\chi_i\times\psi_k$ to $\chi_j\times\psi_l$
if and only if there are $a_{ij}$ arrows from $\chi_i$ to $\chi_j$ in the McKay quiver $Q_\rho(G)$ of $G$ relative to $\rho$ and $b_{kl}$ arrows from
$\psi_k$ to $\psi_l$ in the McKay quiver $Q_\sigma(H)$ of $H$ relative to $\sigma$.
\end{theorem}

\begin{proof}
By \eqref{eq:aij} there are $c_{ik,jl}$ arrows from $\chi_i\times\psi_k$ to $\chi_j\times\psi_l$ with
\begin{align*}
c_{ik,jl}&=\inpr[G\times H]{(\chi_\rho\times\psi_\sigma)(\chi_i\times\psi_k)}{\chi_j\times\psi_l}\\
         &=\frac{1}{|G\times H|}\sum_{\substack{g\in G\\h\in H}}\chi_\rho(g)\psi_\sigma(h)\chi_i(g)\psi_k(h)\overline{\chi_j(g)\psi_l(h)}\\
         &=\Bigl(\frac{1}{|G|}\sum_{g\in G}\chi_\rho(g)\chi_i(g)\overline{\chi_j(g)}\Bigr)\Bigl(\frac{1}{|H|}\sum_{h\in H}\psi_\sigma(h)\psi_k(h)\overline{\psi_l(h)}\Bigr)\\
         \intertext{again by \eqref{eq:aij}}
         &=a_{ij}b_{kl}.\qedhere
\end{align*}
\end{proof}

\section{McKay quivers of finite subgroups of $SU(2)$}\label{sec:MC}

In this section we list the finite subgroups of $SU(2)$ and list their McKay quivers.

\subsection{Finite subgroups of $SU(2)$}\label{ssec:fsgSU2}

The finite subgroups of $SU(2)$ were classified (up to conjugation) by F.~Klein in his book \cite{Klein:Icosahedron} as follows.
Let $SO(3)$ be the group of rotations of $\mathbb{R}^3$, there is a surjective homomorphism $\rho\colon SU(2)\to SO(3)$ with kernel of order $2$.
The finite subgroups of $SO(3)$ are very well known, they are the cyclic groups $\mathbf{C}_q$ of order $q$ ($q\geq 2$),
the dihedral groups $\mathbf{D}_{2q}$ of order $2q$ ($q\geq2$), and the rotation groups of the Platonic solids:
the \textit{tetrahedral group} $\mathbf{T}$ of order $12$, the \textit{octahedral group} $\mathbf{O}$ of order $24$ and the \textit{icosahedral group} $\mathbf{I}$ of order $60$.
If $G$ is a finite subgroup of $SO(3)$, then $BG=\rho^{-1}(G)$ is a finite subgroup of $SU(2)$, we say that $BG$ is the \textit{binary group}\index{binary group} of $G$ since its order is twice the order of $G$.
The finite subgroups of $SU(2)$ are given in Table~\ref{tb:fin.subgr.SU2} \cite{Cayley:Poly,Klein:Icosahedron,Lamotke:RSIS}:

\begin{table}[H]
\begin{equation*}
\setlength{\extrarowheight}{5pt}
\begin{array}{|c|c|c|c|}\hline
\text{Name} & \text{Notation} & \text{Order} & \\\hline
\text{Cyclic groups} & \mathbf{C}_q & q & q\geq 2\\\hline
\text{Binary dihedral groups} & \mathrm{B}\mathbf{D}_{2q} & 4q & q\geq 2\\\hline
\text{Binary tetrahedral group} & \mathrm{B}\mathbf{T} & 24 &\\\hline
\text{Binary octahedral group} & \mathrm{B}\mathbf{O} & 48 &\\\hline
\text{Binary icosahedral group} & \mathrm{B}\mathbf{I} &120 &\\\hline
 \end{array}
\end{equation*}
\caption{Finite subgroups of $SU(2)$.}\label{tb:fin.subgr.SU2}
\end{table}

\subsection{Character tables of finite subgroups of $SU(2)$}\label{ssec:char.tab.SU2}

In this article we only need the character table of cyclic groups since the quivers $Q_{\rho_{\mathrm{Nat}}}(\Gamma)$ are known and we shall list them later.
One can find the character tables of all the finite subgroups of $SU(2)$ in \cite{Cisneros:EITDOSG,Lindh:McKayCorr} and the irreducible representations in \cite[Appendix]{Cisneros-Romano:RMC} or in \cite{Lindh:McKayCorr}.

\subsubsection{\textbf{Cyclic groups.} $\mathbf{C}_n$, $n\geq 2$.}\label{sssec:Cn}

We identify $\mathbf{C}_n$ with the $n$-th roots of unity. Let $\zeta_{n}=e^{\frac{2\pi i}{n}}$ be a $n$-th primitive root of unity.
The group $\mathbf{C}_n$ has $n$ irreducible representations $\beta_j$, $0\leq j< n$ given by $\beta_j(\zeta_{n})=\zeta_n^{j}$.
The natural representation $\rho_{\mathrm{Nat}}\colon\mathbf{C}_n\to SU(2)$ is given by the direct sum $\rho_{\mathrm{Nat}}=\beta_1\oplus\beta_{n-1}$.
Since each conjugacy class has just one element, we denote the class by such element. The character table of $\mathbf{C}_n$ is given in Table~\ref{tab:CT.Cq} where we denote by $\chi_j$ the character of the representation $\beta_j$.

\begin{table}[H]
\begin{equation*}
\setlength{\extrarowheight}{2pt}
\begin{array}{|c|c|c|c|c|c|c|}\hline
\text{\small Class} & 1 & \zeta_{n} &\cdots & \zeta_{n}^l &\cdots & \zeta_{n}^{n-1}\\\hline
\chi_j & 1& \zeta_n^{j} &\cdots &  \zeta_n^{lj} & \cdots&\zeta_{n}^{(n-1)j}\\\hline
\multicolumn{7}{l}{\text{Note: $0\leq j\leq n-1$,  $0\leq l\leq n-1$.}}
\end{array}
\end{equation*}
\caption{Character table of $\mathbf{C}_n$.}\label{tab:CT.Cq}
\end{table}

\subsection{McKay quivers of finite subgroups of $SU(2)$}\label{ssec:MQ.SU2}

Let $\Gamma$ be a finite subgroup of $SU(2)$. Let $\rho_{\mathrm{Nat}}\colon\Gamma\hookrightarrow SU(2)$ be the \textit{natural representation} given by the inclusion.
We are interested in the McKay quiver $Q_{\rho_{\mathrm{Nat}}}(\Gamma)$ of $\Gamma$ relative to $\rho_{\mathrm{Nat}}$, which were first given by McKay in \cite{McKay:GSFG}.
For all the finite subgroups $\Gamma$ of $SU(2)$ the natural representation $\rho_{\mathrm{Nat}}$ is irreducible,
except for the cyclic case, where it is the direct sum of a faithful irreducible representation and its dual.

Using the character tables given in \cite{Cisneros:EITDOSG} and \eqref{eq:aij} one can check that the McKay matrices $A_{\rho_{\mathrm{Nat}}}(\Gamma)$ satisfy the following properties:
\begin{enumerate}[1.]
\item $a_{ij}=a_{ji}$,\label{it:undirected}
\item $a_{ii}=0$,\label{it:noloop}
\item $a_{ij}=\{0,1\}$.\label{it:simple}
\end{enumerate}
Since the representation $\rho_{\mathrm{Nat}}$ is faithful, by Proposition~\ref{prop:connect} the McKay quiver $Q_{\rho_{\mathrm{Nat}}}(\Gamma)$ is connected, by \ref{it:undirected} it is
an undirected graph, by \ref{it:noloop} it has no self-loops and by \ref{it:simple} is a simple graph, that is, there is only up to one (undirected) edge between vertices.
The McKay quivers $Q_{\rho_{\mathrm{Nat}}}(\Gamma)$ are the graphs listed in Table~\ref{tab:MK.graph} and they correspond to the \textit{extended Dynkin diagrams} of type $\tilde{A}$, $\tilde{D}$, $\tilde{E}$.
In other words, the McKay Quivers $Q_{\rho_{\mathrm{Nat}}}(\Gamma)$ give a bijection between isomorphism classes of finite subgroups of $SU(2)$ and extended Dynkin diagrams of type $\tilde{A}$, $\tilde{D}$ and $\tilde{E}$.
\begin{table}[H]
\begin{equation*}
\setlength{\extrarowheight}{3pt}
\begin{array}{|c|c|c|c|c|}\hline
\mathsf{C}_q\leftrightarrow\tilde{A}_{q-1} & \mathsf{BD}_{2q}\leftrightarrow\tilde{D}_{q+2} & \mathsf{BT}\leftrightarrow\tilde{E}_6 & \mathsf{BO}\leftrightarrow\tilde{E}_7 & \mathsf{BI}\leftrightarrow\tilde{E}_8 \\\hline
\tikz\dynkin[scale=1,extended]A{}\dynkinLabelRoot{0}{\rho_0}; &
\tikz\dynkin[scale=1,extended]D{}\dynkinLabelRoot{0}{\rho_0}; &
\tikz\dynkin[scale=1,extended]E6{}\dynkinLabelRoot{0}{\rho_0}; &
\tikz\dynkin [scale=1,extended]E7{}\dynkinLabelRoot{0}{\rho_0}; &
\tikz\dynkin [scale=1,extended]E8{}\dynkinLabelRoot{0}{\rho_0};\\\hline
\end{array}
\end{equation*}
\caption{McKay graphs $Q_{\rho_{\mathrm{Nat}}}(\Gamma)$ of finite subgroups $\Gamma$ of $SU(2)$.}\label{tab:MK.graph}
\end{table}

\section{Small finite subgroups of $GL(2,\mathbb{C})$}\label{sec:QS}

The classification of finite subgroups of  $GL(2,\mathbb{C})$ can be found in \cite{DuVal:HQR} obtained from the classification of finite subgroups of $GL(4,\mathbb{R})$ by Goursat in \cite{Goursat:SOD}.
The list of finite subgroups of $GL(4,\mathbb{R})$ in \cite{DuVal:HQR} is slightly incomplete, as it is pointed out in \cite{Coxeter:RCP,Conway-Smith:QOGAS} where the complete classification is given, but this does not affect the classification of finite subgroups of $GL(2,\mathbb{C})$.

\subsection{Classification of small finite subgroups of $GL(2,\mathbb{C})$}

An element of the group $GL(2,\mathbb{C})$ is a \textit{pseudo-reflection} if it fixes a line, that is, if it has $1$ as an eigenvalue.
In \cite{Prill:LCQCMDG} Prill call a subgroup $G$ of $GL(2,\mathbb{C})$ \textit{small} if no $g\in G$  is a pseudo-reflection.
The list of small finite subgroups of $GL(2,\mathbb{C})$ is given in \cite[Satz~2.9]{Brieskorn:RSKF} (see also \cite{Pe:ONPQSS}). Here we sketch the construction of this list, which also explains the notation to enumerate these groups, for more details see \cite[Appendix~A]{Pe:ONPQSS}.

First notice that any finite subgroup of $GL(2,\mathbb{C})$ is conjugate to a finite subgroup of $U(2)$, the group of $2\times2$ unitary matrices \cite[Lemma~A.17]{Pe:ONPQSS}.

\subsubsection{Finite subgroups of $U(2)$}\label{ssec:fs.U2}

Let $SU(2)$ be the subgroup of matrices in $U(2)$ with determinant $1$.
Any element of $U(2)$ can be written in the following two ways \cite[Lemma~A.18]{Pe:ONPQSS}:
\begin{equation*}
\lambda\begin{pmatrix}
 a & -\bar{c}\\
 c & \bar{a}
\end{pmatrix},\qquad -\lambda \begin{pmatrix}
                               -a & \bar{c}\\
                               -c & -\bar{a}
                              \end{pmatrix}
\end{equation*}
with $\lambda\in\mathbb{S}^1$ and the matrices in $SU(2)$.
Consider the $2:1$ homomorphism given by
\begin{equation}\label{eq:PHI}
 \begin{split}
\Phi\colon\mathbb{S}^1\times SU(2)&\to U(2)\\
\left(\lambda,\begin{pmatrix}
               a & -\bar{c}\\
               c & a
              \end{pmatrix}
\right)&\mapsto \lambda \begin{pmatrix}
               a & .\bar{c}\\
               c & a
              \end{pmatrix}.
 \end{split}
\end{equation}

Given a finite subgroup $G< U(2)$, consider the subgroup of $\mathbb{S}^1\times SU(2)$
\begin{equation*}
 H_G\colon\Phi^{-1}(G)=\{(\lambda,r)\,|\, \lambda r\in G\}.
\end{equation*}
From $H_G$ we can consider the following finite subgroups
\begin{align*}
 L&=\{\lambda\in \mathbb{S}^1\,|\, \exists r\in SU(2), (\lambda,r)\in H_G\}<\mathbb{S}^1,\\
 L_K&=\{\lambda\in \mathbb{S}^1\,|\, (\lambda,1)\in H_G\}<L<\mathbb{S}^1,\\
 R&=\{r\in SU(2)\,|\, \exists \lambda\in \mathbb{S}^1, (\lambda,r)\in H_G\}<SU(2),\\
 R_K&=\{r\in SU(2)\,|\,  (1,r)\in H_G\}<R<SU(2).
\end{align*}
We have that the homomorphism
\begin{equation*}
 \phi\colon L/L_K\to R/R_K
\end{equation*}
defined by $\phi([\lambda])=[r]$ if $(\lambda,r)\in H_G$ is an isomorphism.

Conversely, given the following data:
\begin{itemize}
 \item a finite subgroup $L$ of $\mathbb{S}^1$ and a normal subgroup $L_K$ of $L$,
 \item a finite subgroup $R$ of $SU(2)$ and a normal subgroup $R_K$ of $R$,
 \item an isomorphism $\phi\colon L/L_K\to R/R_K$
\end{itemize}
we can define a subgroup $H$ of $\mathbb{S}^1\times SU(2)$ by
\begin{equation*}
 (\lambda,r)\in H \Longleftrightarrow \phi([\lambda])=[r]
\end{equation*}
which gives the subgroup $G_H=\Phi(H)$ of $U(2)$.

Apart from the choice of isomorphism $\phi$, that may not be unique, they give mutually inverse constructions.

\paragraph{\textbf{Notation:}} The subgroup $H$ and also the associated subgroup $G=\Phi(H)$ of $U(2)$ will be denoted by $ (L,L_K;R,R_K)_\phi$.

Thus, the classification of finite subgroups of $U(2)$, and therefore, the classification of finite subgroups of $GL(2,\mathbb{C})$, is the same as the classification of the subgroups
$G=(L,L_K;R,R_K)_\phi$.  

\subsection{Small finite subgroups of $U(2)$}

The list of the subgroups $(L,L_K;R,R_K)_\phi$ is given by the first 9 families of groups in \cite[p.~57]{DuVal:HQR} (see also \cite[p.~98]{Coxeter:RCP}), among these groups,
Brieskorn found in \cite[Satz~2.9]{Brieskorn:RSKF} the conditions in order to have \textbf{small} subgroups. It is not easy to work with these groups using the notation $G=(L,L_K;R,R_K)_\phi$;
in \cite[p.~38]{Riemenschneide:IEUGL2C} Riemenschneider divides Brieskorn's list of small finite subgroups of $U(2)$ in five families and gives their generators, which allowed him to find
a minimal set of generators of the ring of $G$-invariant polynomials and a minimal set of generators for the relations between the invariant polynomials. Although in \cite{Riemenschneide:IEUGL2C} generators are given for such groups, the relations are missing in order to give a presentation of the groups, which is useful to find the irreducible representations.
In \cite{Arciniega-etal:CRMQSAPS} the authors prove that the list of small finite subgroups of $U(2)$ coincides with the list of finite subgroups of $SO(4)$ acting freely and isometrically on the $3$-dimensional sphere $\mathbb{S}^3$, that is, the fundamental groups of spherical $3$-manifolds (see \cite[\S1.7]{Aschenbrenner-etal:3MG} or \cite[\S3]{Milnor:GASNWFP}), these groups are described in terms of presentations.

In order to give the list of small finite subgroups of $U(2)$ we need to introduce three families of groups. 

\subsubsection{\textbf{Cyclic subgroups of $U(2)$}}

Every non-trivial subgroup $(L,L_K;R,R_K)_\phi$ of $U(2)$ with $L$ and $R$ cyclic is conjugate to one of the following cyclic groups
\begin{equation*}
 \mathbb{C}_{n,q}=\left\{
 \begin{pmatrix}
  \zeta_n & 0\\
  0 & \zeta_n^q
 \end{pmatrix} \right\}\qquad 0<q<n,\quad (n,q)=1,
\end{equation*}
where $\zeta_n$ is the primitive root of unity $e^{\frac{2\pi i}{n}}$. The groups $\mathbb{C}_{p,q}$ with $0<q<n$ and $(n,q)=1$ are small if $n>1$.

Two groups $\mathbb{C}_{n,q}$ and $\mathbb{C}_{n',q'}$ with $(n,q)=(n',q')=1$ are conjugate if and only if $n=n'$ and either $q=q'$ or $qq'\equiv 1 \mod n$ \cite[Lemma~A.12]{Pe:ONPQSS}.

\subsubsection{\textbf{Small finite subgroups of $GL(2,\mathbb{C})$}}\label{sssec:sfsU2}

Let $D_{2^{k}(2r+1)}$ and $P'_{8\cdot 3^k}$ be the groups given by the presentations
\begin{gather*}
D_{2^{k}(2r+1)}:=\langle x,y\mid x^{2^{k}}=1, y^{2r+1}=1, xyx^{-1}=y^{-1}\rangle,\ \text{$k>2$, $r\geq1$,}\\
P'_{8\cdot 3^k}:=\langle x,y,z\mid x^2=(xy)^2=y^2,zxz^{-1}=y, zyz^{-1}=xy, z^{3^k}=1\rangle,\ \text{$k\geq2$.}
\end{gather*}

Now we can list the small finite subgroups of $GL(2,\mathbb{C})$, see \cite[Satz~2.9]{Brieskorn:RSKF}, \cite[p.~38]{Riemenschneide:IEUGL2C}, \cite[\S1.7]{Aschenbrenner-etal:3MG}, \cite[Remark~2.6 \& Theorem~2.7]{Arciniega-etal:CRMQSAPS}.

\begin{theorem}\label{thm:sfs}
Each small finite subgroup of $GL(2,\mathbb{C})$ is conjugate to one of the following groups:
\begin{align*}
\intertext{\textbf{Cyclic groups.} Order $n$. Let $0<q<n$ with $\gcd(n,q)=1$,}
 \mathbb{C}_{n,q}&=\left\langle\left(\begin{smallmatrix}
                                                 \zeta_n & 0 \\ 0 & \zeta_n^q
                                                \end{smallmatrix}\right)\right\rangle.\\
\intertext{\textbf{Dihedral groups.} Order $4qm$. Let $0<q<n$ with $\gcd(n,q)=1$, }
 \mathbb{D}_{n,q}&=\begin{cases}
                   \gcd(m,2)=1, m=n-q: &\\
                   (\mathbf{C}_{2m},\mathbf{C}_{2m};\mathrm{B}\mathbf{D}_{2q},\mathrm{B}\mathbf{D}_{2q})\cong\mathrm{B}\mathbf{D}_{2q}\times\mathbf{C}_m,
                   & \\
                   \gcd(m,2)=2, \text{$m=n-q=2^{k-2}l$ with $l$ odd and $k\geq3$:}&\\
                   (\mathbf{C}_{4m},\mathbf{C}_{2m};\mathrm{B}\mathbf{D}_{2q},\mathbf{C}_{2q})\cong D_{2^{k}\cdot q}\times\mathbf{C}_l.
                   &
                  \end{cases}\\
\intertext{\textbf{Tetrahedral groups.} Order $24m$.}
 \mathbb{T}_m&=\begin{cases}
                \gcd(m,6)=1: &\\
               (\mathbf{C}_{2m},\mathbf{C}_{2m};\mathrm{B}\mathbf{T},\mathrm{B}\mathbf{T})\cong\mathrm{B}\mathbf{T}\times\mathbf{C}_m,
               & \\
               \gcd(m,6)=3, \text{$m=3^{k-1}l$ with $l$ odd, $\gcd(3,l)=1$ and $k\geq2$:}&\\
               (\mathbf{C}_{6m},\mathbf{C}_{2m};\mathrm{B}\mathbf{T},\mathrm{B}\mathbf{D}_{4})\cong P'_{8\cdot 3^k}\times\mathbf{C}_l.
               &\\
              \end{cases}\\
\intertext{\textbf{Octahedral groups.} Order $48m$.}
 \mathbb{O}_m&=(\mathbf{C}_{2m},\mathbf{C}_{2m};\mathrm{B}\mathbf{O},\mathrm{B}\mathbf{O})\cong\mathrm{B}\mathbf{O}\times\mathbf{C}_m,
 \gcd(m,6)=1.\\
\intertext{\textbf{Icosahedral groups.} Order $120m$.}
 \mathbb{I}_m&=(\mathbf{C}_{2m},\mathbf{C}_{2m};\mathrm{B}\mathbf{I},\mathrm{B}\mathbf{I})\cong\mathrm{B}\mathbf{I}\times\mathbf{C}_m,
 \gcd(m,30)=1.
\end{align*}
\end{theorem}

\begin{remark}\label{rem:GxC}
From Theorem~\ref{thm:sfs} we can see that the small finite subgroups of $U(2)$ are $\mathrm{B}\mathbf{D}_{2q}$, $\mathrm{B}\mathbf{T}$, $\mathrm{B}\mathbf{O}$, $\mathrm{B}\mathbf{I}$, $D_{2^{k+1}(2r+1)}$ and $P'_{8\cdot 3^k}$; the direct product of any of the previous groups with a cyclic group of relatively prime order; and the cyclic groups $\mathbb{C}_{n,q}$. See \cite[Remark~2.6 \& Theorem~2.7]{Arciniega-etal:CRMQSAPS}.
\end{remark}

\section{Character tables of $D_{2^{k+1}(2r+1)}$ and $P'_{8\cdot 3^k}$}\label{sec:CT}

In this section, using the presentations of the groups $D_{2^{k+1}(2r+1)}$ and $P'_{8\cdot 3^k}$, we compute their conjugacy classes, and using their irreducible representations given in \cite[\S3.3]{Arciniega-etal:CRMQSAPS} we compute their character tables.

\begin{remark}[\textbf{Notation}]\label{rem:Notation}
We shall denote the conjugacy classes of the group $G$ by the number of elements they have.
If two different conjugacy classes have the same number of elements we shall use a subindex to differentiate them.
\end{remark}

\subsection{Groups $\mathbb{C}_{n,q}$}\label{ssec:Cnq}

The irreducible representations and the character table of the cyclic groups $\mathbb{C}_{n,q}$ are given in subsubsection~\ref{sssec:Cn}. 
We only need to note that the natural representation
$\rho_{\mathrm{Nat}}\colon\mathbb{C}_{n,q}\to U(2)$ is given by the direct sum $\rho_{\mathrm{Nat}}=\beta_1\oplus\beta_{q}$.

\subsection{Groups $D_{2^k(2r+1)}$}\label{ssec:D2k2r+1}

This family of groups belongs to the list of finite subgroup of $SO(4)$ acting freely on $\mathbb{S}^3$ see \cite[\S1.7]{Aschenbrenner-etal:3MG}.
The group $D_{2^{k}\cdot q}$  corresponds to the small dihedral group $\mathbb{D}_{2^{k-2}+q,q}$, with $q$ odd, in Theorem~\ref{thm:sfs}.

\subsubsection{Presentation}

The group $D_{2^k(2r+1)}$ has order $2^k(2r+1)$ and a presentation is
\begin{equation}\label{eq:f83e6c}
 D_{2^k(2r+1)}:=\langle x,y\mid x^{2^k}=1, y^{2r+1}=1, xyx^{-1}=y^{-1}\rangle,
\end{equation}
where $k>2$ and $r\geq1$.

\begin{remark}
When $k=2$ there is an isomorphism between $D_{4(2r+1)}$ and the binary dihedral group $BD_{2(2r+1)}$ \cite[Remark~2.3]{Arciniega-etal:CRMQSAPS}.
\end{remark}

Let $[G,G]$ be the \textit{commutator subgroup} of $G$. Let $\Ab\colon G\to G/[G,G]$
be the projection homomorphism, and denote the \textit{abelianization} of $G$ by $\Ab(G)=G/[G,G]$.

\begin{lemma}[{\cite[Lemma~3.1]{Arciniega-etal:CRMQSAPS}}]\label{lem:Ab.D}
The abelianization of $D_{2^{k}(2r+1)}$ is $\Ab(D_{2^{k}(2r+1)})=\mathbf{C}_{2^{k}}$, where $\Ab(x)$ is the generator of $\mathbf{C}_{2^{k}}$ and $\Ab(y)=1$.
\end{lemma}

\subsubsection{Conjugacy classes}\label{conjugacy-classes}

To find the conjugacy classes, from the second relation in  presentation \eqref{eq:f83e6c} we have
$y^{-1}=y^{2r}$, thus, from the third relation
\begin{equation}\label{eq:04c7a8}
xy^{-1}x^{-1}=xy^{2r}x^{-1}=y^{-2r}=y.
\end{equation}
From the third relation in \eqref{eq:f83e6c} and
\eqref{eq:04c7a8} we get
\begin{equation}\label{eq:10c77e}
y^{-1}x=xy,\quad yx=xy^{-1},\quad y^{-1}x^{-1}=x^{-1}y,\quad yx^{-1}=x^{-1}y^{-1}.
\end{equation}
Using \eqref{eq:10c77e} we obtain
\begin{equation*}
yx^py^{-1}=xy^{-1}x^{p-1}y^{-1}=xy^{-1}xx^{p-2}y^{-1}=x^2yx^{p-2}y^{-1}=\begin{cases}
                                                                         x^p & \text{if $p$ is even,}\\
                                                                         x^py^{-2} & \text{if $p$ is odd.}
                                                                        \end{cases}
\end{equation*}
Using \eqref{eq:10c77e} any word on the generators can be
taken to the form $x^py^q$ with $p,q\in\mathbb{Z}$. Fix an element
$x^py^q$ and conjugate it by an arbitrary element $x^ry^s$ with
$p,q,r,s\in\mathbb{Z}$,
\begin{equation}\label{eq:1437a3}
x^ry^sx^py^qy^{-s}x^{-r}
=\begin{cases}
                                                            x^rx^py^qx^{-r}=\begin{cases}
                                                                             x^py^q & \text{$r$ even,}\\
                                                                             x^py^{-q} & \text{$r$ odd,}
                                                                            \end{cases} & \text{$p$ even,}\\
                                                            x^rx^py^{q-2s}x^{-r}=\begin{cases}
                                                                                 x^py^{q-2s} & \text{$r$ even,}\\
                                                                                 x^py^{-q+2s} & \text{$r$ odd,}
                                                                                \end{cases} & \text{$p$ odd.}
                                                           \end{cases}
\end{equation}
By \eqref{eq:1437a3} we obtain the following conjugacy classes, which by Remark~\ref{rem:Notation}, they are denoted by the number of their elements:
\begin{align*}
\mathsf{1}_l&=\{x^{2l}\},  & l&=0,\dots,2^{k-1}-1,\\
\mathsf{2}_{l,q}&=\{x^{2l}y^q,x^{2l}y^{-q}\},  & l&=0,\dots,2^{k-1}-1,\  q=1,\dots,r,\\
\mathsf{(2r+1)}_l&=\{x^{2l+1},x^{2l+1}y,x^{2l+1}y^{2}\dots,x^{2l+1}y^{2r}\},  & l&=0,\dots,2^{k-1}-1.
\end{align*}
We get $2^{k-1}(r+2)$ conjugacy classes: $2^{k-1}$ conjugacy classes with one element each one,
$2^{k-1}r$ conjugacy classes with $2$ elements and $2^{k-1}$ conjugacy classes with $2r+1$ elements. They are all, since
$2^{k-1}+2\cdot2^{k-1}r+2^{k-1}(2r+1)=2^k(2r+1)$ is the order of the group.

\subsubsection{Irreducible representations}\label{irreducible-representations}

The group $D_{2^k(2r+1)}$ has $2^{k-1}(r+2)$ irreducible representations. For $k=2$ this agrees with the fact that
$D_{4(2r+1)}\cong\mathsf{BD}_{2(2r+1)}$.

It has $2^k$ one-dimensional representations, denoted by $\alpha_j$, $0\leq j\leq 2^k-1$, they correspond to the irreducible representations
of its abelianization, by Lemma~\ref{lem:Ab.D} $\mathrm{Ab}(D_{2^k(2r+1)})=\mathbf{C}_{2^k}$ and they are given by
\begin{equation*}
\alpha_j(x)=\zeta_{2^k}^{j},\quad \alpha_j(y)=1.
\end{equation*}

Consider the two-dimensional representations
\begin{equation*}
\varrho_{t,s}(x)=\zeta_{2^k}^s\left(\begin{smallmatrix}0&1\\ (-1)^t&0\end{smallmatrix}\right),\qquad \varrho_{t,s}(y)=\left(\begin{smallmatrix}\zeta_{2r+1}^{t}&0\\0&\zeta_{2r+1}^{-t}\end{smallmatrix}\right),
\end{equation*}
where $\varrho_{1,1}$ is the natural representation.

\begin{remark}\label{rem:k=2}
When $k=2$ we have $s=0$ in $\varrho_{t,s}$ and we recover the $2$-dimensional irreducible representations of $\mathsf{BD}_{2(2r+1)}$.
It is important to note that in this case the natural representation is $\varrho_{1,0}$ and not $\varrho_{1,1}$.
\end{remark}

It is straightforward to check that $\varrho_{t,s}$ satisfy the relations in \eqref{eq:f83e6c} so they are indeed representations of $D_{2^k(2r+1)}$. 

Let $H$ be a set of elements of $D_{2^k(2r+1)}$ containing exactly
one element from each conjugacy class, for instance
\begin{align*}
H&=\{x^0,x^2,\dots,x^{2l},\dots,x^{2^{k}-2},x,x^3,\dots,x^{2^k-1},y,\dots,y^r,\dots,x^{2^k-2}y,\dots,x^{2^k-2}y^r\}\\
&=\{x^{2l},x^{2l+1},x^{2l}y^q\},\  l=0,\dots,2^{k-1}-1,\  q=1,\dots,r.
\end{align*}
Evaluating $\varrho_{t,s}$ on these elements we get
\begin{align*}
 \varrho_{t,s}(x^{2l})&=\zeta_{2^k}^{2ls}\left(\begin{smallmatrix}
                                         (-1)^{tl} & 0\\ 0 & (-1)^{tl}
                                        \end{smallmatrix}\right),\\
 \varrho_{t,s}(x^{2l+1})&=\zeta_{2^k}^{(2l+1)s}\left(\begin{smallmatrix}
                                         0 & (-1)^{tl} \\ (-1)^{t(l+1)} &0
                                        \end{smallmatrix}\right),\\
 \varrho_{t,s}(x^{2l}y^q)&=(-1)^{lt}\zeta_{2^k}^{2ls}\left(\begin{smallmatrix}
                                         \zeta_{2r+1}^{tq} & 0\\ 0 & \zeta_{2r+1}^{-tq}
                                        \end{smallmatrix}\right).
\end{align*}
Let $\chi_{t,s}$ be the character of the representation
$\varrho_{t,s}$. The corresponding values of the elements of $H$
under the character are
\begin{equation}\label{eq:abea1a}
\begin{split}
\chi_{t,s}(x^{2l})&=(-1)^{tl}2\zeta_{2^k}^{2ls},\\
\chi_{t,s}(x^{2l+1})&=0,\\
\chi_{t,s}(x^{2l}y^q)&=(-1)^{tl}\zeta_{2^k}^{2ls}(\zeta_{2r+1}^{tq}+\zeta_{2r+1}^{-tq}).
\end{split}
\end{equation}

Recall that if $n$ is not a divisor of $t$, the sum of the $t$-th
powers of the $n$-th roots of unity is zero, that is,
\begin{equation}\label{eq:srou=0}
\sum_{q=0}^{n-1}\zeta_n^{tq}=0,\qquad\text{so}\qquad \sum_{q=1}^{n-1}\zeta_n^{tq}=-1
\end{equation}
On the other hand, we have that
\begin{equation}\label{srou.neq0}
\sum_{q=0}^{n-1}\zeta_n^{tq}=n,\qquad\text{if $t$ is a multiple of $n$.}
\end{equation}
Hence, if $t\neq 0$, the inner product of the character $\chi_{t,s}$ with itself is
\begin{align*}
\langle \chi_{t,s},\chi_{t,s}\rangle&=\frac{1}{2^k(2r+1)}\left[\sum_{l=0}^{2^{k-1}-1}\bigl((-1)^{tl}2\zeta_{2^k}^{2ls}\bigr)\bigl((-1)^{tl}2\zeta_{2^k}^{-2ls}\bigr)\right.\\
&\ \left.+2\sum_{l=0}^{2^{k-1}-1}\sum_{q=1}^{r}\bigl((-1)^{tl}\zeta_{2^k}^{2ls}(\zeta_{2r+1}^{tq}+\zeta_{2r+1}^{-tq})\bigr)\bigl((-1)^{tl}\zeta_{2^k}^{-2ls}(\zeta_{2r+1}^{-tq}+\zeta_{2r+1}^{tq})\bigr)\right]\\
&=\frac{1}{2^k(2r+1)}\left[4(2^{k-1})+2\sum_{l=0}^{2^{k-1}-1}\sum_{q=1}^{r}(\zeta_{2r+1}^{tq}+\zeta_{2r+1}^{-tq})(\zeta_{2r+1}^{-tq}+\zeta_{2r+1}^{tq})\right]\\
&=\frac{1}{2^k(2r+1)}\left[4(2^{k-1})+2\sum_{l=0}^{2^{k-1}-1}\sum_{q=1}^{r}(2+\zeta_{2r+1}^{2tq}+\zeta_{2r+1}^{-2tq})\right]\\
&=\frac{1}{2^k(2r+1)}\left[4(2^{k-1})+4(2^{k-1}r)+2\sum_{l=0}^{2^{k-1}-1}\Bigl(\sum_{q=1}^{r}(\zeta_{2r+1}^{q})^{2t}+\sum_{q=1}^{r}(\zeta_{2r+1}^{-q})^{2t}\Bigr)\right]\\
\end{align*}
We have that $\zeta_{2r+1}^{-q}=\zeta_{2r+1}^{2r+1-q}$, thus we have that
\begin{equation}\label{eq:-q.2r}
 \sum_{q=1}^{r}\zeta_{2r+1}^{-q}=\sum_{q=1}^{r}\zeta_{2r+1}^{2r+1-q}=\sum_{q=r+1}^{2r}\zeta_{2r+1}^{q}.
\end{equation}
Therefore
\begin{align}
 \langle \chi_{t,s},\chi_{t,s}\rangle&=\frac{1}{2^k(2r+1)}\left[[2^{k+1}+2^{k+1}r+2\sum_{l=0}^{2^{k-1}-1}\sum_{q=1}^{2r}\zeta_{2r+1}^{2tq}\right]\label{eq:l1}\\
 &=\frac{1}{2^k(2r+1)}\left[2^{k+1}+2^{k+1}r+2\sum_{l=0}^{2^{k-1}-1}(-1)\right]\notag\\
 &=\frac{1}{2^k(2r+1)}\left[2^{k+1}+2^{k+1}r-2(2^{k-1})\right]\notag\\
 &=\frac{1}{2^k(2r+1)}\left[2^{k+1}+2^{k+1}r-2^{k}\right]\notag\\
 &=\frac{1}{2^k(2r+1)}\left[2^{k}(2+2r-1)\right]\notag\\
 &=\frac{1}{2^k(2r+1)}\left[2^{k}(2r+1)\right]=1\notag
\end{align}
This proves that the representions $\varrho_{t,s}$ with $t\neq0$ are irreducible.
In the case $t=0$, from \eqref{eq:l1} it is easy to see that 
\begin{align*}
\langle \chi_{0,s},\chi_{0,s}\rangle&=\frac{1}{2^k(2r+1)}\left[[2^{k+1}+2^{k+1}r+2\sum_{l=0}^{2^{k-1}-1}\sum_{q=1}^{2r}1\right]=2\\
\end{align*}
thus the two-dimensional representation $\varrho_{0,s}$ is not irreducible and it has to be the direct sum
of two one-dimensional representations.

\begin{proposition}\label{prop:t=0}
We have the following isomorphism of representations of $D_{2^k(2r+1)}$
\begin{equation*}
\varrho_{0,s}\cong\alpha_s\oplus\alpha_{2^{k-1}+s}.
\end{equation*}
\end{proposition}

\begin{proof}
Let $\chi_s$ be the character of the one-dimensional representation $\alpha_s$. By \eqref{eq:sum.prod} the character $\chi_{\alpha_s\oplus\alpha_{2^{k-1}+s}}$ of the representation
$\alpha_s\oplus\alpha_{2^{k-1}+s}$ is equal to the sum of characters $\chi_s+\chi_{2^{k-1}+s}$. Comparing with \eqref{eq:abea1a} we have
\begin{align*}
\chi_s(x^{2l})+\chi_{2^{k-1}+s}(x^{2l})&=\zeta_{2^k}^{2ls}+\zeta_{2^k}^{2l(2^{k-1}+s)}=2\zeta_{2^k}^{2ls}=\chi_{0,s}(x^{2l}),\\
\chi_s(x^{2l+1})+\chi_{2^{k-1}+s}(x^{2l+1})&=\zeta_{2^k}^{(2l+1)s}+\zeta_{2^k}^{(2l+1)(2^{k-1}+s)}\\
&=\zeta_{2^k}^{(2l+1)s}+\zeta_{2^k}^{2^{k-1}(2l+1)}\zeta_{2^k}^{(2l+1)s}\\
&=\zeta_{2^k}^{(2l+1)s}-\zeta_{2^k}^{(2l+1)s}=0=\chi_{0,s}(x^{2l+1}),\\
\chi_s(x^{2l}y^q)+\chi_{2^{k-1}+s}(x^{2l}y^q)&=\zeta_{2^k}^{2ls}+\zeta_{2^k}^{2l(2^{k-1}+s)}=2\zeta_{2^k}^{2ls}=\chi_{0,s}(x^{2l}y^q).
\end{align*}
Since the representations $\alpha_s\oplus\alpha_{2^{k-1}+s}$ and $\varrho_{0,s}$ have the same character they are isomorphic.
\end{proof}

Analogously, since we shall need it later, we can prove the following proposition.

\begin{proposition}\label{prop:t=2r+1}
We have the following isomorphism of representations of $D_{2^k(2r+1)}$
\begin{equation*}
\varrho_{2r+1,s}\cong\alpha_{2^{k-2}+s}\oplus\alpha_{2^{k-1}+2^{k-2}+s}.
\end{equation*}
\end{proposition}

\begin{proof}
Let $\chi_s$ be the character of the one-dimensional representation $\alpha_s$. By \eqref{eq:sum.prod} the character $\chi_{\alpha_{2^{k-2}+s}\oplus\alpha_{2^{k-1}+2^{k-2}+s}}$ of the representation
$\alpha_{2^{k-2}+s}\oplus\alpha_{2^{k-1}+2^{k-2}+s}$ is equal to the sum of characters $\chi_{2^{k-2}+s}+\chi_{2^{k-1}+2^{k-2}+s}$. Comparing with \eqref{eq:abea1a} we have
\begin{align*}
\chi_{2^{k-2}+s}(x^{2l})+\chi_{2^{k-1}+2^{k-2}+s}(x^{2l})&=\zeta_{2^k}^{2l(2^{k-2}+s)}+\zeta_{2^k}^{2l(2^{k-1}+2^{k-2}+s)}\\
&=(-1)^l2\zeta_{2^k}^{2l(2^{k-2}+s)}=\chi_{2r+1,s}(x^{2l}),\\
\chi_{2^{k-2}+s}(x^{2l+1})+\chi_{2^{k-1}+2^{k-2}+s}(x^{2l+1})&=\zeta_{2^k}^{(2l+1)(2^{k-2}+s)}+\zeta_{2^k}^{(2l+1)(2^{k-1}+2^{k-2}+s)}\\
&=\zeta_{2^k}^{(2l+1)(2^{k-2}+s)}+\zeta_{2^k}^{2^{k-1}(2l+1)}\zeta_{2^k}^{(2l+1)(2^{k-2}+s)}\\
&=\zeta_{2^k}^{(2l+1)(2^{k-2}+s)}-\zeta_{2^k}^{(2l+1)(2^{k-2}+s)}\\
&=0=\chi_{2r+1,s}(x^{2l+1}),\\
\chi_{2^{k-2}+s}(x^{2l}y^q)+\chi_{2^{k-1}+2^{k-2}+s}(x^{2l}y^q)&=\zeta_{2^k}^{2l(2^{k-2}+s)}+\zeta_{2^k}^{2l(2^{k-1}+2^{k-2}+s)}\\
&=(-1)^l2\zeta_{2^k}^{2l(2^{k-2}+s)}=\chi_{2r+1,s}(x^{2l}y^q).
\end{align*}
Since the representations $\alpha_{2^{k-2}+s}\oplus\alpha_{2^{k-1}+2^{k-2}+s}$ and $\varrho_{2r+1,s}$ have the same character they are isomorphic.
\end{proof}

Naturally we have that $t=1,\dots,2r$ and $s=0,\dots,2^{k}-1$, but taking all this values of $t$ and $s$ gives $2^{k+1}r$ irreducible representions, which are
much more than conjugacy classes, hence, some of them should be isomorphic.

\begin{proposition}[{\cite[(3.5)]{Arciniega-etal:CRMQSAPS}}]\label{prop:2-dim.iso}
We have the following isomorphism of representations of $D_{2^k(2r+1)}$
\begin{equation}\label{eq:8c274}
\varrho_{t,s}\cong\varrho_{t,2^{k-1}+s},\quad\text{and}\quad \varrho_{2r+1-t,s}\cong\varrho_{t,2^{k-2}+s}.
\end{equation}
\end{proposition}

\begin{proof}
From \eqref{eq:abea1a} we have
\begin{equation*}
\begin{split}
\chi_{t,2^{k-1}+s}(x^{2l})&=(-1)^{tl}2\zeta_{2^k}^{2l(2^{k-1}+s)}=\chi_{t,s}(x^{2l}),\\
\chi_{t,2^{k-1}+s}(x^{2l+1})&=0=\chi_{t,s}(x^{2l+1}),\\
\chi_{t,2^{k-1}+s}(x^{2l}y^q)&=(-1)^{tl}\zeta_{2^k}^{2l(2^{k-1}+s)}(\zeta_{2r+1}^{tq}+\zeta_{2r+1}^{-tq})=\chi_{t,s}(x^{2l}y^q).
\end{split}
\end{equation*}
Since the representations $\varrho_{t,s}$ and $\varrho_{t,2^{k-1}+s}$ have the same character they are isomorphic.
\begin{equation*}
\begin{split}
\chi_{2r+1-t,s}(x^{2l})&=(-1)^{(2r+1-t)l}2\zeta_{2^k}^{2ls},\\
&=(-1)^{l}(-1)^{tl}2\zeta_{2^k}^{2ls},\\
&=\zeta_{2^k}^{2^{k-1}l}(-1)^{tl}2\zeta_{2^k}^{2ls},\\
&=(-1)^{tl}2\zeta_{2^k}^{2l(2^{k-2}+s)}=\chi_{t,2^{k-2}+s}(x^{2l}),\\
\chi_{2r+1-t,s}(x^{2l+1})&=0=\chi_{t,2^{k-2}+s}(x^{2l+1}),\\
\chi_{2r+1-t,s}(x^{2l}y^q)&=(-1)^{(2r+1-t)l}\zeta_{2^k}^{2ls}(\zeta_{2r+1}^{(2r+1-t)q}+\zeta_{2r+1}^{-(2r+1-t)q}),\\
&=(-1)^{tl}\zeta_{2^k}^{2l(2^{k-2}+s)}(\zeta_{2r+1}^{tq}+\zeta_{2r+1}^{-tq})=\chi_{t,2^{k-2}+s}(x^{2l}y^q).
\end{split}
\end{equation*}
Since the representations $\varrho_{2r+1-t,s}$ and $\varrho_{t,2^{k-2}+s}$ have the same character they are isomorphic.
\end{proof}

From Proposition~\ref{prop:2-dim.iso} we can take $t=1,\dots,2r$ and $s=0,\dots,2^{k-2}-1$
(or equivalently $t=1,\dots,r$ and $s=0,\dots,2^{k-1}-1$), so they are $2^{k-1}r$ of them, which together with the $2^k$ one-dimensional irreducible representations
gives a total of $2^{k-1}(r+2)$ irreducible representations. Thus, this is the complete list of
irreducible representations of $D_{2^m(2k+1)}$. Another way to see
this is the following. We have listed $2^k$ one-dimensional
representations and $2^{k-1}r$ two-dimensional representations. The
sum of the squares of their ranks is
\begin{equation*}
    2^k + 2^2 2^{k-1}r = 2^k + 2^{k+1}r=2^k(2r+1),
\end{equation*}
which is the order of $D_{2^k(2r+1)}$. By \eqref{eq:irr.dim.ord} there can not exist any other irreducible representation.
The character table of $D_{2^k(2r+1)}$ is given in Table~\ref{tab:CT.D}.

\begin{table}[H]
\begin{equation*}
\setlength{\extrarowheight}{5pt}
\begin{array}{|c|c|c|c|}\hline
\text{\small Class}& \mathsf{1}_l & \mathsf{2}_{l,q} & \mathsf{(2r+1)}_l\\\hline
\chi_j & \zeta_{2^k}^{2lj} &\zeta_{2^k}^{2lj} & \zeta_{2^k}^{(2l+1)j}\\\hline
\chi_{t,s} & (-1)^{tl}2\zeta_{2^k}^{2ls} & (-1)^{tl}\zeta_{2^k}^{2ls}(\zeta_{2r+1}^{tq}+\zeta_{2r+1}^{-tq}) & 0\\\hline
\multicolumn{4}{l}{\text{Note: $\zeta_{n}=e^{\frac{2\pi i}{n}}$, $0\leq j\leq 2^k-1$, $1\leq t\leq 2r$, $0\leq s\leq 2^{k-2}-1$,}}\\
\multicolumn{4}{l}{\text{\hspace{1.35cm} $0\leq l\leq 2^{k-1}-1$, $1\leq q\leq r$.}}
\end{array}
\end{equation*}
\caption{Character table of $D_{2^k(2r+1)}$.}\label{tab:CT.D}
\end{table}

\subsection{Groups $P'_{8\cdot 3^k}$}\label{definition}

This family of groups belongs to the list of finite subgroup of $SO(4)$ acting freely on $\mathbb{S}^3$ see \cite[\S1.7]{Aschenbrenner-etal:3MG}.
The group $P'_{8\cdot 3^k}$ corresponds to the small tetrahedral group $\mathbb{T}_{3^{k-1}}$ in Theorem~\ref{thm:sfs}.

\subsubsection{Presentation}

The group $P'_{8\cdot 3^k}$ has order $8\cdot 3^k$ and a presentation is
\begin{equation}\label{eq:271bf2}
P'_{8\cdot 3^k}:=\langle x,y,z\mid x^2=(xy)^2=y^2,zxz^{-1}=y, zyz^{-1}=xy, z^{3^k}=1\rangle,
\end{equation}
where $k\geq2$.

\begin{remark}
When $k=1$ there is an isomorphism between $P'_{24}$ and $\mathrm{B}\mathbf{T}$, the binary tetrahedral group \cite[Lemma~2.1 \& Remark~2.3]{Arciniega-etal:CRMQSAPS}.
\end{remark}

\begin{lemma}[{\cite[Lemma~3.1]{Arciniega-etal:CRMQSAPS}}]\label{lem:Ab.P}
The abelianization of $P'_{8\cdot 3^k}$ is $\mathrm{Ab}(P'_{8\cdot 3^k})=\mathbf{C}_{3^k}$, where  $\mathrm{Ab}(z)$ is the generator of $\mathbf{C}_{3^k}$ and $\mathrm{Ab}(x)=\mathrm{Ab}(y)=1$.
\end{lemma}

\subsubsection{Consequences of the relations}\label{consequences-of-the-relations}

The following consequences from the relations in presentation \eqref{eq:271bf2} will be useful.

From the first relation $x^2=xyxy=y^2$ we get
\begin{equation}\label{eq:907aee}
x=yxy,\quad y=xyx,
\end{equation}
which imply
\begin{equation}\label{eq:d71836}
yx=xy^{-1},\quad y^{-1}x=xy,\quad yx^{-1}=xy,\quad y^{-1}x^{-1}=x^{-1}y=xy^{-1}.
\end{equation}
From \eqref{eq:907aee} we also get
\begin{equation}\label{eq:b4f557}
xyx^{-1}=y^{-1},\quad yxy^{-1}=x^{-1}.
\end{equation}
From the second relation $zxz^{-1}=y$ in \eqref{eq:271bf2} we get
\begin{equation}\label{eq:ed1372}
zx=yz,\quad zx^{-1}=y^{-1}z,\quad z^{-1}y=xz^{-1},\quad z^{-1}y^{-1}=x^{-1}z^{-1}.
\end{equation}
Also from the third relation in \eqref{eq:ed1372} and the third relation in \eqref{eq:d71836} we get
\begin{equation*}
xz^{-1}x^{-1}=z^{-1}yx^{-1}=z^{-1}xy,
\end{equation*}
thus $xz^{-1}x^{-1}z=z^{-1}xyz=y$ and we get
\begin{equation}\label{eq:5d1b38}
xzx^{-1}=zy^{-1}.
\end{equation}
From the third relation $zyz^{-1}=xy$ in \eqref{eq:271bf2} we get
\begin{equation}\label{eq:35c845}
\begin{aligned}
z^{-1}x&=yz^{-1}y^{-1}=yx^{-1}z^{-1}, & z^{-1}x^{-1}&=y^{-1}z^{-1}y^{-1}=x^{-1}yz^{-1},\\
zy&=xyz, & zy^{-1}&=y^{-1}x^{-1}z=xy^{-1}z.
\end{aligned}
\end{equation}
Taking the inverse of the first relation in \eqref{eq:35c845} we
have
\begin{equation}\label{eq:4a1c10}
yzy^{-1}=x^{-1}z.
\end{equation}
Using the consequences \eqref{eq:d71836},
\eqref{eq:ed1372} and \eqref{eq:35c845}, any word on the
generators can be taken to the form $x^py^qz^r$, with $p=0,1,2,3$,
$q=0,1$ and $r=0,\dots 3^{k}-1$.

\subsubsection{Conjugacy classes}\label{conjugacy-classes.P}

We are going to compute the conjugacy classes using the second and third
relations in \eqref{eq:271bf2} and \eqref{eq:b4f557},
\eqref{eq:5d1b38}, \eqref{eq:4a1c10}.

\begin{lemma}
We have
\begin{equation}\label{eq:81d8c1}
z^3x^pz^{-3}=x^p\quad\text{and}\quad z^3y^pz^{-3}=y^p.
\end{equation} In fact, conjugating $x^p$ several times by $z$ we get the
following cycle:
\begin{equation*}
x^p\mapsto y^p\mapsto (xy)^p\mapsto x^p.
\end{equation*}
\end{lemma}

\begin{proof}
Using the second and third
relations in \eqref{eq:271bf2} and the first one in
\eqref{eq:907aee} we get
\begin{equation*}
z^3x^pz^{-3}=z^2y^pz^{-2}=z(xy)^pz^{-1}=(yxy)^p=x^p.\qedhere
\end{equation*}
\end{proof}

\begin{lemma}
We have
\begin{equation*}
x^2z^rx^{-2}=z^r.
\end{equation*}
In fact, conjugating $z^r$ several times by $x$ we get the
following cycle:
\begin{equation}\label{eq:4eeb11}
z^r\mapsto (zy^{-1})^r\mapsto z^r.
\end{equation}
\end{lemma}

\begin{proof}
Using \eqref{eq:5d1b38}
and the inverse of the first relation in \eqref{eq:b4f557} we
have
\begin{equation*}
x^2z^rx^{-2}=x(zy^{-1})^rx^{-1}=(zy^{-1}y)^r=z^r.\qedhere
\end{equation*}
\end{proof}

\begin{lemma}\label{lem:394310}
We have
\begin{equation}\label{eq:3f5e6e}
xz^rx^{-1}=\begin{cases}
z^r, & r\equiv 0\mod 3,\\
xy^{-1}z^r, & r\equiv 1\mod 3,\\
yz^r, & r\equiv 2\mod 3.
\end{cases}
\end{equation}
\end{lemma}

\begin{proof}
Using the second equation in \eqref{eq:ed1372} we get
\begin{equation}\label{eq:394310}
xzx^{-1}=xy^{-1}z.
\end{equation}

Using the second equation in \eqref{eq:ed1372}, the fourth
equation in \eqref{eq:35c845} and the fourth one in
\eqref{eq:d71836} we obtain
\begin{equation}\label{eq:21208d}
xz^2x^{-1}=xzy^{-1}z=xy^{-1}x^{-1}z^2=xx^{-1}yz^2=yz^2.
\end{equation}

Notice that from \eqref{eq:21208d} we get that
$z^2x^{-1}=x^{-1}yz^2$, by second equation in
\eqref{eq:ed1372}, third one in \eqref{eq:35c845} and
\eqref{eq:b4f557}
\begin{equation}\label{eq:d177a2}
xz^3x^{-1}=xzx^{-1}yz^2=xy^{-1}zyz^2=xy^{-1}xyz^3=xx^{-1}z^3=z^3.
\end{equation}

Now, let $r=3k+l$, with $l=0,1,2$, then by \eqref{eq:394310},
\eqref{eq:21208d} and \eqref{eq:d177a2} we have
\begin{equation*}
xz^rx^{-1}=xz^lx^{-1}xz^{3k}x^{-1}=xz^lx^{-1}z^{3k}=\begin{cases}
z^r, & l=0,\\
xy^{-1}z^r, & l=1,\\
yz^r, & l=2.
\end{cases}
\end{equation*}
This proof the lemma.
\end{proof}

\begin{lemma}
We have
\begin{equation}\label{eq:1f7a51}
y^2z^ry^{-2}=z^r.
\end{equation} In fact, conjugating $z^r$ several times by $y$ we get the
following cycle:
\begin{equation*}
z^r\mapsto (x^{-1}z)^r\mapsto z^r.
\end{equation*}
\end{lemma}

\begin{proof}
Using \eqref{eq:4a1c10}
and the inverse of the second relation in \eqref{eq:b4f557} we
have
\begin{equation*}
y^2z^ry^{-2}=y(x^{-1}z)^ry^{-1}=(xx^{-1}z)^r=z^r.\qedhere
\end{equation*}
\end{proof}

\begin{lemma}
We have
\begin{equation}\label{eq:5a5486}
yz^ry^{-1}=\begin{cases}
z^r & r\equiv0\mod 3,\\
x^{-1}z^r & r\equiv1\mod 3,\\
xyz^r & r\equiv2\mod 3.
\end{cases}
\end{equation}
\end{lemma}

\begin{proof}
Using the fourth equation in
\eqref{eq:35c845} and \eqref{eq:b4f557} we have
\begin{equation}\label{eq:d001e7}
yzy^{-1}=yxy^{-1}z=x^{-1}z.
\end{equation}

Using the fourth equation in \eqref{eq:35c845}, the first one in
\eqref{eq:ed1372}, \eqref{eq:d001e7} and the third
equation in \eqref{eq:d71836} we get
\begin{equation}\label{eq:f7d88f}
yz^2y^{-1}=yzxy^{-1}z=yyzy^{-1}z=yx^{-1}z^2=xyz^2.
\end{equation}

From \eqref{eq:f7d88f} and the second equation in
\eqref{eq:ed1372} we have
\begin{equation}\label{eq:e318ab}
yz^3y^{-1}=yzz^2y^{-1}=yzx^{-1}z^2=yy^{-1}z^3=z^3.
\end{equation}

Now, let $r=3k+l$, with $l=0,1,2$, then by \eqref{eq:d001e7},
\eqref{eq:f7d88f} and \eqref{eq:e318ab} we have
\begin{equation*}
yz^ry^{-1}=yz^ly^{-1}yz^{3k}y^{-1}=yz^ry^{-1}=yz^ly^{-1}z^3k=\begin{cases}
z^r & l=0,\\
x^{-1}z^r & l=1,\\
xyz^r & l=2.
\end{cases}\qedhere
\end{equation*}
\end{proof}

Let $x^py^qz^r$ be an
arbitrary element of $P'_{8\cdot 3^k}$. From \eqref{eq:b4f557},
\eqref{eq:4eeb11}, \eqref{eq:1f7a51} and
\eqref{eq:81d8c1} we have that conjugating by $x^2$, $y^2$ or
$z^3$ fix the element, thus, to get its conjugacy class, it is enough
to conjugate by $x$, $y$, $z$ and $z^2$ in all the possible
combinations.

From \eqref{eq:b4f557}, \eqref{eq:3f5e6e} and
\eqref{eq:5a5486} we have that conjugating an element of the
group $P'_{8\cdot 3^k}$ of the form $x^py^qz^r$ by either $x$,
$y$ or $z$, the exponent $r$ of $z^r$ does not change, and that
conjugation of $z^r$ by $x$ or $y$ depends of the residue class of
$r$ modulo $3$. We see the group $P'_{8\cdot 3^k}$ as the disjoint
union of the following $24$ sets, each one with $3^{k-1}$ elements
taking $l=0,\dots,3^{k-1}-1$.

\begin{equation*}
\setlength{\extrarowheight}{5pt}
\begin{array}{|c|c|c|}\hline
xz^{3l} & xz^{3l+1} & xz^{3l+2} \\\hline
x^2z^{3l} & x^2z^{3l+1} & x^2z^{3l+2} \\\hline
x^3z^{3l} & x^3z^{3l+1} & x^3z^{3l+2} \\\hline
yz^{3l} & yz^{3l+1} & yz^{3l+2} \\\hline
xyz^{3l} & xyz^{3l+1} & xyz^{3l+2} \\\hline
x^2yz^{3l} & x^2yz^{3l+1} & x^2yz^{3l+2} \\\hline
x^3yz^{3l} & x^3yz^{3l+1} & x^3yz^{3l+2} \\\hline
\end{array}
\end{equation*}

As we mentioned above, to get the conjugacy classes, we need to
conjugate each element of the table by $x$, $y$, $z$ and $z^2$
in all the possible combinations, and we will get elements in the same
column with the same value of $l$. We get the following
$7\cdot 3^{k-1}$ conjugacy classes:
\begin{align*}
\mathsf{1}_l&=\{z^{3l}\},  & l=0,\dots,3^{k-1}-1,\\
\mathsf{1}_{l}^{\mathsf{+}}&=\{x^2z^{3l}\},  & l=0,\dots,3^{k-1}-1,\\
\mathsf{4}_l^{\mathsf{a}}&=\{z^{3l+1}, x^3z^{3l+1},x^2yz^{3l+1},x^3yz^{3l+1}\},  & l=0,\dots,3^{k-1}-1,\\
\mathsf{4}_l^{\mathsf{b}}&=\{z^{3l+2},xz^{3l+2},yz^{3l+2},xyz^{3l+2}\}  & l=0,\dots,3^{k-1}-1,\\
\mathsf{4}_l^{\mathsf{c}}&=\{xz^{3l+1},x^2z^{3l+1},yz^{3l+1},xyz^{3l+1}\},  & l=0,\dots,3^{k-1}-1,\\
\mathsf{4}_l^{\mathsf{d}}&=\{x^2z^{3l+2},x^3z^{3l+2},x^2yz^{3l+2},x^3yz^{3l+2}\},  & l=0,\dots,3^{k-1}-1,\\
\mathsf{6}_l&=\{xz^{3l},yz^{3l},x^3z^{3l},xyz^{3l},x^2yz^{3l},x^3yz^{3l}\},  & l=0,\dots,3^{k-1}-1.
\end{align*}
They are $2\cdot 3^{k-1}$ classes with one element, $4\cdot 3^{k-1}$ classes with $4$ elements and $3^{k-1}$ classes with $6$ elements.
They are all since $2\cdot 3^{k-1}+4\cdot 4\cdot 3^{k-1}+6\cdot 3^{k-1}=24\cdot 3^{k-1}=8\cdot 3^k$ is the order of the group.

\subsubsection{Irreducible representations}\label{irreducible-representations.P}

The group $P'_{8\cdot 3^k}$ has $7\cdot 3^{k-1}$ irreducible representations

\begin{remark}
For $k=1$ this agrees with the fact that $P'_{24}\cong\mathsf{BT}$.
\end{remark}

It has $3^k$ one-dimensional representations, denoted by $\alpha_j$, $0\leq j\leq 3^k-1$, they correspond to the irreducible representations
of its abelianization, by Lemma~\ref{lem:Ab.P}
$\mathrm{Ab}(P'_{8\cdot 3^k})=\mathbf{C}_{3^k}$ and they are given by
\begin{equation*}
\alpha_j(z)=\zeta_{3^k}^{j},\quad \alpha_j(x)=\alpha_j(y)=1.
\end{equation*}
It also has $3^k$ two-dimensional representations, denoted by $\varrho_{j}$ with $j=0,\dots,3^{k}-1$, given by
\begin{equation}\label{eq:2drPp83k}
\begin{aligned}
 \varrho_{j}(x)&=\left(\begin{smallmatrix}
                       0 & \zeta_3^2\\
                       -\zeta_3 & 0
                      \end{smallmatrix}\right), & \varrho_{j}(y)&=\left(\begin{smallmatrix}
                                                                \zeta_3^2 & 1\\
                                                                \zeta_3^2 & -\zeta_3^2
                                                               \end{smallmatrix}\right), & \varrho_{j}(z)&=\zeta_{3^k}^j\left(\begin{smallmatrix}
                                                                                                          0 & \zeta_3\\
                                                                                                          -\zeta_3^2 & -1
                                                                                                         \end{smallmatrix}\right),\\
\end{aligned}
\end{equation}
where $\varrho_{1}$ is the natural representation.
It has $3^{k-1}$ three-dimensional irreducible representations $\varsigma_s$, $s=0,\dots,3^{k-1}-1$ given by
\begin{equation}\label{eq:3drPp83k}
\begin{aligned}
\varsigma_s(x)&=\left(\begin{smallmatrix}
                     -1 & -1 & -1\\
                     0 & 0 & 1\\
                     0 & 1 & 0
                    \end{smallmatrix}\right), & \varsigma_s(y)&=\left(\begin{smallmatrix}
                                                                     0 & 0 & 1\\
                                                                     -1 & -1 & -1\\
                                                                     1 & 0 & 0
                                                                    \end{smallmatrix}\right), & \varsigma_s(z)&=\zeta_{3^k}^s\left(\begin{smallmatrix}
                                                                                                                     -1 & -1 & -1\\
                                                                                                                     0 & 1 & 0\\
                                                                                                                     1 & 0 & 0
                                                                                                                    \end{smallmatrix}\right).
\end{aligned}
\end{equation}

It is straightforward to check that $\varrho_{j}$ and $\varsigma_s$ satisfy the relations in \eqref{eq:271bf2} so they are indeed representations of $P'_{8\cdot 3^k}$.

Let $H$ be a set of elements of $P'_{8\cdot 3^k}$ containing exactly
one element from each conjugacy class, for instance
\begin{equation*}
H=\{z^{3l},x^2z^{3l},z^{3l+1},z^{3l+2},xz^{3l+1},x^2z^{3l+2},xz^{3l}\}, \ l=0,\dots,3^{k-1}-1.
\end{equation*}

\subsubsection*{The representation $\varrho_{j}$ is irreducible} Evaluating $\varrho_{j}$ on the elements of $H$ we get
\begin{align*}
\varrho_{j}(z^{3l})&=\left(\begin{smallmatrix}
\zeta_{3^{k}}^{3lj} & 0\\
0 & \zeta_{3^{k}}^{3lj}
\end{smallmatrix}\right), &
\varrho_{j}(x^2z^{3l})&=\left(\begin{smallmatrix}
-\zeta_{3^{k}}^{3lj} & 0\\
0 & -\zeta_{3^{k}}^{3lj}
\end{smallmatrix}\right),\\
\varrho_{j}(z^{3l+1})&=\left(\begin{smallmatrix}
0 & \zeta_3\zeta_{3^k}^{(3l+1)j}\\
-\zeta_3^2\zeta_{3^k}^{(3l+1)j} & -\zeta_{3^k}^{(3l+1)j}
\end{smallmatrix}\right), &
\varrho_{j}(z^{3l+2})&=\left(\begin{smallmatrix}
-\zeta_{3^k}^{(3l+2)j} & -\zeta_3\zeta_{3^k}^{(3l+2)j}\\
\zeta_3^2\zeta_{3^k}^{(3l+2)j} & 0
\end{smallmatrix}\right),\\
\varrho_{j}(xz^{3l+1})&=\left(\begin{smallmatrix}
-\zeta_3\zeta_{3^k}^{(3l+1)j} & -\zeta_3^2\zeta_{3^k}^{(3l+1)j}\\
0 & -\zeta_3^2\zeta_{3^k}^{(3l+1)j}
\end{smallmatrix}\right),&
\varrho_{j}(x^2z^{3l+2})&=\left(\begin{smallmatrix}
\zeta_{3^k}^{(3l+2)j} & \zeta_3\zeta_{3^k}^{(3l+2)j}\\
-\zeta_3^2\zeta_{3^k}^{(3l+2)j} & 0
\end{smallmatrix}\right),\\
\varrho_{j}(xz^{3l})&=\left(\begin{smallmatrix}
0 & \zeta_3^2\zeta_{3^{k}}^{3lj}\\
-\zeta_3\zeta_{3^{k}}^{3lj} & 0
\end{smallmatrix}\right). &\\
\end{align*}

Let $\psi_{j}$ be the character of the representation $\varrho_{j}$. The corresponding values of the elements of $H$
under the character are
\begin{align*}
\psi_{j}(z^{3l})&=2\zeta_{3^{k}}^{3lj}, &
\psi_{j}(x^2z^{3l})&=-2\zeta_{3^{k}}^{3lj}, &
\psi_{j}(z^{3l+1})&=-\zeta_{3^k}^{(3l+1)j},\\
\psi_{j}(z^{3l+2})&=-\zeta_{3^k}^{(3l+2)j}, &
\psi_{j}(xz^{3l+1})&=\zeta_{3^k}^{(3l+1)j}, &
\psi_{j}(x^2z^{3l+2})&=\zeta_{3^k}^{(3l+2)j}\\
\psi_{j}(xz^{3l})&=0.
\end{align*}
Note that
$\psi_{j}(xz^{3l+1})=\zeta_{3^k}^{(3l+1)j}(-\zeta_3-\zeta_3^2)$, but
since $1+\zeta_3+\zeta_3^2=0$ we have $-\zeta_3-\zeta_3^2=1$. This is used in other cases.

Hence the inner product of the character $\psi_{j}$ with itself is
\begin{align*}
\langle \psi_{j},\psi_{j}\rangle&=\frac{1}{8\cdot 3^k}\left[\sum_{l=0}^{3^{k-1}-1}2\zeta_{3^{k-1}}^{lj}2\zeta_{3^{k-1}}^{-lj}
+\sum_{l=0}^{3^{k-1}-1}(-2\zeta_{3^{k-1}}^{lj})(-2\zeta_{3^{k-1}}^{-lj})\right.\\
&\ +4\sum_{l=0}^{3^{k-1}-1}(-\zeta_{3^k}^{(3l+1)j})(-\zeta_{3^k}^{-(3l+1)j})
+4\sum_{l=0}^{3^{k-1}-1}(-\zeta_{3^k}^{(3l+2)j})(-\zeta_{3^k}^{-(3l+2)j})\\
&\ \left.+4\sum_{l=0}^{3^{k-1}-1}(\zeta_{3^k}^{(3l+1)j})(\zeta_{3^k}^{-(3l+1)j})
+4\sum_{l=0}^{3^{k-1}-1}(\zeta_{3^k}^{(3l+2)j})(\zeta_{3^k}^{-(3l+2)j})
+6\sum_{l=0}^{3^{k-1}-1}0\right]\\
&=\frac{1}{8\cdot 3^k}\left[6\cdot4\cdot 3^{k-1}\right]=\frac{1}{8\cdot 3^k}\left[8\cdot 3^k\right]=1.
\end{align*}
This proves that the representations $\varrho_{j}$ are
irreducible.

\subsubsection*{The representation $\varsigma_s$ is irreducible}
Evaluating $\varsigma_{s}$ on the elements of $H$ we get
\begin{align*}
\varsigma_s(z^{3l})&=\zeta_{3^{k-1}}^{ls}\left(\begin{smallmatrix}
1 & 0 & 0\\
0 & 1 & 0\\
0 & 0 & 1
\end{smallmatrix}\right), &
\varsigma_s(x^2z^{3l})&=\zeta_{3^{k-1}}^{ls}\left(\begin{smallmatrix}
1 & 0 & 0\\
0 & 1 & 0\\
0 & 0 & 1
\end{smallmatrix}\right)\\
\varsigma_s(z^{3l+1})&=\zeta_{3^k}^{(3l+1)s}\left(\begin{smallmatrix}
-1 & -1 & -1\\
0 & 1 & 0\\
1 & 0 & 0
\end{smallmatrix}\right), &
\varsigma_s(z^{3l+2})&=\zeta_{3^k}^{(3l+2)s}\left(\begin{smallmatrix}
0 & 0 & 1\\
0 & 1 & 0\\
-1 & -1 & -1
\end{smallmatrix}\right)\\
\varsigma_s(xz^{3l+1})&=\zeta_{3^k}^{(3l+1)s}\left(\begin{smallmatrix}
0 & 0 & 1\\
1  & 0 & 0\\
0 & 1 & 0
\end{smallmatrix}\right), &
\varsigma_s(x^2z^{3l+2})&=\zeta_{3^k}^{(3l+2)s}\left(\begin{smallmatrix}
0 & 0 & 1\\
0 & 1 & 0\\
-1 & -1 & -1
\end{smallmatrix}\right)\\
\varsigma_s(xz^{3l})&=\zeta_{3^{k-1}}^{ls}\left(\begin{smallmatrix}
-1 & -1 & -1\\
0 & 0 & 1\\
0 & 1 & 0
\end{smallmatrix}\right)\\
\end{align*}
Let $\varphi_s$ be the character of the representation $\varsigma_{s}$. The corresponding values of the elements of $H$ under the character are
\begin{equation}\label{eq:ch.val.vss}
\begin{aligned}
\varphi_s(z^{3l})&=3\zeta_{3^{k-1}}^{ls}, &
\varphi_s(x^2z^{3l})&=3\zeta_{3^{k-1}}^{ls}, &
\varphi_s(z^{3l+1})&=0,\\
\varphi_s(z^{3l+2})&=0, &
\varphi_s(xz^{3l+1})&=0, &
\varphi_s(x^2z^{3l+2})&=0,\\
\varphi_s(xz^{3l})&=-\zeta_{3^{k-1}}^{ls}.\\
\end{aligned}
\end{equation}
Hence the inner product of the character $\varphi_{s}$ with itself is
\begin{align*}
\langle \varphi_{s},\varphi_{s}\rangle&=\frac{1}{8\cdot 3^k}\left[\sum_{l=0}^{3^{k-1}-1}(3\zeta_{3^{k-1}}^{ls})(3\zeta_{3^{k-1}}^{-ls})
+\sum_{l=0}^{3^{k-1}-1}(3\zeta_{3^{k-1}}^{ls})(3\zeta_{3^{k-1}}^{-ls})\right.\\
&\ \left.+6\sum_{l=0}^{3^{k-1}-1}(-\zeta_{3^{k-1}}^{ls})(-\zeta_{3^{k-1}}^{-ls})\right]\\
&=\frac{1}{8\cdot 3^k}\left[9\cdot 3^{k-1}+9\cdot 3^{k-1}+6\cdot 3^{k-1}\right]=\frac{1}{8\cdot 3^k}\left[8\cdot 3^k\right]=1.
\end{align*}
This proves that the representations $\varsigma_{s}$ are irreducible.

Naturally we have that $s=0,\dots,3^{k}-1$, but taking all this values of $s$ gives $3^{k}$ three-dimensional irreducible representions, which
together with the $2\cdot3^{k}$ irreducible representations of dimensions one and two, are
much more than conjugacy classes, hence, some of them should be isomorphic.

\begin{proposition}\label{prop:3-dim.iso}
We have the following isomorphism of representations of $P'_{8\cdot 3^k}$
\begin{equation}\label{eq:3-dim.iso}
\varsigma_{3^{k-1}+s}\cong\varsigma_{s}.
\end{equation}
\end{proposition}

\begin{proof}
From \eqref{eq:ch.val.vss} we have
\begin{align*}
\varphi_{3^{k-1}+s}(z^{3l})&=3\zeta_{3^{k-1}}^{ls}, &
\varphi_{3^{k-1}+s}(x^2z^{3l})&=3\zeta_{3^{k-1}}^{ls}, &
\varphi_{3^{k-1}+s}(z^{3l+1})&=0,\\
\varphi_{3^{k-1}+s}(z^{3l+2})&=0, &
\varphi_{3^{k-1}+s}(xz^{3l+1})&=0, &
\varphi_{3^{k-1}+s}(x^2z^{3l+2})&=0,\\
\varphi_{3^{k-1}+s}(xz^{3l})&=-\zeta_{3^{k-1}}^{ls},\\
\end{align*}
which is the character $\varphi_s$ of the irreducible representation $\varsigma_s$.
\end{proof}

This is the complete list of irreducible representations of
$P'_{8\cdot 3^k}$ since the sum of the squares of their ranks is
\begin{equation*}
 3^k+4(3^k)+9(3^{k-1})=3^k+4(3^k)+3(3^k)=8(3^k),
\end{equation*}
which is the order of $P'_{8\cdot 3^k}$. By \eqref{eq:irr.dim.ord} there can not exist any other irreducible representation.
The character table of $P'_{8\cdot 3^k}$ is given in Table~\ref{tab:CT.P}.

\begin{table}[H]
\begin{equation*}
\setlength{\extrarowheight}{5pt}
\begin{array}{|c|c|c|c|c|c|c|c|}\hline
\text{\small Class}& \mathsf{1}_l & \mathsf{1}_{l}^{\mathsf{+}} & \mathsf{4}_l^{\mathsf{a}} & \mathsf{4}_l^{\mathsf{b}} & \mathsf{4}_l^{\mathsf{c}} & \mathsf{4}_l^{\mathsf{d}} & \mathsf{6}_l\\\hline
\chi_j & \zeta_{3^k}^{3lj} &\zeta_{3^k}^{3lj} & \zeta_{3^k}^{(3l+1)j} & \zeta_{3^k}^{(3l+2)j} &\zeta_{3^k}^{(3l+1)j} & \zeta_{3^k}^{(3l+2)j} & \zeta_{3^k}^{3lj}\\\hline
\psi_{j} & 2\zeta_{3^k}^{3lj} & -2\zeta_{3^k}^{3lj} & -\zeta_{3^k}^{(3l+1)j} & -\zeta_{3^k}^{(3l+2)j} & \zeta_{3^k}^{(3l+1)j} & \zeta_{3^k}^{(3l+2)j} & 0\\\hline
\varphi_{s} & 3\zeta_{3^k}^{3ls} & 3\zeta_{3^k}^{3ls} & 0 &  0 &  0 &  0 & -\zeta_{3^k}^{3ls}\\\hline
\multicolumn{8}{l}{\text{Note: $\zeta_{3^k}=e^{\frac{2\pi i}{3^k}}$, $0\leq j\leq 3^{k}-1$, $0\leq s\leq 3^{k-1}-1$, $0\leq l\leq 3^{k-1}-1$.}}
\end{array}
\end{equation*}
\caption{Character table of $P'_{8\cdot 3^k}$.}\label{tab:CT.P}
\end{table}

\section{McKay quivers of small finite subgroups of $GL(2,\mathbb{C})$}\label{sec:MQ.all}

By Remark~\ref{rem:GxC} the small finite subgroups of $GL(2,\mathbb{C})$ are the groups $\mathrm{B}\mathbf{D}_{2q}$, $\mathrm{B}\mathbf{T}$, $\mathrm{B}\mathbf{O}$, $\mathrm{B}\mathbf{I}$, $D_{2^{k+1}(2r+1)}$ and $P'_{8\cdot 3^k}$; the direct products $\Gamma\times\mathbf{C}_m$, where $\Gamma$ is one of the groups $\mathrm{B}\mathbf{D}_{2q}$, $\mathrm{B}\mathbf{T}$, $\mathrm{B}\mathbf{O}$, $\mathrm{B}\mathbf{I}$, $D_{2^{k+1}(2r+1)}$ or $P'_{8\cdot 3^k}$ and $\mathbf{C}_m$ is a cyclic group of order $m$, with $m$ relative prime to the order of $\Gamma$; and the cyclic groups $\mathbb{C}_{n,q}$.
In Table~\ref{tab:MK.graph} there are the McKay graphs of the finite subgroups of $SU(2)$; in the next subsections we compute the McKay quivers of the groups $\mathbb{C}_{n,q}$, $D_{2^k(2r+1)}$ and $P'_{8\cdot 3^k}$ using character theory (see Subsection~\ref{ssec:MQ.CT}).
Then, we compute the McKay quivers of the groups of the form $\Gamma\times\mathbf{C}_m$ using Theorem~\ref{thm:MQ.prod}. Since we are considering the McKay quivers with respect to the natural representation which is faithful, by Proposition~\ref{prop:connect} they are connected.
We draw examples of the different types of these McKay quivers, this shows their many symmetries.

\subsection{McKay quiver of the groups $\mathbb{C}_{n,q}$}\label{ssec:MQ.Cnq}

Recall from Subsection~\ref{sssec:Cn} that the group $\mathbb{C}_{n,q}$, with $\gcd(n,q)=1$, has $n$ irreducible one-dimensional representations $\beta_0,\dots,\beta_{n-1}$ which correspond to each of the vertices of the McKay Quiver $Q_{\rho_{\mathrm{Nat}}}(\mathbb{C}_{n,q})$. From Subsection~\ref{ssec:Cnq} the natural representation $\rho_{\mathrm{Nat}}\colon\mathbb{C}_{n,q}\to U(2)$ is given by the direct sum $\rho_{\mathrm{Nat}}=\beta_1\oplus\beta_{q}$. Let $\chi_j$ be the character of the representation $\beta_j$, from Table~\ref{tab:CT.Cq} and \eqref{eq:aij}  we have
\begin{equation*}
a_{ij}=\inpr{\chi_{\mathrm{Nat}}\chi_i}{\chi_j}=\frac{1}{n}\sum_{l=0}^{n-1}(\zeta_n^l+\zeta_n^{lq})\zeta_n^{li}\zeta_n^{-lj}=\frac{1}{n}\sum_{l=0}^{n-1}\zeta_n^{l(i-j+1)}+\frac{1}{n}\sum_{l=0}^{n-1}\zeta_n^{l(i-j+q)}.
\end{equation*}
Recall that if $n$ is not a divisor of $t$ we have $\sum_{l=0}^{n-1}\zeta_n^{tl}=0$, and if $t\equiv 0\mod n$, then $\sum_{l=0}^{n-1}\zeta_n^{tl}=n$. Hence
\begin{equation*}
a_{ij}=\begin{cases}
        0 & \text{if $j\not\equiv i+1\mod n$ and $j\not\equiv i+q\mod n$,}\\
        1 & \text{if $j\equiv i+1\mod n$ and $j\not\equiv i+q\mod n$,}\\
        1 & \text{if $j\not\equiv i+1\mod n$ and $j\equiv i+q\mod n$.}\\
       \end{cases}
\end{equation*}
Therefore, we have two arrows going out from the vertex $\beta_i$:
\begin{equation}\label{eq:MQ.Cyclic}
\begin{tikzcd}
   & \beta_{i+1}\\
\beta_i\arrow[rd] \arrow[ru] &  \\
& \beta_{i+q}
\end{tikzcd}\qquad \text{with addition modulo $n$.}
\end{equation}

Thus, if we put the vertices of the McKay quiver $Q_{\rho_{\mathrm{Nat}}}(\mathbb{C}_{n,q})$ in the vertices of a regular $n$-gon numerated counterclockwise and we draw \textcolor{red}{red} arrows given by the first congruence and \textcolor{blue}{blue} arrows given by the second congruence, we get an $n$-gone with a cycle of \textcolor{red}{red arrows} on its edges pointing counterclockwise and \textcolor{blue}{blue arrows} on the diagonals going from vertex $\beta_i$ to vertex $\beta_{i+q}$ forming a cycle passing through all the vertices. There are two special cases: when $q=1$ we have an $n$-gone with red arrows on its edges with multiplicity $2$ pointing counterclockwise without blue diagonals; when $q=n-1$ we have that $j\equiv i+n-1\mod n$ is equivalent to $j\equiv i-1\mod n$, thus we get on each edge of the $n$-gone two arrows pointing to opposite directions, hence by our convention we get an unoriented $n$-gon  which is the extended Dynkin diagram $\tilde{A}_{n-1}$ and in this case $\mathbb{C}_{n,n-1}\leq SU(2)$ (see Table~\ref{tab:MK.graph}). In Figure~\ref{fig:MQCnq} we show the McKay quivers $\mathbb{C}_{8,q}$ for $q=1,3,5,7$.

\begin{figure}[H]
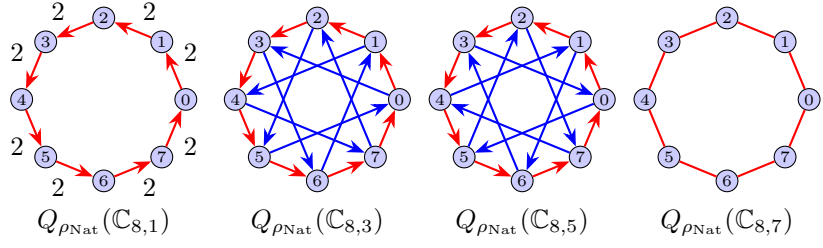

\begin{center}
\MQCnq{8}{1}\MQCnq{8}{3}\MQCnq{8}{5}\MQCnq{8}{7}
\end{center}
\caption{McKay quivers of the groups $\mathbb{C}_{8,q}$.}\label{fig:MQCnq}
\end{figure}

\subsection{McKay quiver of the groups $D_{2^k(2r+1)}$}

From Subsection~\ref{irreducible-representations} the group $D_{2^k(2r+1)}$ has $2^{k-1}(r+2)$ irreducible representations: $2^k$ one-dimensional representations
$\alpha_j$ with $j=0,\dots,2^k-1$ and $2^{k-1}r$ two-dimensional representations $\varrho_{t,s}$ with $t=1,\dots,2r$ and $s=0,\dots,2^{k-2}-1$ which correspond to each of the vertices of the McKay Quiver $Q_{\rho_{\mathrm{Nat}}}(G)$. The natural representation $\rho_{\mathrm{Nat}}\colon D_{2^k(2r+1)}\to U(2)$ is $\rho_{\mathrm{Nat}}=\varrho_{1,1}$, whose character is denoted by $\chi_{1,1}$.

\subsubsection{Arrows going out from one-dimensional representations $\alpha_j$}

Consider a one-dimensional representation $\alpha_i$, recall that we are denoting its character by $\chi_i$. From Table~\ref{tab:CT.D} the character $\chi_{1,1}\cdot\chi_i$ of the representation
$\rho_{\mathrm{Nat}}\otimes\alpha_i=\varrho_{1,1}\otimes\alpha_i$ is given in Table~\ref{tab:Char.Nat.Aj}.

\begin{table}[H]
\begin{equation*}
\setlength{\extrarowheight}{5pt}
\begin{array}{|c|c|c|c|}\hline
\text{\small Class}& \mathsf{1}_l & \mathsf{2}_{l,q} & \mathsf{(2r+1)}_l\\\hline
\chi_{1,1}\cdot\chi_i & (-1)^{l}2\zeta_{2^k}^{2l(i+1)} & (-1)^{l}\zeta_{2^k}^{2l(i+1)}(\zeta_{2r+1}^{q}+\zeta_{2r+1}^{-q}) & 0\\\hline
\multicolumn{4}{l}{\text{Note: $\zeta_{2^k}=e^{\frac{2\pi i}{2^k}}$, $0\leq i\leq 2^k-1$, $0\leq j\leq 2^{k-1}-1$, $1\leq q\leq r$.}}\\
\end{array}
\end{equation*}
\caption{Character of $\rho_{\mathrm{Nat}}\otimes\alpha_i=\varrho_{1,1}\otimes\alpha_i$.}\label{tab:Char.Nat.Aj}
\end{table}

Comparing Table~\ref{tab:Char.Nat.Aj} with Table~\ref{tab:CT.D} we can see that the character $\chi_{1,1}\cdot\chi_i$ is the character $\chi_{1,i+1}$ of the two-dimensional irreducible representation $\varrho_{1,i+1}$. Hence
\begin{equation}\label{eq:1dim.Case}
\rho_{\mathrm{Nat}}\otimes\alpha_i=\varrho_{1,1}\otimes\alpha_i\cong\varrho_{1,i+1}.
\end{equation}

\subsubsection{Arrows going out from two-dimensional representations $\varrho_{t,s}$}

Consider a two-dimensional representation $\varrho_{t,s}$ and its character $\chi_{t,s}$. From Table~\ref{tab:CT.D} the character $\chi_{1,1}\cdot\chi_{t,s}$ of the representation $\rho_{\mathrm{Nat}}\otimes\varrho_{t,s}=\varrho_{1,1}\otimes\varrho_{t,s}$ is given in Table~\ref{tab:Char.Nat.varrhots}.

\begin{table}[H]
\begin{equation*}
\setlength{\extrarowheight}{5pt}
\begin{array}{|c|c|c|c|}\hline
\text{\small Class}& \mathsf{1}_l & \mathsf{2}_{l,q} & \mathsf{(2r+1)}_l\\\hline
\chi_{1,1}\cdot\chi_{t,s} &\scriptstyle (-1)^{l(t+1)}4\zeta_{2^k}^{2l(s+1)} &\scriptstyle (-1)^{l(t+1)}\zeta_{2^k}^{2l(s+1)}(\zeta_{2r+1}^{q}+\zeta_{2r+1}^{-q})(\zeta_{2r+1}^{tq}+\zeta_{2r+1}^{-tq}) & 0\\\hline
\multicolumn{4}{l}{\text{Note: $\zeta_{2^k}=e^{\frac{2\pi i}{2^k}}$, $1\leq t\leq 2r$, $0\leq s\leq 2^{k-2}-1$, $0\leq j\leq 2^{k-1}-1$, $1\leq q\leq r$.}}\\
\end{array}
\end{equation*}
\caption{Character of $\rho_{\mathrm{Nat}}\otimes\alpha_i=\varrho_{1,1}\otimes\alpha_i$.}\label{tab:Char.Nat.varrhots}
\end{table}
Writing
\begin{equation*}
(-1)^{l(t+1)}4\zeta_{2^k}^{2l(s+1)}=(-1)^{l(t+1)}2\zeta_{2^k}^{2l(s+1)}+(-1)^{l(t+1)}2\zeta_{2^k}^{2l(s+1)},
\end{equation*}
and
\begin{multline*}
(-1)^{l(t+1)}\zeta_{2^k}^{2l(s+1)}(\zeta_{2r+1}^{q}+\zeta_{2r+1}^{-q})(\zeta_{2r+1}^{tq}+\zeta_{2r+1}^{-tq})\\
=(-1)^{l(t-1)}\zeta_{2^k}^{2l(s+1)}(\zeta_{2r+1}^{(t-1)q}+\zeta_{2r+1}^{-(t-1)q})+(-1)^{l(t+1)}\zeta_{2^k}^{2l(s+1)}(\zeta_{2r+1}^{(t+1)q}+\zeta_{2r+1}^{-(t+1)q}),
\end{multline*}
and comparing Table~\ref{tab:Char.Nat.varrhots} with Table~\ref{tab:CT.D} we can see that there are three cases.

\paragraph{\textbf{Case A:}}
If $t\neq 1,2r$ the character $\chi_{1,1}\cdot\chi_{t,s}$ is the character $\chi_{t-1,s+1}+\chi_{t+1,s+1}$ of the four-dimensional irreducible representation $\varrho_{t-1,s+1}\oplus\varrho_{t+1,s+1}$. Hence
\begin{equation}\label{eq:2dim.CaseA}
\rho_{\mathrm{Nat}}\otimes\varrho_{t,s}=\varrho_{1,1}\otimes\varrho_{t,s}\cong\varrho_{t-1,s+1}\oplus\varrho_{t+1,s+1}.
\end{equation}
\paragraph{\textbf{Case B:}}
If $t=1$, by Proposition~\ref{prop:t=0} we have that the  character $\chi_{1,1}\cdot\chi_{1,s}$ is the character $\chi_{s+1}+\chi_{2^{k-1}+s+1}+\chi_{2,s+1}$ of the four-dimensional representation $\alpha_{s+1}\oplus\alpha_{2^{k-1}+s+1}\oplus\varrho_{2,s+1}$. Hence
\begin{equation}\label{eq:2dim.CaseB}
\rho_{\mathrm{Nat}}\otimes\varrho_{t,s}=\varrho_{1,1}\otimes\varrho_{1,s}\cong\alpha_{s+1}\oplus\alpha_{2^{k-1}+s+1}\oplus\varrho_{2,s+1}
\end{equation}
\paragraph{\textbf{Case C:}}
If $t=2r$, by Proposition~\ref{prop:t=2r+1} we have that the  character $\chi_{1,1}\cdot\chi_{2r,s}$ is the character $\chi_{2^{k-2}+s+1}+\chi_{2^{k-1}+2^{k-2}+s+1}+\chi_{2r-1,s+1}$ of the four-dimensional representions $\alpha_{2^{k-2}+s+1}\oplus\alpha_{2^{k-1}+2^{k-2}+s+1}\oplus\varrho_{2r-1,s+1}$. Hence
\begin{equation}\label{eq:2dim.CaseC}
\rho_{\mathrm{Nat}}\otimes\varrho_{t,s}=\varrho_{1,1}\otimes\varrho_{2r,s}\cong\alpha_{2^{k-2}+s+1}\oplus\alpha_{2^{k-1}+2^{k-2}+s+1}\oplus\varrho_{2r-1,s+1}.
\end{equation}

\subsubsection{The McKay quiver $Q_{\rho_{\mathrm{Nat}}}(D_{2^k(2r+1)})$}\label{sssec:MQ.Dnq}

The McKay quiver $Q_{\rho_{\mathrm{Nat}}}(D_{2^k(2r+1)})$ is given by the following arrows corresponding to the decompositions \eqref{eq:1dim.Case}, \eqref{eq:2dim.CaseA}, \eqref{eq:2dim.CaseB} and \eqref{eq:2dim.CaseC}.
\begin{flalign*}
\alpha_i&\to \varrho_{1,i+1} & &\text{if $i=0,\dots,2^{k-2}-2$,}\\
\alpha_i&\to \varrho_{1,2^{k-2}+j}\cong\varrho_{2r,j} & &\text{if $i=2^{k-2}+j-1$ with $j=0,\dots,2^{k-2}-1$,}\\
\alpha_i&\to \varrho_{1,2^{k-1}+j}\cong\varrho_{1,j} & &\text{if $i=2^{k-1}+j-1$ with $j=0,\dots,2^{k-2}-1$,}\\
\alpha_i&\to \varrho_{1,3\cdot2^{k-2}+j}\cong\varrho_{2r,j} & &\text{if $i=3\cdot2^{k-2}+j-1$ with $j=0,\dots,2^{k-2}-1$,}\\
\alpha_{2^{k}-1}&\to \varrho_{1,2^{k}}\cong\varrho_{1,0}.\\[2pt]
\end{flalign*}\\
\begin{tikzcd}
   & \alpha_{s+1}\\
\varrho_{1,s}\arrow[rd] \arrow[r] \arrow[ru] & \alpha_{2^{k-1}+s+1} \\
& \varrho_{2,s+1}
\end{tikzcd}\qquad if $s=2^{k-2}-1$ we have $\varrho_{2,2^{k-2}}\cong\varrho_{2r-1,0}$.
\vspace{8pt}

If $t\neq 1$ and $t\neq 2r$

\begin{tikzcd}
   & \varrho_{t-1,s+1}\\
\varrho_{t,s}\arrow[rd] \arrow[ru] &  \\
& \varrho_{t+1,s+1}
\end{tikzcd}\qquad if $s=2^{k-2}-1$ we have
\begin{tikzcd}
 \varrho_{t-1,2^{k-2}}\cong\varrho_{2r+1-(t-1),0},\\
 \\
 \varrho_{t+1,2^{k-2}}\cong\varrho_{2r+1-(t+1),0}.\\[10pt]
\end{tikzcd}

\begin{tikzcd}
   & \alpha_{2^{k-2}+s+1}\\
\varrho_{2r,s}\arrow[rd] \arrow[r] \arrow[ru] & \alpha_{2^{k-1}+2^{k-2}+s+1} \\
& \varrho_{2r-1,s+1}
\end{tikzcd}\qquad if $s=2^{k-2}-1$ we have $\varrho_{2r-1,2^{k-2}}\cong\varrho_{2,0}$.

Figures~\ref{fig:MQ:D4.1} and \ref{fig:MQ:D4.2} show the McKay quivers $Q_{\rho_{\mathrm{Nat}}}(D_{2^4(3)})$ and $Q_{\rho_{\mathrm{Nat}}}(D_{2^4(5)})$.
We are labeling the vertices corresponding to the one-dimensional representations $\alpha_i$ with the index $i$, and the vertices corresponding to the two-dimensional representions $\varrho_{t,s}$ with the pair $(t,s)$.
\begin{figure}[H]
\begin{center}
\begin{tikzpicture}[scale=1,>=Stealth,auto=left,every node/.style={circle,fill=blue!20,inner sep=1pt,font=\tiny,draw},twodim/.style={circle,fill=green!20,inner sep=0pt,font=\tiny,draw}]
\begin{scope}[name prefix = in-]
\node (so) at  (-1,-1) {$15$};
\node[twodim] (o) at  (-1,0) {$(2,2)$};
\node (no) at  (-1,1) {$13$};
\node[twodim] (n) at  (0,1) {$(2,0)$};
\node (ne) at  (1,1) {$11$};
\node[twodim] (e) at  (1,0) {$(1,2)$};
\node (se) at  (1,-1) {$9$};
\node[twodim] (s) at  (0,-1) {$(1,0)$};
\node (wso) at  (-.25,-.25) {$7$};
\node (wno) at  (-.25,.25) {$5$};
\node (wne) at  (.25,.25) {$3$};
\node (wse) at  (.25,-.25) {$1$};
\draw [->, thick] (o) -- (so);
\draw [->, thick] (no) -- (o);
\draw [->, red, thick] (n) -- (no);
\draw [->, thick] (ne) -- (n);
\draw [->, thick] (e) -- (ne);
\draw [->, thick] (se) -- (e);
\draw [->, thick] (s) -- (se);
\draw [->, thick] (so) -- (s);
\draw [->, thick] (wso) -- (s);
\draw [->, thick] (o) -- (wso);
\draw [->, thick] (wno) -- (o);
\draw [->, red, thick] (n) -- (wno);
\draw [->, thick] (wne) -- (n);
\draw [->, thick] (e) -- (wne);
\draw [->, thick] (wse) -- (e);
\draw [->, thick] (s) -- (wse);
\end{scope}
\begin{scope}[name prefix = out-]
\node[twodim] (so) at  (-2,-2) {$(1,3)$};
\node (o) at  (-2,0) {$10$};
\node[twodim] (no) at  (-2,2) {$(1,1)$};
\node (n) at  (0,2) {$8$};
\node[twodim] (ne) at  (2,2) {$(2,3)$};
\node (e) at  (2,0) {$14$};
\node[twodim] (se) at  (2,-2) {$(2,1)$};
\node (s) at  (0,-2) {$12$};
\node (wn) at (0,3) {$0$};
\node (we) at (3,0) {$6$};
\node (ws) at (0,-3) {$4$};
\node (wo) at (-3,0) {$2$};
\draw [->, thick] (o) -- (so);
\draw [->, thick] (no) -- (o);
\draw [->, red, thick] (n) -- (no);
\draw [->, thick] (ne) -- (n);
\draw [->, thick] (e) -- (ne);
\draw [->, thick] (se) -- (e);
\draw [->, thick] (s) -- (se);
\draw [->, thick] (so) -- (s);
\draw [->, thick] (wo) -- (so);
\draw [->, thick] (no) -- (wo);
\draw [->, red, thick] (wn) -- (no);
\draw [->, thick] (ne) -- (wn);
\draw [->, thick] (we) -- (ne);
\draw [->, thick] (se) -- (we);
\draw [->, thick] (ws) -- (se);
\draw [->, thick] (so) -- (ws);
\end{scope}
\draw [->, thick] (in-o) -- (out-so);
\draw [->, thick] (out-no) -- (in-o);
\draw [->, red, thick] (in-n) -- (out-no);
\draw [->, thick] (out-ne) -- (in-n);
\draw [->, thick] (in-e) -- (out-ne);
\draw [->, thick] (out-se) -- (in-e);
\draw [->, thick] (in-s) -- (out-se);
\draw [->, thick] (out-so) -- (in-s);
\end{tikzpicture}
\end{center}
\caption{McKay quiver $Q_{\rho_{\mathrm{Nat}}}(D_{2^4(3)})$.}\label{fig:MQ:D4.1}
\end{figure}
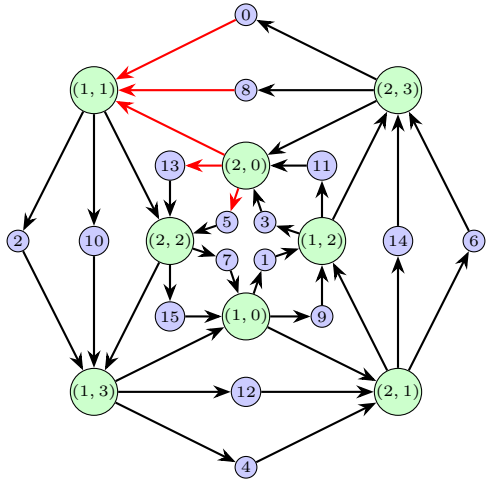

\begin{figure}[H]
\begin{center}
\begin{tikzpicture}
[scale=.9,>=Stealth,auto=left,every node/.style={circle,fill=blue!20,inner sep=1pt,font=\tiny,draw},twodim/.style={circle,fill=green!20,inner sep=0pt,font=\tiny,draw}]
\begin{scope}[name prefix = in-]
\node (so) at  (-1,-1) {$15$};
\node[twodim] (o) at  (-1,0) {$(4,2)$};
\node (no) at  (-1,1) {$13$};
\node[twodim] (n) at  (0,1) {$(4,0)$};
\node (ne) at  (1,1) {$11$};
\node[twodim] (e) at  (1,0) {$(1,2)$};
\node (se) at  (1,-1) {$9$};
\node[twodim] (s) at  (0,-1) {$(1,0)$};
\node (wso) at  (-.25,-.25) {$7$};
\node (wno) at  (-.25,.25) {$5$};
\node (wne) at  (.25,.25) {$3$};
\node (wse) at  (.25,-.25) {$1$};
\draw [->, thick] (o) -- (so);
\draw [->, thick] (no) -- (o);
\draw [->, red, thick] (n) -- (no);
\draw [->, thick] (ne) -- (n);
\draw [->, thick] (e) -- (ne);
\draw [->, thick] (se) -- (e);
\draw [->, thick] (s) -- (se);
\draw [->, thick] (so) -- (s);
\draw [->, thick] (wso) -- (s);
\draw [->, thick] (o) -- (wso);
\draw [->, thick] (wno) -- (o);
\draw [->, red, thick] (n) -- (wno);
\draw [->, thick] (wne) -- (n);
\draw [->, thick] (e) -- (wne);
\draw [->, thick] (wse) -- (e);
\draw [->, thick] (s) -- (wse);
\end{scope}
\begin{scope}[name prefix = mid-]
\node[twodim] (so) at  (-2,-2) {$(3,3)$};
\node[twodim] (o) at  (-2,0) {$(2,2)$};
\node[twodim] (no) at  (-2,2) {$(3,1)$};
\node[twodim] (n) at  (0,2) {$(2,0)$};
\node[twodim] (ne) at  (2,2) {$(2,3)$};
\node[twodim] (e) at  (2,0) {$(3,2)$};
\node[twodim] (se) at  (2,-2) {$(2,1)$};
\node[twodim] (s) at  (0,-2) {$(3,0)$};
\draw [->, thick] (o) -- (so);
\draw [->, thick] (no) -- (o);
\draw [->, red, thick] (n) -- (no);
\draw [->, thick] (ne) -- (n);
\draw [->, thick] (e) -- (ne);
\draw [->, thick] (se) -- (e);
\draw [->, thick] (s) -- (se);
\draw [->, thick] (so) -- (s);
\end{scope}
\draw [->, thick] (in-o) -- (mid-so);
\draw [->, thick] (mid-no) -- (in-o);
\draw [->, red, thick] (in-n) -- (mid-no);
\draw [->, thick] (mid-ne) -- (in-n);
\draw [->, thick] (in-e) -- (mid-ne);
\draw [->, thick] (mid-se) -- (in-e);
\draw [->, thick] (in-s) -- (mid-se);
\draw [->, thick] (mid-so) -- (in-s);
\begin{scope}[name prefix = out-]
\node[twodim] (so) at  (-3,-3) {$(1,3)$};
\node (o) at  (-3,0) {$10$};
\node[twodim] (no) at  (-3,3) {$(1,1)$};
\node (n) at  (0,3) {$8$};
\node[twodim] (ne) at  (3,3) {$(4,3)$};
\node (e) at  (3,0) {$14$};
\node[twodim] (se) at  (3,-3) {$(4,1)$};
\node (s) at  (0,-3) {$12$};
\node (wn) at (0,4) {$0$};
\node (we) at (4,0) {$6$};
\node (ws) at (0,-4) {$4$};
\node (wo) at (-4,0) {$2$};
\draw [->, thick] (o) -- (so);
\draw [->, thick] (no) -- (o);
\draw [->, red, thick] (n) -- (no);
\draw [->, thick] (ne) -- (n);
\draw [->, thick] (e) -- (ne);
\draw [->, thick] (se) -- (e);
\draw [->, thick] (s) -- (se);
\draw [->, thick] (so) -- (s);
\draw [->, thick] (wo) -- (so);
\draw [->, thick] (no) -- (wo);
\draw [->, red, thick] (wn) -- (no);
\draw [->, thick] (ne) -- (wn);
\draw [->, thick] (we) -- (ne);
\draw [->, thick] (se) -- (we);
\draw [->, thick] (ws) -- (se);
\draw [->, thick] (so) -- (ws);
\end{scope}
\draw [->, thick] (mid-o) -- (out-so);
\draw [->, thick] (out-no) -- (mid-o);
\draw [->, red, thick] (mid-n) -- (out-no);
\draw [->, thick] (out-ne) -- (mid-n);
\draw [->, thick] (mid-e) -- (out-ne);
\draw [->, thick] (out-se) -- (mid-e);
\draw [->, thick] (mid-s) -- (out-se);
\draw [->, thick] (out-so) -- (mid-s);
\end{tikzpicture}
\end{center}
\caption{McKay quiver $Q_{\rho_{\mathrm{Nat}}}(D_{2^4(5)})$.}\label{fig:MQ:D4.2}
\end{figure}
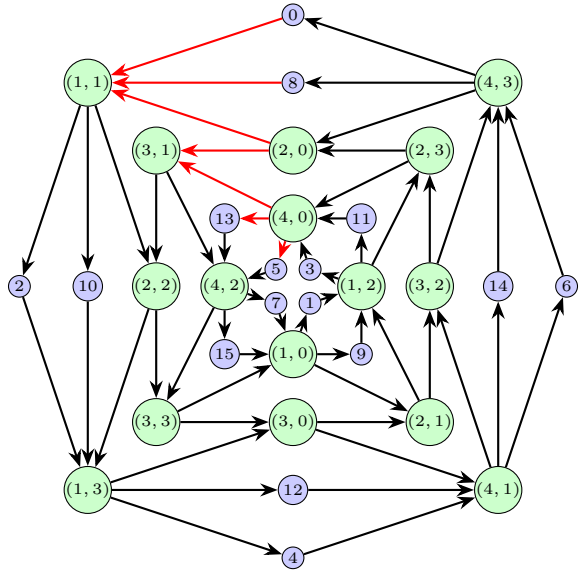

\subsection{McKay quiver of the groups $P'_{8\cdot 3^k}$}

From Subsection~\ref{irreducible-representations.P} the group $P'_{8\cdot 3^k}$ has $7\cdot3^{k-1}$ irreducible representations: $3^k$ one-dimensional representations
$\alpha_j$ with $j=0,\dots,3^k-1$, $3^{k}$ two-dimensional representations $\varrho_{t}$ with $j=0,\dots,3^k-1$ and $3^{k-1}$ three-dimensional representations $\varsigma_s$ with $s=0,\dots,3^{k-1}-1$ which correspond to each of the vertices of the McKay Quiver $Q_{\rho_{\mathrm{Nat}}}(G)$. The natural representation $\rho_{\mathrm{Nat}}\colon P'_{8\cdot 3^k}\to U(2)$ is $\rho_{\mathrm{Nat}}=\varrho_{1}$, whose character is denoted by $\psi_{1}$.

\subsubsection{Arrows going out from one-dimensional representations $\alpha_j$}

Consider a one-dimensional representation $\alpha_i$, recall that we are denoting its character by $\chi_i$. From Table~\ref{tab:CT.P} the character $\psi_{1}\cdot\chi_i$ of the representation
$\rho_{\mathrm{Nat}}\otimes\alpha_i=\varrho_{1}\otimes\alpha_i$ is given in Table~\ref{tab:Char.Nat.1dim}.

\begin{table}[H]
\begin{equation*}
\setlength{\extrarowheight}{5pt}
\begin{array}{|c|c|c|c|c|c|c|}\hline
 \mathsf{1}_l & \mathsf{1}_{l}^{\mathsf{+}} & \mathsf{4}_l^{\mathsf{a}} & \mathsf{4}_l^{\mathsf{b}} & \mathsf{4}_l^{\mathsf{c}} & \mathsf{4}_l^{\mathsf{d}} & \mathsf{6}_l\\\hline
 2\zeta_{3^k}^{3l(i+1)} & -2\zeta_{3^k}^{3l(i+1)} & -\zeta_{3^k}^{(3l+1)(i+1)} & -\zeta_{3^k}^{(3l+2)(i+1)} & \zeta_{3^k}^{(3l+1)(i+1)} & \zeta_{3^k}^{(3l+2)(i+1)} & 0\\\hline
\multicolumn{7}{l}{\text{Note: $\zeta_{3^k}=e^{\frac{2\pi i}{3^k}}$, $0\leq i\leq 3^{k}-1$, $0\leq j\leq 3^{k-1}-1$.}}
\end{array}
\end{equation*}
\caption{Character of $\rho_{\mathrm{Nat}}\otimes\alpha_i=\varrho_{1}\otimes\alpha_i$.}\label{tab:Char.Nat.1dim}
\end{table}

Comparing Table~\ref{tab:Char.Nat.1dim} with Table~\ref{tab:CT.P} we can see that the character $\psi_{1}\cdot\chi_i$ is the character $\psi_{i+1}$ of the two-dimensional irreducible representation $\varrho_{i+1}$. Hence
\begin{equation}\label{eq:1dim.Case.P}
\rho_{\mathrm{Nat}}\otimes\alpha_i=\varrho_{1}\otimes\alpha_i\cong\varrho_{i+1}.
\end{equation}

\subsubsection{Arrows going out from two-dimensional representations $\varrho_i$}

Consider a two-dimensional representation $\varrho_{i}$, recall that we are denoting its character by $\psi_{i}$. From Table~\ref{tab:CT.P} the character $\psi_{1}\cdot\psi_{i}$ of the representation $\rho_{\mathrm{Nat}}\otimes\varrho_{i}=\varrho_{1}\otimes\varrho_{i}$ is given in Table~\ref{tab:Char.Nat.2dim.P}.

\begin{table}[H]
\begin{equation*}
\setlength{\extrarowheight}{5pt}
\begin{array}{|c|c|c|c|c|c|c|}\hline
\mathsf{1}_l & \mathsf{1}_{l}^{\mathsf{+}} & \mathsf{4}_l^{\mathsf{a}} & \mathsf{4}_l^{\mathsf{b}} & \mathsf{4}_l^{\mathsf{c}} & \mathsf{4}_l^{\mathsf{d}} & \mathsf{6}_l\\\hline
 4\zeta_{3^k}^{3l(i+1)} & 4\zeta_{3^k}^{3l(i+1)} & \zeta_{3^k}^{(3l+1)(i+1)} & \zeta_{3^k}^{(3l+2)(i+1)} & \zeta_{3^k}^{(3l+1)(i+1)} & \zeta_{3^k}^{(3l+2)(i+1)} & 0\\\hline
\multicolumn{7}{l}{\text{Note: $\zeta_{3^k}=e^{\frac{2\pi i}{3^k}}$, $0\leq i\leq 3^{k}-1$, $0\leq j\leq 3^{k-1}-1$.}}
\end{array}
\end{equation*}
\caption{Character of $\rho_{\mathrm{Nat}}\otimes\varrho_{i}=\varrho_{1}\otimes\varrho_{i}$.}\label{tab:Char.Nat.2dim.P}
\end{table}

Writing $4\zeta_{3^k}^{3l(i+1)}=\zeta_{3^k}^{3l(i+1)}+3\zeta_{3^k}^{3l(i+1)}$ and $0=\zeta_{3^k}^{3l(i+1)}-\zeta_{3^k}^{3l(i+1)}$ and
comparing Table~\ref{tab:Char.Nat.2dim.P} with Table~\ref{tab:CT.P} we can see that the character $\psi_{1}\cdot\psi_{i}$ is the character $\chi_{i+1}+\varphi_{i+1}$ of the four-dimensional representation $\alpha_{i+1}\oplus\varsigma_{i+1}$. Hence
\begin{equation}\label{eq:2dim.Case.P}
\rho_{\mathrm{Nat}}\otimes\varrho_i=\varrho_{1}\otimes\varrho_i\cong\alpha_{i+1}\oplus\varsigma_{i+1}.
\end{equation}

\subsubsection{Arrows going out from three-dimensional representations $\varsigma_s$}

Consider a three-dimensional representation $\varsigma_{s}$, recall that we are denoting its character by $\varphi_{s}$. From Table~\ref{tab:CT.P} the character $\psi_{1}\cdot\varphi_{s}$ of the representation $\rho_{\mathrm{Nat}}\otimes\varsigma_{s}=\varrho_{1}\otimes\varsigma_{s}$ is given in Table~\ref{tab:Char.Nat.3dim.P}.

\begin{table}[H]
\begin{equation*}
\setlength{\extrarowheight}{5pt}
\begin{array}{|c|c|c|c|c|c|c|c|}\hline
\text{\small Class}& \mathsf{1}_l & \mathsf{1}_{l}^{\mathsf{+}} & \mathsf{4}_l^{\mathsf{a}} & \mathsf{4}_l^{\mathsf{b}} & \mathsf{4}_l^{\mathsf{c}} & \mathsf{4}_l^{\mathsf{d}} & \mathsf{6}_l\\\hline
\psi_{1}\cdot\varphi_{s} & 6\zeta_{3^k}^{3l(s+1)} & -6\zeta_{3^k}^{3l(s+1)} & 0 &  0 &  0 &  0 & 0\\\hline
\multicolumn{8}{l}{\text{Note: $\zeta_{3^k}=e^{\frac{2\pi i}{3^k}}$, $0\leq s\leq 3^{k-1}-1$, $0\leq j\leq 3^{k-1}-1$.}}
\end{array}
\end{equation*}
\caption{Character of $\rho_{\mathrm{Nat}}\otimes\varsigma_{s}=\varrho_{1}\otimes\varsigma_{s}$.}\label{tab:Char.Nat.3dim.P}
\end{table}

Observing that
\begin{equation*}
\pm6\zeta_{3^k}^{3l(s+1)}=\pm2\zeta_{3^k}^{3l(s+1)}\pm2\zeta_{3^k}^{3l(3^{k-1}+s+1)}\pm2\zeta_{3^k}^{3l(2\cdot3^{k-1}+s+1)}
\end{equation*}
and that for $m=1,2$ we have
\begin{align*}
&\zeta_{3^k}^{(3l+m)(s+1)}+\zeta_{3^k}^{(3l+m)(3^{k-1}+s+1)}+\zeta_{3^k}^{(3l+m)(2\cdot3^{k-1}+s+1)}\\
&=\zeta_{3^k}^{(3l+m)(s+1)}+\zeta_{3^k}^{(3l+m)3^{k-1}}\zeta_{3^k}^{(3l+m)(s+1)}+\zeta_{3^k}^{(3l+m)2\cdot3^{k-1}}\zeta_{3^k}^{(3l+m)(s+1)}\\
&=\zeta_{3^k}^{(3l+m)(s+1)}+\zeta_{3}^{(3l+m)}\zeta_{3^k}^{(3l+m)(s+1)}+\zeta_{3}^{2(3l+m)}\zeta_{3^k}^{(3l+m)(s+1)}\\
&=\zeta_{3^k}^{(3l+m)(s+1)}\bigl(1+\zeta_{3}^{m}+\zeta_{3}^{2m}\bigr)=\zeta_{3^k}^{(3l+m)(s+1)}\bigl(1+\zeta_{3}+\zeta_{3}^{2}\bigr)=0.
\end{align*}
and comparing Table~\ref{tab:Char.Nat.3dim.P} with Table~\ref{tab:CT.P} we can see that the character $\psi_{1}\cdot\varphi_{s}$ is the character $\psi_{s+1}+\psi_{3^{k-1}+s+1}+\psi_{2\cdot3^{k-1}+s+1}$ of the six-dimensional representation $\varrho_{s+1}\oplus\varrho_{3^{k-1}+s+1}\oplus\varrho_{2\cdot3^{k-1}+s+1}$. Hence
\begin{equation}\label{eq:3dim.Case.P}
\rho_{\mathrm{Nat}}\otimes\varsigma_s=\varrho_{1}\otimes\varsigma_{s}\cong\varrho_{s+1}\oplus\varrho_{3^{k-1}+s+1}\oplus\varrho_{2\cdot3^{k-1}+s+1}.
\end{equation}

\subsubsection{The McKay quiver $Q_{\rho_{\mathrm{Nat}}}(P'_{8\cdot 3^k})$}\label{sssec:MQ.P}

In summary, the McKay quiver $Q_{\rho_{\mathrm{Nat}}}(P'_{8\cdot 3^k})$ is given by the following arrows corresponding to the decompositions \eqref{eq:1dim.Case.P}, \eqref{eq:2dim.Case.P} and \eqref{eq:3dim.Case.P}. Remember that the indices $j$ of $\alpha_j$ and $\varrho_j$ are modulo $3^k$ and the indices $s$ of $\varsigma_s$ are modulo $3^{k-1}$.

\begin{equation}\label{mq:P83k}
\begin{tikzcd}
\alpha_{i} \arrow[r]  & \varrho_{i+1}\\
\end{tikzcd}\qquad\qquad
\begin{tikzcd}
   & \alpha_{i+1} \\
\varrho_{i}\arrow[rd] \arrow[ru] &  \\
& \varsigma_{i+1}\\
\end{tikzcd}\qquad\qquad
\begin{tikzcd}
   & \varrho_{s+1}\\
\varsigma_{s}\arrow[rd] \arrow[r] \arrow[ru] & \varrho_{3^{k-1}+s+1} \\
& \varrho_{2\cdot3^{k-1}+s+1}
\end{tikzcd}
\end{equation}

Figure~\ref{fig:MQ:P3.2} shows the McKay quiver $Q_{\rho_{\mathrm{Nat}}}(P'_{8\cdot 3^2})$.
The vertices corresponding to the one-dimensional representions $\alpha_{j}$ are \textcolor{blue}{blue}, the vertices corresponding to the two-dimensional representions $\varrho_{j}$ are \textcolor{green}{green} and the vertices corresponding to the three-dimensional representions $\varsigma_{s}$ are \textcolor{red}{red}. Note that this quiver is not planar.
\begin{figure}[H]
\begin{center}
\begin{tikzpicture}[scale=.8,>=Stealth,auto=right]
 \tikzmath{
\n=9; 
\s=\n-1;
\r=\n/3-1;
}
\begin{scope}[name prefix = 1-,every node/.style={circle,fill=blue!20,inner sep=1pt,font=\tiny,draw}]
\foreach \i in {0,1,...,\s} {\node (a\i) at ({4.5*cos(4*(360/\n)*\i)},{4.5*sin(4*(360/\n)*\i)}) {$\alpha_{\i}$};};
\end{scope}
\begin{scope}[name prefix = 2-,every node/.style={circle,fill=green!20,inner sep=1pt,font=\tiny,draw}]
\foreach \i in {0,1,...,\s} {\node (a\i) at ({2.5*cos(4*(360/\n)*\i+(360/2))},{2.5*sin(4*(360/\n)*\i+(360/2))}) {$\varrho_{\i}$};};
\end{scope}
\begin{scope}[name prefix = 3-,every node/.style={circle,fill=red!20,inner sep=1pt,font=\tiny,draw}]
\foreach \i in {0,1,...,\r} {\node (a\i) at ({1*cos(3*(360/\n)*\i+3*(360/(2*\n))},{1*sin(3*(360/\n)*\i+3*(360/(2*\n))}) {$\varsigma_{\i}$};};
\end{scope}
\foreach \i [evaluate=\i as \k using {int(mod(\i+1,\n))}] in {0,1,...,\s} {\draw [->, thick, blue] (1-a\i) -- (2-a\k);};
\foreach \i [evaluate=\i as \k using {int(mod(\i+1,\n))}] in {0,1,...,\s} {\draw [->, thick] (2-a\i) -- (1-a\k);};
\foreach \i [evaluate=\i as \k using {int(mod(\i+1,\n))},evaluate=\i as \j using {int(mod(\i+1,\r+1))}] in {0,1,...,\s} {\draw [->, thick] (2-a\i) -- (3-a\j);};
\foreach \i [evaluate=\i as \k using {int(mod(\i+1,\n))},evaluate=\i as \j using {int(mod(\i+4,\n))},evaluate=\i as \l using {int(mod(\i+7,\n))}] in {0,1,...,\r} {\draw [->, thick, red] (3-a\i) -- (2-a\k);\draw [->, thick, red] (3-a\i) -- (2-a\j);\draw [->, thick, red] (3-a\i) -- (2-a\l);};
\end{tikzpicture}
\end{center}
\caption{McKay quiver $Q_{\rho_{\mathrm{Nat}}}(P'_{8\cdot 3^2})$}\label{fig:MQ:P3.2}
\end{figure}
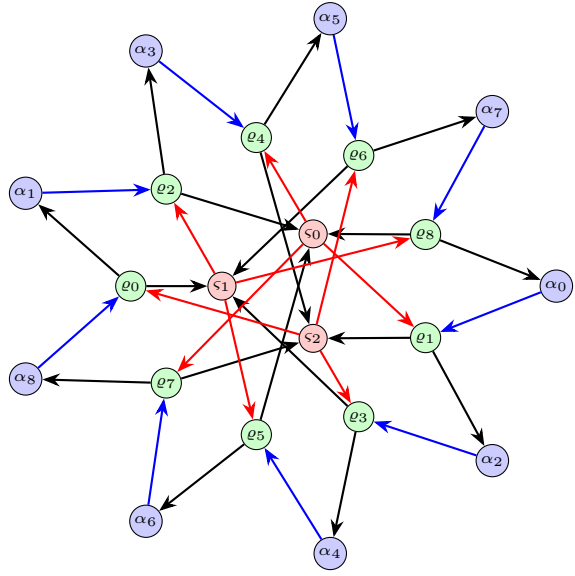

\subsection{McKay quivers of groups of the form $\Gamma\times\mathbf{C}_m$}\label{ssec:MQ.GxCm}

Now that we know the McKay quivers $Q_{\rho_{\mathrm{Nat}}}(\Gamma)$, where $\Gamma$ is one of the groups $\mathrm{B}\mathbf{D}_{2q}$, $\mathrm{B}\mathbf{T}$, $\mathrm{B}\mathbf{O}$, $\mathrm{B}\mathbf{I}$, $D_{2^{k+1}(2r+1)}$ and $P'_{8\cdot 3^k}$, given in Table~\ref{tab:MK.graph} and Subsections~\ref{sssec:MQ.Dnq} and \ref{sssec:MQ.P}, using Theorem~\ref{thm:MQ.prod} we can compute the McKay quivers $Q_{\rho_{\mathrm{NAT}}}(\Gamma\times\mathbf{C}_m)$ of the rest of the small finite subgroups of $U(2)$, which are of the form $\Gamma\times\mathbf{C}_m$ with $m$ relatively prime to the order of $\Gamma$ (see Remark~\ref{rem:GxC}).

The natural representation $\rho_{\mathrm{NAT}}$ of the group $\Gamma\times\mathbf{C}_m$ is the tensor product $\rho_{\mathrm{Nat}}\otimes\beta_1$, where $\rho_{\mathrm{Nat}}$ is the natural representation of the group $\Gamma$ and $\beta_1$ is the one-dimensional representation of $\mathbf{C}_m$ given by $\beta_1(\zeta_{m})=\zeta_m$ (see Subsection~\ref{sssec:Cn}).

In order to use Theorem~\ref{thm:MQ.prod} we need the McKay quiver $Q_{\beta_{1}}(\mathbf{C}_m)$. Recall from Subsection~\ref{sssec:Cn} that $\mathbf{C}_m$ has $m$ one-dimensional representations $\beta_j$, $0\leq j< m$ with characters denoted by $\chi_j$.
From Table~\ref{tab:CT.Cq} and \eqref{eq:aij}  we have
\begin{equation*}
b_{jk}=\inpr{\chi_{1}\chi_j}{\chi_k}=\frac{1}{m}\sum_{q=0}^{m-1}\zeta_m^q\zeta_m^{jq}\zeta_m^{-kq}=\frac{1}{m}\sum_{q=0}^{m-1}\zeta_m^{q(j-k+1)}.
\end{equation*}
Recall that if $m$ is not a divisor of $t$ we have $\sum_{q=0}^{m-1}\zeta_n^{tq}=0$, and if $t\equiv 0\mod m$, then $\sum_{q=0}^{m-1}\zeta_m^{tq}=m$.
Hence the McKay matrix $A_{\rho_1}(\mathbf{C}_m)=\{b_{jk}\}_{j,k=0}^{m-1}$ of $\mathbf{C}_m$ relative to $\beta_1$ is given by
\begin{equation*}
b_{jk}=\begin{cases}
        0 & \text{if $k\not\equiv j+1\mod m$,}\\
        1 & \text{if $k\equiv j+1\mod m$.}\\
       \end{cases}
\end{equation*}
Therefore, in the McKay quiver $Q_{\beta_{1}}(\mathbf{C}_m)$ we have one arrow going out from the vertex $\beta_j$:
\begin{equation}\label{eq:Qb1.Cm}
 \beta_j\rightarrow \beta_{j+1\pmod m}.
\end{equation}
Figure~\ref{fig:MQCn.rho1} shows some examples of McKay quivers $Q_{\beta_{1}}(\mathbf{C}_m)$.

\begin{figure}[H]
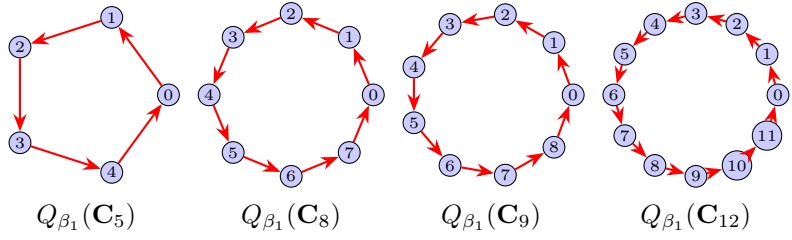

\begin{center}
\MQCn{5}\MQCn{8}\MQCn{9}\MQCn{12}
\end{center}
\caption{McKay quiver $Q_{\beta_{1}}(\mathbf{C}_m)$.}\label{fig:MQCn.rho1}
\end{figure}

Let $\psi_i$, with $0\leq i\leq r$, be the irreducible characters of $\Gamma$. Let $A_{\rho_\mathrm{Nat}}(\Gamma)=\{a_{ih}\}_{i,h=0}^r$ be the McKay matrix of $\Gamma$ relative to the natural representation $\rho_\mathrm{Nat}$.
By Theorem~\ref{thm:CT.prod} the irreducible characters of $\Gamma\times\mathbf{C}_m$ are
\begin{equation*}
 \psi_i\times\chi_j,\quad 0\leq i\leq r,\ 0\leq j< m.
\end{equation*}
By Theorem~\ref{thm:MQ.prod} we have the following proposition.

\begin{proposition}\label{prop:MQ.GxCm}
The McKay quiver $Q_{\rho_{\mathrm{NAT}}}(\Gamma\times\mathbf{C}_m)$ is given as follows:
\begin{description}
 \item[Vertices] The vertices are given by pairs
\begin{equation*}
 (\psi_i,\chi_j),\quad 0\leq i\leq r,\ 0\leq j< m.
\end{equation*}
\item[Arrows] From the vertex $(\psi_i,\chi_j)$ there are $a_{ih}$ arrows to the vertex $(\psi_h,\chi_{j+1})$
\begin{equation*}
 (\psi_i,\chi_j)\xrightarrow{a_{ih}}(\psi_h,\chi_{j+1}).
\end{equation*}
\end{description}
\end{proposition}
Following we present examples of McKay quivers $Q_{\rho_{\mathrm{NAT}}}(\Gamma\times\mathbf{C}_m)$.

\subsubsection{McKay quiver $Q_{\rho_{\mathrm{NAT}}}(\mathrm{B}\mathbf{D}_{2(4)}\times\mathbf{C}_3)$}\label{sssec:D2(4)C3}

By Theorem~\ref{thm:sfs} the group $\mathrm{B}\mathbf{D}_{2(4)}\times\mathbf{C}_3$ is the small dihedral group
\begin{equation*}
\mathbb{D}_{7,4}\cong\mathrm{B}\mathbf{D}_{2(4)}\times\mathbf{C}_3.
\end{equation*}
The group $\mathbf{D}_{2(4)}$ has $4$ one-dimensional irreducible representations $\rho_i$ with $i=0,1,2,3$, and $3$ two-dimensional irreducible representations $\rho_{3+t}$ with $t=1,2,3$. Hence, $\mathrm{B}\mathbf{D}_{2(4)}\times\mathbf{C}_3$ has $21$ irreducible representations given by
\begin{equation*}
\rho_i\times\beta_j,\quad 0\leq i<7,\ 0\leq j<3.
\end{equation*}
Thus, there are $21$ vertices in the McKay quiver $Q_{\rho_{\mathrm{NAT}}}(\mathrm{B}\mathbf{D}_{2(4)}\times\mathbf{C}_3)$.
Figure~\ref{fig:MQ.D2q.Cm} shows the McKay quiver $Q_{\rho_{\mathrm{NAT}}}(\mathrm{B}\mathbf{D}_{2(4)}\times\mathbf{C}_3)$.
We are labeling the vertices corresponding to the representations $\rho_i\times\beta_j$ with $\overset{\textcolor{red}{j}}{(i)}$, the vertices corresponding to the one-dimensional representions are \textcolor{blue}{blue} and the vertices corresponding to the two-dimensional representions are \textcolor{green}{green}.

\begin{figure}[H]
\begin{center}
\begin{tikzpicture}[>=Stealth,auto=right]
\tikzmath{
\n=3; 
\s=\n-1;
\r=\n/3-1;
}
\begin{scope}[every node/.style={circle,fill=green!20,inner sep=0pt,font=\tiny,draw}]
\foreach \i [evaluate=\i as \j using {int(mod(2*\i,\n))}] in {0,1,...,\s} {\node (a\i) at ({5*cos((360/\n)*\i-(360/4)},{5*sin((360/\n)*\i-(360/4)}) {$\overset{\textcolor{red}{\j}}{(4)}$};};
\foreach \i [evaluate=\i as \j using {int(mod(\i+1,\n))},evaluate=\i as \k using {int(mod(2*\i+1,\n))}] in {0,...,\s} \path (a\i) -- (a\j) node (b\i) [pos=0.5] {$\overset{\textcolor{red}{\k}}{(5)}$};
\foreach \i [evaluate=\i as \j using {int(mod(\i+1,\n))},evaluate=\i as \k using {int(mod(2*\i+2,\n))}] in {0,...,\s} \path (b\i) -- (b\j) node (c\i) [pos=0.5] {$\overset{\textcolor{red}{\k}}{(6)}$};
\end{scope}
\begin{scope}[every node/.style={circle,fill=blue!20,inner sep=0pt,font=\tiny,draw}]
\foreach \i [evaluate=\i as \j using {int(mod(\i+1,\n))},evaluate=\i as \k using {int(mod(2*\i,\n))}] in {0,...,\s} {\path (c\i) -- (c\j) node (d\i) [pos=0.5] {$\overset{\textcolor{red}{\k}}{(2)}$};};
\foreach \i [evaluate=\i as \k using {int(mod(2*\i,\n))}] in {0,1,...,\s} {\node (e\i) at ({1.7*cos((360/\n)*\i+(360/4)},{1.7*sin((360/\n)*\i+(360/4)}) {$\overset{\textcolor{red}{\k}}{(3)}$};};
\foreach \i [evaluate=\i as \k using {int(mod(2*\i,\n))}] in {0,1,...,\s} {\node (f\i) at ({3.5*cos((360/\n)*\i+(360/4)},{3.5*sin((360/\n)*\i+(360/4)}) {$\overset{\textcolor{red}{\k}}{(1)}$};};
\foreach \i [evaluate=\i as \k using {int(mod(2*\i,\n))}] in {0,1,...,\s} {\node (g\i) at ({4.5*cos((360/\n)*\i+(360/4)},{4.5*sin((360/\n)*\i+(360/4)}) {$\overset{\textcolor{red}{\k}}{(0)}$};};
\end{scope}
\draw [->, thick] (a1) -- (g0);
\draw [->, thick] (a1) -- (f0);
\draw [->, thick] (a2) -- (g1);
\draw [->, thick] (a2) -- (f1);
\draw [->, thick] (a0) -- (g2);
\draw [->, thick] (a0) -- (f2);
\draw [->, thick] (c0) -- (e0);
\draw [->, thick] (c0) -- (d0);
\draw [->, thick] (c1) -- (e1);
\draw [->, thick] (c1) -- (d1);
\draw [->, thick] (c2) -- (e2);
\draw [->, thick] (c2) -- (d2);
\draw [->, thick] (g1) -- (a0);
\draw [->, thick] (f1) -- (a0);
\draw [->, thick] (b2) -- (a0);
\draw [->, thick] (g2) -- (a1);
\draw [->, thick] (f2) -- (a1);
\draw [->, thick] (b0) -- (a1);
\draw [->, red, thick] (g0) -- (a2);
\draw [->, red, thick] (f0) -- (a2);
\draw [->, red, thick] (b1) -- (a2);
\draw [->, thick] (a0) -- (b0);
\draw [->, thick] (c2) -- (b0);
\draw [->, thick] (a1) -- (b1);
\draw [->, thick] (c0) -- (b1);
\draw [->, thick] (a2) -- (b2);
\draw [->, thick] (c1) -- (b2);
\draw [->, thick] (b0) -- (c0);
\draw [->, thick] (e2) -- (c0);
\draw [->, thick] (d2) -- (c0);
\draw [->, red, thick] (b1) -- (c1);
\draw [->, red, thick] (e0) -- (c1);
\draw [->, red, thick] (d0) -- (c1);
\draw [->, thick] (b2) -- (c2);
\draw [->, thick] (e1) -- (c2);
\draw [->, thick] (d1) -- (c2);
\end{tikzpicture}
\end{center}
\caption{McKay quiver $Q_{\rho_{\mathrm{NAT}}}(\mathrm{B}\mathbf{D}_{2(4)}\times\mathbf{C}_3)$.}\label{fig:MQ.D2q.Cm}
\end{figure}
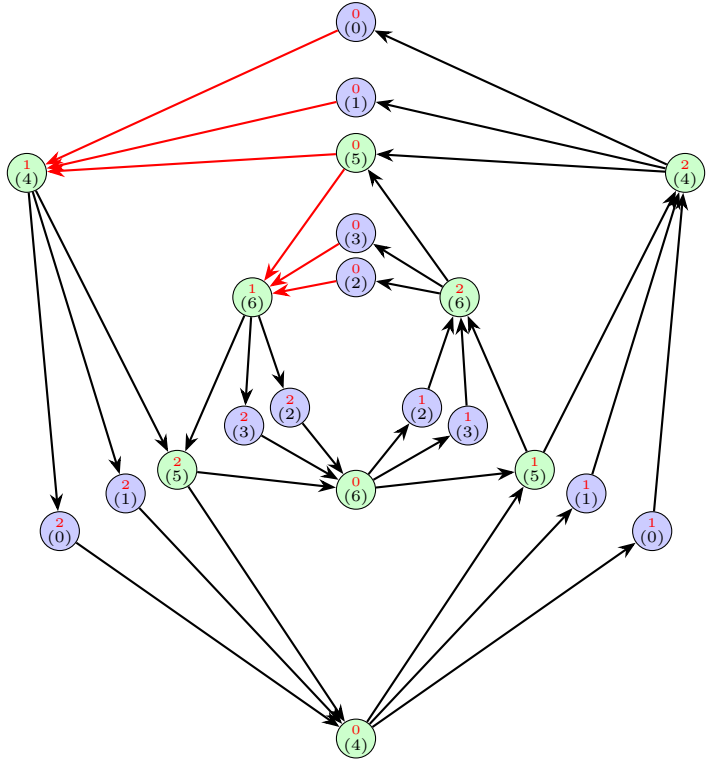

\subsubsection{McKay quiver $Q_{\rho_{\mathrm{NAT}}}(D_{8(3)}\times\mathbf{C}_5)$}

By Theorem~\ref{thm:sfs} the group $D_{8(3)}\times\mathbf{C}_5$ is the small dihedral group
\begin{equation*}
\mathbb{D}_{13,3}\cong D_{8(3)}\times\mathbf{C}_5.
\end{equation*}
Remember from Subsection~\ref{irreducible-representations} that $D_{8(3)}$ has $8$ one-dimensional irreducible representations $\alpha_i$ with $i=0,\dots,7$, and $4$ two-dimensional irreducible representations $\varrho_{t,s}$ with $t=1,2$ and $s=0,1$. Hence, $D_{8(3)}\times\mathbf{C}_5$ has $60$ irreducible representations: $40$ one-dimensional representations
\begin{equation*}
\alpha_i\times\beta_j,\quad 0\leq i<8,\ 0\leq j<5,
\end{equation*}
and $20$ two-dimensional irreducible representations
\begin{equation*}
\varrho_{t,s}\times\beta_j,\quad t=1,2,\ s=0,1,\ 0\leq j<5.
\end{equation*}
Hence, there are $60$ vertices in the McKay quiver $Q_{\rho_{\mathrm{NAT}}}(D_{8(3)}\times\mathbf{C}_5)$.
Figure~\ref{fig:MQ.D2k2r.Cm} shows the McKay quiver $Q_{\rho_{\mathrm{NAT}}}(D_{8(3)}\times\mathbf{C}_5)$.
We are labeling the \textcolor{blue}{blue} vertices corresponding to the one-dimensional representations $\alpha_i\times\beta_j$ with $\overset{\textcolor{red}{j}}{(i)}$, and the \textcolor{green}{green} vertices corresponding to the two-dimensional representions $\varrho_{t,s}\times\beta_j$ with $\overset{\textcolor{red}{j}}{(t,s)}$.

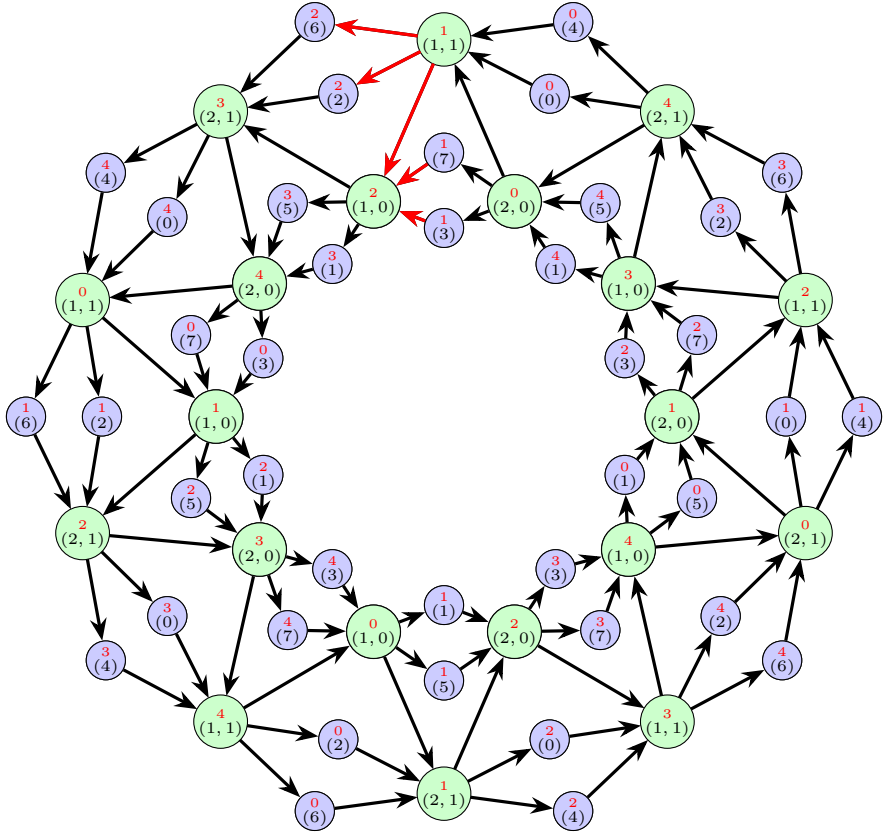
\begin{figure}[H]
\begin{center}
\begin{tikzpicture}[>=Stealth,auto=right]
\tikzmath{
\n=10; 
\s=\n-1;
\r=\n/3-1;
}
\begin{scope}[every node/.style={circle,fill=green!20,inner sep=0pt,font=\tiny,draw}]
\foreach \i in {0,1,...,\s} {\node (a\i) at ({5*cos((360/\n)*\i+(360/(2*\n))},{5*sin((360/\n)*\i+(360/(2*\n))}) {$\overset{a}{(\i,0)}$};};
\foreach \i in {0,1,...,\s} {\node (d\i) at ({3*cos((360/\n)*\i},{3*sin((360/\n)*\i}) {$\overset{d}{(0,\i)}$};};
\end{scope}
\begin{scope}[every node/.style={circle,fill=blue!20,inner sep=0pt,font=\tiny,draw}]
\foreach \i in {0,1,...,\s} {\node (b\i) at ({5.5*cos((360/\n)*\i},{5.5*sin((360/\n)*\i}) {$\overset{0}{(0)}$};};
\foreach \i in {0,1,...,\s} {\node (c\i) at ({4.5*cos((360/\n)*\i},{4.5*sin((360/\n)*\i}) {$\overset{0}{(0)}$};};
\foreach \i in {0,1,...,\s} {\node (e\i) at ({3.5*cos((360/\n)*\i+(360/(2*\n))},{3.5*sin((360/\n)*\i+(360/(2*\n))}) {$\overset{0}{(0)}$};};
\foreach \i in {0,1,...,\s} {\node (f\i) at ({2.5*cos((360/\n)*\i+(360/(2*\n))},{2.5*sin((360/\n)*\i+(360/(2*\n))}) {$\overset{0}{(0)}$};};
\foreach \i [evaluate=\i as \j using {int(mod(\i+1,\n))}] in {0,...,\s} \draw [->, very thick] (a\i) -- (b\j);
\foreach \i [evaluate=\i as \j using {int(mod(\i+1,\n))}] in {0,...,\s} \draw [->, very thick] (a\i) -- (c\j);
\foreach \i [evaluate=\i as \j using {int(mod(\i,\n))}] in {0,...,\s} \draw [->, very thick] (b\i) -- (a\j);
\foreach \i [evaluate=\i as \j using {int(mod(\i,\n))}] in {0,...,\s} \draw [->, very thick] (c\i) -- (a\j);
\foreach \i [evaluate=\i as \j using {int(mod(\i,\n))}] in {0,...,\s} \draw [->, very thick] (d\i) -- (e\j);
\foreach \i [evaluate=\i as \j using {int(mod(\i,\n))}] in {0,...,\s} \draw [->, very thick] (d\i) -- (f\j);
\foreach \i [evaluate=\i as \j using {int(mod(\i+1,\n))}] in {0,...,\s} \draw [->, very thick] (e\i) -- (d\j);
\foreach \i [evaluate=\i as \j using {int(mod(\i+1,\n))}] in {0,...,\s} \draw [->, very thick] (f\i) -- (d\j);
\foreach \i [evaluate=\i as \j using {int(mod(\i,\n))}] in {0,...,\s} \draw [->, very thick] (d\i) -- (a\j);
\foreach \i [evaluate=\i as \j using {int(mod(\i+1,\n))}] in {0,...,\s} \draw [->, very thick] (a\i) -- (d\j);
\end{scope}
\begin{scope}[every node/.style={circle,fill=blue!20,inner sep=0pt,font=\tiny,draw}]
\node at (b2) {$\overset{\textcolor{red}{0}}{(4)}$};
\node at (b0) {$\overset{\textcolor{red}{1}}{(4)}$};
\node at (b8) {$\overset{\textcolor{red}{2}}{(4)}$};
\node at (b6) {$\overset{\textcolor{red}{3}}{(4)}$};
\node at (b4) {$\overset{\textcolor{red}{4}}{(4)}$};
\node at (f9) {$\overset{\textcolor{red}{0}}{(1)}$};
\node at (f7) {$\overset{\textcolor{red}{1}}{(1)}$};
\node at (f5) {$\overset{\textcolor{red}{2}}{(1)}$};
\node at (f3) {$\overset{\textcolor{red}{3}}{(1)}$};
\node at (f1) {$\overset{\textcolor{red}{4}}{(1)}$};
\node at (b7) {$\overset{\textcolor{red}{0}}{(6)}$};
\node at (b5) {$\overset{\textcolor{red}{1}}{(6)}$};
\node at (b3) {$\overset{\textcolor{red}{2}}{(6)}$};
\node at (b1) {$\overset{\textcolor{red}{3}}{(6)}$};
\node at (b9) {$\overset{\textcolor{red}{4}}{(6)}$};
\node at (f4) {$\overset{\textcolor{red}{0}}{(3)}$};
\node at (f2) {$\overset{\textcolor{red}{1}}{(3)}$};
\node at (f0) {$\overset{\textcolor{red}{2}}{(3)}$};
\node at (f8) {$\overset{\textcolor{red}{3}}{(3)}$};
\node at (f6) {$\overset{\textcolor{red}{4}}{(3)}$};
\node at (c2) {$\overset{\textcolor{red}{0}}{(0)}$};
\node at (c0) {$\overset{\textcolor{red}{1}}{(0)}$};
\node at (c8) {$\overset{\textcolor{red}{2}}{(0)}$};
\node at (c6) {$\overset{\textcolor{red}{3}}{(0)}$};
\node at (c4) {$\overset{\textcolor{red}{4}}{(0)}$};
\node at (e9) {$\overset{\textcolor{red}{0}}{(5)}$};
\node at (e7) {$\overset{\textcolor{red}{1}}{(5)}$};
\node at (e5) {$\overset{\textcolor{red}{2}}{(5)}$};
\node at (e3) {$\overset{\textcolor{red}{3}}{(5)}$};
\node at (e1) {$\overset{\textcolor{red}{4}}{(5)}$};
\node at (c7) {$\overset{\textcolor{red}{0}}{(2)}$};
\node at (c5) {$\overset{\textcolor{red}{1}}{(2)}$};
\node at (c3) {$\overset{\textcolor{red}{2}}{(2)}$};
\node at (c1) {$\overset{\textcolor{red}{3}}{(2)}$};
\node at (c9) {$\overset{\textcolor{red}{4}}{(2)}$};
\node at (e4) {$\overset{\textcolor{red}{0}}{(7)}$};
\node at (e2) {$\overset{\textcolor{red}{1}}{(7)}$};
\node at (e0) {$\overset{\textcolor{red}{2}}{(7)}$};
\node at (e8) {$\overset{\textcolor{red}{3}}{(7)}$};
\node at (e6) {$\overset{\textcolor{red}{4}}{(7)}$};
\end{scope}
\begin{scope}[every node/.style={circle,fill=green!20,inner sep=0pt,font=\tiny,draw}]
\node at (d7) {$\overset{\textcolor{red}{0}}{(1,0)}$};
\node at (d5) {$\overset{\textcolor{red}{1}}{(1,0)}$};
\node at (d3) {$\overset{\textcolor{red}{2}}{(1,0)}$};
\node at (d1) {$\overset{\textcolor{red}{3}}{(1,0)}$};
\node at (d9) {$\overset{\textcolor{red}{4}}{(1,0)}$};
\node at (a4) {$\overset{\textcolor{red}{0}}{(1,1)}$};
\node at (a2) {$\overset{\textcolor{red}{1}}{(1,1)}$};
\node at (a0) {$\overset{\textcolor{red}{2}}{(1,1)}$};
\node at (a8) {$\overset{\textcolor{red}{3}}{(1,1)}$};
\node at (a6) {$\overset{\textcolor{red}{4}}{(1,1)}$};
\node at (d2) {$\overset{\textcolor{red}{0}}{(2,0)}$};
\node at (d0) {$\overset{\textcolor{red}{1}}{(2,0)}$};
\node at (d8) {$\overset{\textcolor{red}{2}}{(2,0)}$};
\node at (d6) {$\overset{\textcolor{red}{3}}{(2,0)}$};
\node at (d4) {$\overset{\textcolor{red}{4}}{(2,0)}$};
\node at (a9) {$\overset{\textcolor{red}{0}}{(2,1)}$};
\node at (a7) {$\overset{\textcolor{red}{1}}{(2,1)}$};
\node at (a5) {$\overset{\textcolor{red}{2}}{(2,1)}$};
\node at (a3) {$\overset{\textcolor{red}{3}}{(2,1)}$};
\node at (a1) {$\overset{\textcolor{red}{4}}{(2,1)}$};
\end{scope}
\draw [->, red, very thick] (a2) -- (b3);
\draw [->, red, very thick] (a2) -- (c3);
\draw [->, red, very thick] (a2) -- (d3);
\draw [->, red, very thick] (e2) -- (d3);
\draw [->, red, very thick] (f2) -- (d3);
\end{tikzpicture}
\end{center}
\caption{McKay quiver $Q_{\rho_{\mathrm{NAT}}}(D_{8(3)}\times\mathbf{C}_5)$.}\label{fig:MQ.D2k2r.Cm}
\end{figure}

\subsubsection{McKay quiver $Q_{\rho_{\mathrm{NAT}}}(B\mathbf{T}\times\mathbf{C}_5)$}\label{sssec:BTC5}

By Theorem~\ref{thm:sfs} the group $B\mathbf{T}\times\mathbf{C}_5$ is the small tetrahedral group
\begin{equation*}
\mathbb{T}_{5}\cong B\mathbf{T}\times\mathbf{C}_5.
\end{equation*}
The group $B\mathbf{T}$ has $7$ irreducible representations: $3$ one-dimensional representations $\rho_0$, $\rho_1$ and $\rho_2$; $3$ two-dimensional representations $\rho_3$, $\rho_4$ and $\rho_5$ and $1$ three-dimensional representation $\rho_6$. Hence, $B\mathbf{T}\times\mathbf{C}_5$ has $35$ irreducible representations given by
\begin{equation*}
\rho_i\times\beta_j,\quad 0\leq i<7,\ 0\leq j<5.
\end{equation*}
Thus, there are $35$ vertices in the McKay quiver $Q_{\rho_{\mathrm{NAT}}}(B\mathbf{T}\times\mathbf{C}_5)$.
Figure~\ref{fig:MQ.BT.C5} shows the McKay quiver $Q_{\rho_{\mathrm{NAT}}}(B\mathbf{T}\times\mathbf{C}_5)$.
We are labeling the vertices corresponding to the representations $\rho_i\times\beta_j$ with $\overset{\textcolor{red}{j}}{(i)}$, the vertices corresponding to the one-dimensional representions are \textcolor{blue}{blue} and the vertices corresponding to the two-dimensional representions are \textcolor{green}{green} and the vertices corresponding to the three-dimensional representions are \textcolor{red}{red}. Note that this quiver is not planar.

\begin{figure}[H]
\begin{center}
\begin{tikzpicture}[>=Stealth,auto=right,every node/.style={circle,fill=blue!20,inner sep=0pt,font=\tiny,draw},twodim/.style={circle,fill=green!20,inner sep=0pt,font=\tiny,draw}]
 \tikzmath{
\n=10; 
\s=\n-1;
\r=\n/2-1;
}
\begin{scope}
\foreach \i in {0,1,...,\s} {\node (a\i) at ({4.5*cos(1*(360/\n)*\i)},{4.5*sin(1*(360/\n)*\i)}) {$\overset{a}{(\i)}$};};
\foreach \i [evaluate=\i as \j using {int(mod(\i+1,\n))}] in {0,...,\s} \draw[->, thick] (a\i) -- (a\j);
\foreach \i [evaluate=\i as \j using {int(mod(2*\i,\n))},evaluate=\i as \k using {int(mod(2*\i,\n/2))}] in {0,...,\s} \node at (a\j) {$\overset{\textcolor{red}{\k}}{(0)}$};
\foreach \i [evaluate=\i as \j using {int(mod(2*\i+1,\n))},evaluate=\i as \k using {int(mod(2*\i+1,\n/2))}] in {0,...,\s} \node[twodim] at (a\j) {$\overset{\textcolor{red}{\k}}{(3)}$};
\end{scope}
\begin{scope}[every node/.style={circle,fill=red!20,inner sep=0pt,font=\tiny,draw}]
\foreach \i in {0,1,...,\r} {\node (b\i) at ({3.5*cos(1*(360/(\n/2))*\i)},{3.5*sin(1*(360/(\n/2))*\i)}) {$\overset{b}{(\i)}$};};
\foreach \i [evaluate=\i as \j using {int(mod(2*\i,\n/2))}] in {0,...,\r} \node at (b\i) {$\overset{\textcolor{red}{\j}}{(6)}$};
\end{scope}
\begin{scope}
\foreach \i in {0,1,...,\s} {\node (c\i) at ({2.8*cos(1*(360/\n)*\i)},{2.8*sin(1*(360/\n)*\i)}) {$\overset{c}{(\i)}$};};
\foreach \i [evaluate=\i as \j using {int(mod(\i+1,\n))}] in {0,...,\s} \draw[->, thick] (c\i) -- (c\j);
\foreach \i [evaluate=\i as \j using {int(mod(2*\i,\n))},evaluate=\i as \k using {int(mod(2*\i,\n/2))}] in {0,...,\s} \node at (c\j) {$\overset{\textcolor{red}{\k}}{(1)}$};
\foreach \i [evaluate=\i as \j using {int(mod(2*\i+1,\n))},evaluate=\i as \k using {int(mod(2*\i+1,\n/2))}] in {0,...,\s} \node[twodim] at (c\j) {$\overset{\textcolor{red}{\k}}{(4)}$};
\end{scope}
\begin{scope}
\foreach \i in {0,1,...,\s} {\node (d\i) at ({2*cos(1*(360/\n)*\i)},{2*sin(1*(360/\n)*\i)}) {$\overset{d}{(\i)}$};};
\foreach \i [evaluate=\i as \j using {int(mod(\i+1,\n))}] in {0,...,\s} \draw[->, thick] (d\i) -- (d\j);
\foreach \i [evaluate=\i as \j using {int(mod(2*\i,\n))},evaluate=\i as \k using {int(mod(2*\i,\n/2))}] in {0,...,\s} \node at (d\j) {$\overset{\textcolor{red}{\k}}{(2)}$};
\foreach \i [evaluate=\i as \j using {int(mod(2*\i+1,\n))},evaluate=\i as \k using {int(mod(2*\i+1,\n/2))}] in {0,...,\s} \node[twodim] at (d\j) {$\overset{\textcolor{red}{\k}}{(5)}$};
\end{scope}
\foreach \i [evaluate=\i as \j using {int(mod(2*\i,\n))}] in {0,...,\r} {\pgfmathtruncatemacro{\k}{\j + 1};\draw [->, blue, thick] (b\i) -- (a\k);\draw [->, blue, thick] (b\i) -- (c\k);\draw [->, blue, thick] (b\i) -- (d\k);};
\foreach \i [evaluate=\i as \j using {int(mod(2*\i+1,\n))},evaluate=\i as \k using {int(mod(\i+1,\n/2))}] in {0,...,\r} {\draw [->, blue, thick] (a\j) -- (b\k);\draw [->, blue, thick] (c\j) -- (b\k);\draw [->, blue, thick] (d\j) -- (b\k);};
\draw [->, red, thick] (a2) -- (a3);
\draw [->, red, thick] (b1) -- (a3);
\draw [->, red, thick] (c2) -- (c3);
\draw [->, red, thick] (b1) -- (c3);
\draw [->, red, thick] (d2) -- (d3);
\draw [->, red, thick] (b1) -- (d3);
\end{tikzpicture}
\end{center}
\caption{McKay quiver $Q_{\rho_{\mathrm{NAT}}}(B\mathbf{T}\times\mathbf{C}_5)$.}\label{fig:MQ.BT.C5}
\end{figure}

\subsubsection{McKay quiver $Q_{\rho_{\mathrm{NAT}}}(P'_{8\cdot 3^2}\times\mathbf{C}_5)$}

By Theorem~\ref{thm:sfs} the group $P'_{8\cdot 3^2}\times\mathbf{C}_5$ is the small tetrahedral group
\begin{equation*}
\mathbb{T}_{15}\cong P'_{8\cdot 3^2}\times\mathbf{C}_5.
\end{equation*}
Remember from Subsection~\ref{irreducible-representations.P} that the group $P'_{8\cdot 3^2}$ has $9$ one-dimensional irreducible representations $\alpha_i$ and $9$ two-dimensional irreducible representations $\varrho_i$ with $0\leq i\leq 8$, and $3$ three-dimensional irreducible representations $\varsigma_{s}$ with $s=0,1,2$.
Hence, $P'_{8\cdot 3^2}\times\mathbf{C}_5$ has $105$ irreducible representations: $45$ one-dimensional representations
\begin{equation*}
\alpha_i\times\beta_j,\quad 0\leq i<9,\ 0\leq j<5,
\end{equation*}
$45$ two-dimensional irreducible representations
\begin{equation*}
\varrho_{i}\times\beta_j,\quad 0\leq i<9,\ 0\leq j<5,
\end{equation*}
and $15$ three-dimensional irreducible representations
\begin{equation*}
\varsigma_{s}\times\beta_j,\quad s=0,1,2,\ 0\leq j<5.
\end{equation*}
Hence, there are $105$ vertices in the McKay quiver $Q_{\rho_{\mathrm{NAT}}}(P'_{8\cdot 3^2}\times\mathbf{C}_5)$.
Figure~\ref{fig:MQ.Pk.Cm} shows the McKay quiver $Q_{\rho_{\mathrm{NAT}}}(P'_{8\cdot 3^2}\times\mathbf{C}_5)$.
We are labeling the \textcolor{blue}{blue} vertices corresponding to the one-dimensional representations $\alpha_i\times\beta_j$ with $\overset{\textcolor{red}{j}}{(i)}$, the \textcolor{green}{green} vertices corresponding to the two-dimensional representions $\varrho_{i}\times\beta_j$ with $\overset{\textcolor{red}{j}}{(i)}$, and the \textcolor{red}{red} vertices corresponding to the three-dimensional representions $\varsigma_{s}\times\beta_j$ with $\overset{\textcolor{red}{j}}{(s)}$. Note that this quiver is not planar.

\begin{figure}[H]
\begin{center}
\begin{tikzpicture}[scale=.9,>=Stealth,auto=right]
\tikzmath{
\n=45; 
\s=\n-1;
\r=\n/3-1;
\t=\n/3;
}
\begin{scope}[name prefix = 1-,every node/.style={circle,fill=blue!20,inner sep=0pt,font=\tiny,draw}]
\foreach \i [evaluate=\i as \k using {int(mod(\i,9))},evaluate=\i as \j using {int(mod(\i,5))}] in {0,1,...,\s} {\node (a\i) at ({6*cos(-23*(360/\n)*\i)},{6*sin(-23*(360/\n)*\i)}) {$\overset{\textcolor{red}{\j}}{(\k)}$};};
\end{scope}
\begin{scope}[name prefix = 2-,every node/.style={circle,fill=green!20,inner sep=0pt,font=\tiny,draw}]
\foreach \i [evaluate=\i as \k using {int(mod(\i,9))},evaluate=\i as \j using {int(mod(\i,5))}] in {0,1,...,\s} {\node (a\i) at ({4*cos(-23*(360/\n)*\i+(360/2*\n))},{4*sin(-23*(360/\n)*\i+(360/2*\n))}) {$\overset{\textcolor{red}{\j}}{(\k)}$};};
\end{scope}
\begin{scope}[name prefix = 3-,every node/.style={circle,fill=red!20,inner sep=0pt,font=\tiny,draw}]
\foreach \i [evaluate=\i as \k using {int(mod(\i,3))},evaluate=\i as \j using {int(mod(\i,5))}] in {0,1,...,\r} {\node (a\i) at ({3*cos(4*(360/\t)*\i-3*(360/(2*\n)))},{3*sin(4*(360/\t)*\i-3*(360/(2*\n)))}) {$\overset{\textcolor{red}{\j}}{(\k)}$};};
\end{scope}
\foreach \i [evaluate=\i as \k using {int(mod(\i+1,\n))}] in {0,1,...,\s} {\draw [->, thick, blue] (1-a\i) -- (2-a\k);};
\foreach \i [evaluate=\i as \k using {int(mod(\i+1,\n))}] in {0,1,...,\s} {\draw [->, thick] (2-a\i) -- (1-a\k);};
\foreach \i [evaluate=\i as \k using {int(mod(\i+1,\n))},evaluate=\i as \j using {int(mod(\i+1,\r+1))}] in {0,1,...,\s} {\draw [->, thick] (2-a\i) -- (3-a\j);};
\foreach \i [evaluate=\i as \k using {int(mod(\i+1,\n))},evaluate=\i as \j using {int(mod(\i+16,\n))},evaluate=\i as \l using {int(mod(\i+31,\n))}] in {0,1,...,\r} {\draw [->, thick, red] (3-a\i) -- (2-a\k);\draw [->, thick, red] (3-a\i) -- (2-a\j);\draw [->, thick, red] (3-a\i) -- (2-a\l);};
\end{tikzpicture}
\end{center}
\caption{McKay quiver $Q_{\rho_{\mathrm{NAT}}}(P'_{8\cdot 3^2}\times\mathbf{C}_5)$.}\label{fig:MQ.Pk.Cm}
\end{figure}

\begin{remark}\label{rem:MQ.Pk.Cm.RW}
Rename the irreducible representations of the group $P'_{8\cdot 3^k}\times\mathbf{C}_l$ as follows:
For every pair $(i,j)$ with $0\leq i< 3^k$ and $0\leq j< l$,
let $p=i-j\mod l$, that is $0\leq p< l$, and $q=3^kp+i \mod 3^kl$, that is $0\leq q< 3^kl$ and set
\begin{align*}
\tilde{\alpha}_{q}&=\alpha_i\times\beta_j,\quad 0\leq i< 3^k,\ 0\leq j< l,\\
\tilde{\varrho}_{q}&=\varrho_{i}\times\beta_j,\quad 0\leq i< 3^k,\ 0\leq j<l,
\end{align*}
for the one and two-dimensional irreducible representations.
For every pair $(s,j)$ with $0\leq s< 3^{k-1}$ and $0\leq j< l$,
let $p=s-j\mod l$, that is $0\leq p< l$, and $q=3^kp+i \mod 3^{k-1}l$, that is $0\leq q< 3^{k-1}l$ and set
\begin{equation*}
\tilde{\varsigma}_{q}=\varsigma_{s}\times\beta_j,\quad 0\leq s<3^{k-1},\ 0\leq j< l,
\end{equation*}
for the three-dimensional representations. With this renaming the McKay quiver $Q_{\rho_{\mathrm{NAT}}}(P'_{8\cdot 3^k}\times\mathbf{C}_l)$ is given by
\begin{equation}\label{mq:P83k.Cm}
\begin{tikzcd}
\tilde{\alpha}_{q} \arrow[r]  & \tilde{\varrho}_{q+1}\\
\end{tikzcd}\qquad\qquad
\begin{tikzcd}
   & \tilde{\alpha}_{q+1} \\
\tilde{\varrho}_{q}\arrow[rd] \arrow[ru] &  \\
& \tilde{\varsigma}_{q+1}\\
\end{tikzcd}\qquad\qquad
\begin{tikzcd}
   & \tilde{\varrho}_{q+1}\\
\tilde{\varsigma}_{q}\arrow[rd] \arrow[r] \arrow[ru] & \tilde{\varrho}_{3^{k-1}l+q+1} \\
& \tilde{\varrho}_{2\cdot3^{k-1}l+q+1}
\end{tikzcd}
\end{equation}
When $k=2$ and $l=1$ we get the McKay quiver $Q_{\rho_{\mathrm{Nat}}}(P'_{8\cdot 3^2})$ given in Figure~\ref{fig:MQ:P3.2} by \eqref{mq:P83k}. When
$k=2$ and $l=5$ we get the McKay quiver of $Q_{\rho_{\mathrm{NAT}}}(P'_{8\cdot 3^2}\times\mathbf{C}_5)$ given in
Figure~\ref{fig:MQ.Pk.Cm} by Proposition~\ref{prop:MQ.GxCm}.
\end{remark}

\subsubsection{McKay quiver $Q_{\rho_{\mathrm{NAT}}}(B\mathbf{O}\times\mathbf{C}_5)$}\label{sssec:BOC5}

By Theorem~\ref{thm:sfs} the group $B\mathbf{O}\times\mathbf{C}_5$ is the small octahedral group
\begin{equation*}
\mathbb{O}_{5}\cong B\mathbf{O}\times\mathbf{C}_5.
\end{equation*}
The group $B\mathbf{O}$ has $8$ irreducible representations: $2$ one-dimensional representations $\rho_0$ and $\rho_1$; $3$ two-dimensional representations $\rho_2$, $\rho_3$ and $\rho_4$; $2$ three-dimensional representation $\rho_5$ and $\rho_6$, and $1$ four-dimensional representation $\rho_7$. Hence, $B\mathbf{O}\times\mathbf{C}_5$ has $40$ irreducible representations given by
\begin{equation*}
\rho_i\times\beta_j,\quad 0\leq i<8,\ 0\leq j<5.
\end{equation*}
Thus, there are $40$ vertices in the McKay quiver $Q_{\rho_{\mathrm{NAT}}}(B\mathbf{O}\times\mathbf{C}_5)$.
Figure~\ref{fig:MQ.BO.C5} shows the McKay quiver $Q_{\rho_{\mathrm{NAT}}}(B\mathbf{O}\times\mathbf{C}_5)$.
We are labeling the vertices corresponding to the representations $\rho_i\times\beta_j$ with $\overset{\textcolor{red}{j}}{(i)}$, the vertices corresponding to the one-dimensional representions are \textcolor{blue}{blue} and the vertices corresponding to the two-dimensional representions are \textcolor{green}{green}, the vertices corresponding to the three-dimensional representions are \textcolor{red}{red} and the vertices corresponding to the four-dimensional representions are \textcolor{yellow}{yellow}.

\begin{figure}[H]
\begin{center}
\begin{tikzpicture}[scale=.8,>=Stealth,auto=right,every node/.style={circle,fill=blue!20,inner sep=0pt,font=\tiny,draw},twodim/.style={circle,fill=green!20,inner sep=0pt,font=\tiny,draw},threedim/.style={circle,fill=red!20,inner sep=0pt,font=\tiny,draw}]
 \tikzmath{
\n=10; 
\s=\n-1;
\r=\n/2-1;
}
\begin{scope}
\foreach \i in {0,1,...,\s} {\node (a\i) at ({4.5*cos(1*(360/\n)*\i)},{4.5*sin(1*(360/\n)*\i)}) {$\overset{a}{(\i)}$};};
\foreach \i [evaluate=\i as \j using {int(mod(\i+1,\n))}] in {0,...,\s} \draw[->, thick] (a\i) -- (a\j);
\foreach \i [evaluate=\i as \j using {int(mod(2*\i,\n))},evaluate=\i as \k using {int(mod(2*\i,\n/2))}] in {0,...,\s} \node at (a\j) {$\overset{\textcolor{red}{\k}}{(0)}$};
\foreach \i [evaluate=\i as \j using {int(mod(2*\i+1,\n))},evaluate=\i as \k using {int(mod(2*\i+1,\n/2))}] in {0,...,\s} \node[twodim] at (a\j) {$\overset{\textcolor{red}{\k}}{(2)}$};
\end{scope}
\begin{scope}[fourdim/.style={circle,fill=yellow!20,inner sep=0pt,font=\tiny,draw}]
\foreach \i in {0,1,...,\s} {\node (c\i) at ({3.7*cos(1*(360/\n)*\i)},{3.7*sin(1*(360/\n)*\i)}) {$\overset{c}{(\i)}$};};
\foreach \i [evaluate=\i as \j using {int(mod(\i+1,\n))}] in {0,...,\s} \draw[->, thick] (c\i) -- (c\j);
\foreach \i [evaluate=\i as \j using {int(mod(2*\i,\n))},evaluate=\i as \k using {int(mod(2*\i,\n/2))}] in {0,...,\s} \node[threedim] at (c\j) {$\overset{\textcolor{red}{\k}}{(4)}$};
\foreach \i [evaluate=\i as \j using {int(mod(2*\i+1,\n))},evaluate=\i as \k using {int(mod(2*\i+1,\n/2))}] in {0,...,\s} \node[fourdim] at (c\j) {$\overset{\textcolor{red}{\k}}{(7)}$};
\end{scope}
\begin{scope}
\foreach \i in {0,1,...,\s} {\node (d\i) at ({2*cos(1*(360/\n)*\i)},{2*sin(1*(360/\n)*\i)}) {$\overset{d}{(\i)}$};};
\foreach \i [evaluate=\i as \j using {int(mod(\i+1,\n))}] in {0,...,\s} \draw[->, thick] (d\i) -- (d\j);
\foreach \i [evaluate=\i as \j using {int(mod(2*\i,\n))},evaluate=\i as \k using {int(mod(2*\i,\n/2))}] in {0,...,\s} \node[threedim] at (d\j) {$\overset{\textcolor{red}{\k}}{(6)}$};
\foreach \i [evaluate=\i as \j using {int(mod(2*\i+1,\n))},evaluate=\i as \k using {int(mod(2*\i+1,\n/2))}] in {0,...,\s} \node[twodim] at (d\j) {$\overset{\textcolor{red}{\k}}{(3)}$};
\end{scope}
\begin{scope}
\foreach \i in {0,1,...,\r} {\node (b\i) at ({1*cos(1*(360/(\n/2))*\i)},{1*sin(1*(360/(\n/2))*\i)}) {$\overset{b}{(\i)}$};};
\foreach \i [evaluate=\i as \j using {int(mod(2*\i,\n/2))}] in {0,...,\r} \node at (b\i) {$\overset{\textcolor{red}{\j}}{(1)}$};
\end{scope}
\begin{scope}[every node/.style={circle,fill=green!20,inner sep=0pt,font=\tiny,draw}]
\foreach \i in {0,1,...,\r} {\node (e\i) at ({2.9*cos(1*(360/(\n/2))*\i)},{2.9*sin(1*(360/(\n/2))*\i)}) {$\overset{b}{(\i)}$};};
\foreach \i [evaluate=\i as \j using {int(mod(2*\i,\n/2))}] in {0,...,\r} \node at (e\i) {$\overset{\textcolor{red}{\j}}{(5)}$};
\end{scope}
\foreach \i [evaluate=\i as \j using {int(mod(2*\i,\n))}] in {0,...,\r} {\pgfmathtruncatemacro{\k}{\j + 1};\draw [->, thick] (c\j) -- (a\k);\draw [->, thick] (d\j) -- (c\k);};
\foreach \i [evaluate=\i as \j using {int(mod(2*\i+1,\n))},evaluate=\i as \k using {int(mod(\j+1,\n))}] in {0,...,\r} {\draw [->, thick] (a\j) -- (c\k);\draw [->, thick] (c\j) -- (d\k);};
\foreach \i [evaluate=\i as \j using {int(mod(2*\i+1,\n))},evaluate=\i as \k using {int(mod(\i+1,\n/2))}] in {0,...,\r} {\draw [->, thick] (d\j) -- (b\k);};
\foreach \i [evaluate=\i as \j using {int(mod(2*\i+1,\n))},evaluate=\i as \k using {int(mod(\i+1,\n/2))}] in {0,...,\r} {\draw [->, thick] (b\i) -- (d\j);};
\foreach \i [evaluate=\i as \j using {int(mod(2*\i+1,\n))},evaluate=\i as \k using {int(mod(\i+1,\n/2))}] in {0,...,\r} {\draw [->, thick] (e\i) -- (c\j);};
\foreach \i [evaluate=\i as \j using {int(mod(2*\i+1,\n))},evaluate=\i as \k using {int(mod(\i+1,\n/2))}] in {0,...,\r} {\draw [->, thick] (c\j) -- (e\k);};
\draw [->, red, thick] (a1) -- (a2);
\draw [->, red, thick] (a1) -- (c2);
\draw [->, red, thick] (c1) -- (c2);
\draw [->, red, thick] (c1) -- (e1);
\draw [->, red, thick] (c1) -- (d2);
\draw [->, red, thick] (d1) -- (d2);
\draw [->, red, thick] (d1) -- (b1);
\end{tikzpicture}
\end{center}
\caption{McKay quiver $Q_{\rho_{\mathrm{NAT}}}(B\mathbf{O}\times\mathbf{C}_5)$.}\label{fig:MQ.BO.C5}
\end{figure}

\subsubsection{McKay quiver $Q_{\rho_{\mathrm{NAT}}}(B\mathbf{I}\times\mathbf{C}_7)$}\label{sssec:BIC7}

By Theorem~\ref{thm:sfs} the group $B\mathbf{I}\times\mathbf{C}_7$ is the small icosahedral group
\begin{equation*}
\mathbb{I}_{7}\cong B\mathbf{I}\times\mathbf{C}_7.
\end{equation*}
The group $B\mathbf{I}$ has $9$ irreducible representations: $1$ one-dimensional representations $\rho_0$; $2$ two-dimensional representations $\rho_1$ and $\rho_2$; $2$ three-dimensional representation $\rho_3$ and $\rho_4$; $2$ four-dimensional representation $\rho_5$ and $\rho_6$; $1$ five-dimensional representation $\rho_7$ and $1$ six-dimensional representation $\rho_8$. Hence, $B\mathbf{I}\times\mathbf{C}_7$ has $63$ irreducible representations given by
\begin{equation*}
\rho_i\times\beta_j,\quad 0\leq i<9,\ 0\leq j<7.
\end{equation*}
Thus, there are $63$ vertices in the McKay quiver $Q_{\rho_{\mathrm{NAT}}}(B\mathbf{I}\times\mathbf{C}_7)$.
Figure~\ref{fig:MQ.BI.C7} shows the McKay quiver $Q_{\rho_{\mathrm{NAT}}}(B\mathbf{I}\times\mathbf{C}_7)$.
We are labeling the vertices corresponding to the representations $\rho_i\times\beta_j$ with $\overset{\textcolor{red}{j}}{(i)}$, the vertices corresponding to the one-dimensional representions are \textcolor{blue}{blue} and the vertices corresponding to the two-dimensional representions are \textcolor{green}{green}, the vertices corresponding to the three-dimensional representions are \textcolor{red}{red}, the vertices corresponding to the four-dimensional representions are \textcolor{yellow}{yellow}, the vertices corresponding to the five-dimen\-sio\-nal representions are \textcolor{cyan}{cyan} and the vertices corresponding to the six-dimensional representions are \textcolor{darkgray}{gray}.

\begin{figure}[H]
\begin{center}
\begin{tikzpicture}[scale=.9,>=Stealth,auto=right,every node/.style={circle,fill=blue!20,inner sep=0pt,font=\tiny,draw},twodim/.style={circle,fill=green!20,inner sep=0pt,font=\tiny,draw},threedim/.style={circle,fill=red!20,inner sep=0pt,font=\tiny,draw},fourdim/.style={circle,fill=yellow!20,inner sep=0pt,font=\tiny,draw},fivedim/.style={circle,fill=cyan!20,inner sep=0pt,font=\tiny,draw},sixdim/.style={circle,fill=darkgray!20,inner sep=0pt,font=\tiny,draw}]
 \tikzmath{
\n=14; 
\s=\n-1;
\r=\n/2-1;
}
\begin{scope}
\foreach \i in {0,1,...,\s} {\node (a\i) at ({5*cos(1*(360/\n)*\i)},{5*sin(1*(360/\n)*\i)}) {$\overset{a}{(1)}$};};
\foreach \i [evaluate=\i as \j using {int(mod(\i+1,\n))}] in {0,...,\s} \draw[->, thick] (a\i) -- (a\j);
\foreach \i [evaluate=\i as \j using {int(mod(2*\i,\n))},evaluate=\i as \k using {int(mod(2*\i,\n/2))}] in {0,...,\s} \node at (a\j) {$\overset{\textcolor{red}{\k}}{(0)}$};
\foreach \i [evaluate=\i as \j using {int(mod(2*\i+1,\n))},evaluate=\i as \k using {int(mod(2*\i+1,\n/2))}] in {0,...,\s} \node[twodim] at (a\j) {$\overset{\textcolor{red}{\k}}{(1)}$};
\end{scope}
\begin{scope}
\foreach \i in {0,1,...,\s} {\node (b\i) at ({4*cos(1*(360/\n)*\i)},{4*sin(1*(360/\n)*\i)}) {$\overset{c}{(1)}$};};
\foreach \i [evaluate=\i as \j using {int(mod(\i+1,\n))}] in {0,...,\s} \draw[->, thick] (b\i) -- (b\j);
\foreach \i [evaluate=\i as \j using {int(mod(2*\i,\n))},evaluate=\i as \k using {int(mod(2*\i,\n/2))}] in {0,...,\s} \node[threedim] at (b\j) {$\overset{\textcolor{red}{\k}}{(4)}$};
\foreach \i [evaluate=\i as \j using {int(mod(2*\i+1,\n))},evaluate=\i as \k using {int(mod(2*\i+1,\n/2))}] in {0,...,\s} \node[fourdim] at (b\j) {$\overset{\textcolor{red}{\k}}{(6)}$};
\end{scope}
\begin{scope}
\foreach \i in {0,1,...,\s} {\node (c\i) at ({3.3*cos(1*(360/\n)*\i)},{3.3*sin(1*(360/\n)*\i)}) {$\overset{d}{(1)}$};};
\foreach \i [evaluate=\i as \j using {int(mod(\i+1,\n))}] in {0,...,\s} \draw[->, thick] (c\i) -- (c\j);
\foreach \i [evaluate=\i as \j using {int(mod(2*\i,\n))},evaluate=\i as \k using {int(mod(2*\i,\n/2))}] in {0,...,\s} \node[fivedim] at (c\j) {$\overset{\textcolor{red}{\k}}{(7)}$};
\foreach \i [evaluate=\i as \j using {int(mod(2*\i+1,\n))},evaluate=\i as \k using {int(mod(2*\i+1,\n/2))}] in {0,...,\s} \node[sixdim] at (c\j) {$\overset{\textcolor{red}{\k}}{(8)}$};
\end{scope}
\begin{scope}
\foreach \i in {0,1,...,\s} {\node (d\i) at ({2*cos(1*(360/\n)*\i)},{2*sin(1*(360/\n)*\i)}) {$\overset{d}{(1)}$};};
\foreach \i [evaluate=\i as \j using {int(mod(\i+1,\n))}] in {0,...,\s} \draw[->, thick] (d\i) -- (d\j);
\foreach \i [evaluate=\i as \j using {int(mod(2*\i,\n))},evaluate=\i as \k using {int(mod(2*\i,\n/2))}] in {0,...,\s} \node[fourdim] at (d\j) {$\overset{\textcolor{red}{\k}}{(5)}$};
\foreach \i [evaluate=\i as \j using {int(mod(2*\i+1,\n))},evaluate=\i as \k using {int(mod(2*\i+1,\n/2))}] in {0,...,\s} \node[twodim] at (d\j) {$\overset{\textcolor{red}{\k}}{(2)}$};
\end{scope}
\begin{scope}
\foreach \i in {0,1,...,\r} {\node (e\i) at ({2.7*cos(1*(360/(\n/2))*\i)},{2.7*sin(1*(360/(\n/2))*\i)}) {$\overset{b}{(\i)}$};};
\foreach \i [evaluate=\i as \j using {int(mod(2*\i,\n/2))}] in {0,...,\r} \node[threedim] at (e\i) {$\overset{\textcolor{red}{\j}}{(3)}$};
\end{scope}
\foreach \i [evaluate=\i as \j using {int(mod(2*\i,\n))}] in {0,...,\r} {\pgfmathtruncatemacro{\k}{\j + 1};\draw [->, thick] (b\j) -- (a\k);\draw [->, thick] (c\j) -- (b\k);\draw [->, thick] (d\j) -- (c\k);};
\foreach \i [evaluate=\i as \j using {int(mod(2*\i+1,\n))},evaluate=\i as \k using {int(mod(\j+1,\n))}] in {0,...,\r} {\draw [->, thick] (a\j) -- (b\k);\draw [->, thick] (b\j) -- (c\k);\draw [->, thick] (c\j) -- (d\k);};
\foreach \i [evaluate=\i as \j using {int(mod(2*\i+1,\n))},evaluate=\i as \k using {int(mod(\i+1,\n/2))}] in {0,...,\r} {\draw [->, thick] (e\i) -- (c\j);};
\foreach \i [evaluate=\i as \j using {int(mod(2*\i+1,\n))},evaluate=\i as \k using {int(mod(\i+1,\n/2))}] in {0,...,\r} {\draw [->, thick] (c\j) -- (e\k);};
\draw [->, red, thick] (a1) -- (a2);
\draw [->, red, thick] (a1) -- (b2);
\draw [->, red, thick] (b1) -- (b2);
\draw [->, red, thick] (b1) -- (c2);
\draw [->, red, thick] (c1) -- (c2);
\draw [->, red, thick] (c1) -- (e1);
\draw [->, red, thick] (c1) -- (d2);
\draw [->, red, thick] (d1) -- (d2);
\end{tikzpicture}
\end{center}
\caption{McKay quiver $Q_{\rho_{\mathrm{NAT}}}(B\mathbf{I}\times\mathbf{C}_7)$.}\label{fig:MQ.BI.C7}
\end{figure}

\section{Comparison with the results of Auslander and Reiten in \cite{Auslander-Reiten:MQEDD}}\label{sec:AR}

As we mentioned in the introduction, the McKay quivers of small finite subgroups of $GL(2,\mathbb{C})$ were computed by Auslander and Reiten in \cite[Proposition~7]{Auslander-Reiten:MQEDD}.
In this section we make some remarks in order to compare our results with those of \cite{Auslander-Reiten:MQEDD} and to check that they are indeed the same.

In \cite[\S2]{Auslander-Reiten:MQEDD} the following constructions of quivers are introduced. For a tree $T$ and a positive integer $s$, the quivers $(T,s)$ and $[T,s]$ are defined as follows.
In both cases the vertices are $\{(v,i)\mid v\in T;\, 0\leq i<s\}$. For each edge \dynkin[labels={v,w}]{A}{2} in $T$ we have arrows $(v,i+1)\to(w,i)$ and $(w,i+1)\to(v,i)$ in $(T,s)$, with addition modulo $s$. Assigning $+$ or $-$ to the vertices of $T$ such that neighbors have opposite sign, we have for each \dynkin[labels={v,w},labels*={-,+}]{A}{2} in $T$, arrows $(v,1)\to(w,i)$ and $(w,i+1)\to(v,i)$ in $[T,s]$.

\begin{remark}\label{rem:copy.i}
If in the quiver $[T,s]$ we fix $0\leq i<s$ and for each \dynkin[labels={v,w},labels*={-,+}]{A}{2} in $T$ we consider the arrows $(v,i)\to(w,i)$, we obtain a subquiver of $[T,s]$ which is an oriented copy of the tree $T$, with the orientation of the arrows depending on the choice of signs given to the vertices of $T$. We denote this subquiver by $T^i$ and call it the \textit{$i$-th basic subquiver of $[T,s]$}. The arrows $(w,i+1)\to(v,i)$ in $[T,s]$ connect the $i+1$-th basic subquiver $T^{i+1}$ with the $i$-th basic subquiver $T^i$ of $[T,s]$.
\end{remark}

\begin{lemma}[{\cite[Lemma~6]{Auslander-Reiten:MQEDD}}]\label{lem:T.s.prop}
The quivers $(T,s)$ and $[T,s]$ have the following properties:
\begin{enumerate}[(a)]
 \item $(T,2m)$ is the disjoint union of two copies of $[T,m]$.
 \item If $T$ is connected, then $[T,m]$ is connected.
 \item If $m$ is odd, then $(T,m)$ and $[T,m]$ are isomorphic.\label{it:modd.iso}
\end{enumerate}
\end{lemma}

\begin{remark}
Recall from Subsection~\ref{sssec:Cn} that $\mathbf{C}_n$ has $n$ one-dimensional representations $\beta_j$, $0\leq j< n$. 
Rename these representations by setting
\begin{equation}\label{eq:rename}
 v_l=\beta_{n-l\pmod n}.
\end{equation}
With this renaming, the McKay quiver $Q_{\rho_{\mathrm{Nat}}}(\mathbb{C}_{n,q})$ given in Subsection~\ref{ssec:MQ.Cnq} has vertices $v_0,\dots,v_{n-1}$ and arrows $v_l\to v_{l-1}$ and $v_l\to v_{l-q}$ with addition modulo $n$, since by \eqref{eq:MQ.Cyclic} we have
\begin{equation*}
v_l=\beta_{n-l}\to\beta_{n-l+1}=\beta_{n-(l-1)}=v_{l-1}\ \text{and}\ v_l=\beta_{n-l}\to\beta_{n-l+q}=\beta_{n-(l-q)}=v_{l-q}.
\end{equation*}
This corresponds to the McKay quiver for $\mathbb{C}_{n,q}$ given in \cite[Proposition~7-(a)]{Auslander-Reiten:MQEDD}.
\end{remark}

\begin{remark}\label{rem:T.m}
With renaming \eqref{eq:rename}, the McKay quiver $Q_{\beta_{1}}(\mathbf{C}_m)$ becomes the McKay quiver $Q_{v_{m-1}}(\mathbf{C}_m)$ and from \eqref{eq:Qb1.Cm} it is given by one arrow going out from the vertex $v_l$ to the vertex $v_{l-1}$:
\begin{equation*}
v_{l}=\beta_{m-l}\rightarrow \beta_{(m-l)+1}=\beta_{m-(l-1)}=v_{l-1\pmod m}.
\end{equation*}
Denote by $\chi_j$ the character of $v_l$. Using the McKay quiver $Q_{v_{m-1}}(\mathbf{C}_m)$ to compute the McKay quivers $Q_{\rho_{\mathrm{NAT}}}(\Gamma\times\mathbf{C}_m)$, where
$\Gamma=\mathrm{B}\mathbf{D}_{2q},\mathrm{B}\mathbf{T},\mathrm{B}\mathbf{O},\mathrm{B}\mathbf{I}$, as explained in Subsection~\ref{ssec:MQ.GxCm}, one can see that the McKay quiver $Q_{\rho_{\mathrm{NAT}}}(\Gamma\times\mathbf{C}_m)$ given in Proposition~\ref{prop:MQ.GxCm} is precisely the quiver $(T,m)$ where $T=Q_{\rho_{\mathrm{Nat}}}(\Gamma)$ are the trees given
by the extended Dynkin diagrams of type $\tilde{A}$, $\tilde{D}$, $\tilde{E}$ presented in Table~\ref{tab:MK.graph}.
\end{remark}

\begin{remark}
In Theorem~\ref{thm:sfs} the dihedral groups $\mathbb{D}_{n,q}$ with $\gcd(n,q)=1$ and $m=n-q$ odd; the tetrahedral groups $\mathbb{T}_m$ with $\gcd(m,6)=1$; the octahedral groups $\mathbb{O}_m$ with $\gcd(m,6)=1$ and the icosahedral groups $\mathbb{I}_m$ with $\gcd(m,30)=1$, in the notation given in Subsection~\ref{ssec:fs.U2}, are of the form
\begin{equation*}
(\mathbf{C}_{2m},\mathbf{C}_{2m};\Gamma,\Gamma)_\phi,\qquad\text{with $\Gamma=\mathrm{B}\mathbf{D}_{2q},\mathrm{B}\mathbf{T},\mathrm{B}\mathbf{O},\mathrm{B}\mathbf{I}$ respectively.}
\end{equation*}
That is, we have $L=L_k$, $R=R_K$ and $\phi\colon L/L_K\to R/R_K$ is the isomorphism between trivial groups. Hence, the subgroup $H$ of $\mathbb{S}^1\times SU(2)$ is
$H=\mathbf{C}_{2m}\times\Gamma$ and the corresponding subgroup $G_H=\Phi(H)$ of $U(2)$ under the $2:1$ homomorphism \eqref{eq:PHI} is
\begin{equation*}
G_H=\mathbf{C}_{2m}\Gamma=\{\lambda\gamma\mid \lambda\in\mathbf{C}_{2m}, \gamma\in\Gamma\}.
\end{equation*}
By \cite[Proposition~7-(c)]{Auslander-Reiten:MQEDD} the McKay quiver $Q_{\rho_{\mathrm{Nat}}}(G_H)$ of the group $G_H$ with respect to the natural representation given by the inclusion $G_H\hookrightarrow U(2)$ is $[T,m]$ where $T$ is the extended Dynkin diagram corresponding to the McKay quiver $Q_{\rho_{\mathrm{Nat}}}(\Gamma)$.
On the other hand, by Theorem~\ref{thm:sfs} (see \cite[p.~98]{Coxeter:RCP}) we have that
\begin{equation*}
G_H=(\mathbf{C}_{2m},\mathbf{C}_{2m};\Gamma,\Gamma)_\phi\cong\Gamma\times\mathbf{C}_m,
\end{equation*}
where by the conditions on $m$ we have that $m$ is odd. By Proposition~\ref{prop:MQ.GxCm} (and Remark~\ref{rem:T.m}) the McKay quiver $Q_{\rho_{\mathrm{NAT}}}(\Gamma\times\mathbf{C}_m)$
is the quiver $(T,m)$. Since $m$ is odd, by Lemma~\ref{lem:T.s.prop}-\eqref{it:modd.iso} we have that $(T,m)$ and $[T,m]$ are isomorphic.

In the examples of the McKay quivers of the group $\mathbb{D}_{7,4}$ in Figure~\ref{fig:MQ.D2q.Cm}, the group $\mathbb{T}_5$ in Figure~\ref{fig:MQ.BT.C5}, the group $\mathbb{O}_5$ in Figure~\ref{fig:MQ.BO.C5}, and the group $\mathbb{I}_7$ in Figure~\ref{fig:MQ.BI.C7}, the red arrows show a basic subquiver $T^0$, the other successive basic subquivers are obtained by clockwise rotation.
\end{remark}

\begin{remark}
In Theorem~\ref{thm:sfs} the dihedral groups $\mathbb{D}_{n,q}$ with $\gcd(n,q)=1$, $m=n-q$ and $\gcd(m,2)=2$ are of the form
\begin{equation*}
(\mathbf{C}_{4m},\mathbf{C}_{2m};\mathrm{B}\mathbf{D}_{2q},\mathbf{C}_{2q})_\phi,
\end{equation*}
that is, we have $L=\mathbf{C}_{4m}$, $L_K=\mathbf{C}_{2m}$, $R=\mathrm{B}\mathbf{D}_{2q}$, $R_K=\mathbf{C}_{2q}$ and $\phi\colon \mathbf{C}_{4m}/\mathbf{C}_{2m}\to \mathrm{B}\mathbf{D}_{2q}/\mathbf{C}_{2q}$ is an
isomorphism between cyclic groups of order two. Hence, the subgroup $H$ of $\mathbb{S}^1\times SU(2)$ is
$H=\{(\lambda,\gamma)\in \mathbf{C}_{4m}\times\mathrm{B}\mathbf{D}_{2q}\mid \phi(\lambda+\mathbf{C}_{2m})=(\gamma+\mathbf{C}_{2q})\}$ and the corresponding subgroup $\mathbb{D}_{n,q}=\Phi(H)$ of $U(2)$ under the $2:1$ homomorphism \eqref{eq:PHI} is
\begin{equation*}
\mathbb{D}_{n,q}=\Phi(H)=\{\lambda\gamma\mid (\lambda,\gamma)\in H\}.
\end{equation*}
By \cite[Proposition~7-(e)]{Auslander-Reiten:MQEDD} the McKay quiver $Q_{\rho_{\mathrm{Nat}}}(\mathbb{D}_{n,q})$ of the group $\mathbb{D}_{n,q}$ relative to the natural representation is $[T,m]$ where $T$ is the extended Dynkin diagram corresponding to the McKay quiver $Q_{\rho_{\mathrm{Nat}}}(\mathrm{B}\mathbf{D}_{2q})$ given in Table~\ref{tab:MK.graph}.
On the other hand, by Theorem~\ref{thm:sfs} (see \cite[Theorem~2.7]{Arciniega-etal:CRMQSAPS}) we have that
\begin{equation*}
\mathbb{D}_{n,q}=(\mathbf{C}_{4m},\mathbf{C}_{2m};\mathrm{B}\mathbf{D}_{2q},\mathbf{C}_{2q})_\phi\cong D_{2^{k}\cdot q}\times\mathbf{C}_l,\ \text{with $m=2^{k-2}l$, $l$ odd and $k\geq3$}.
\end{equation*}
To see that the quiver $Q_{\rho_{\mathrm{NAT}}}(D_{2^{k+1}\cdot q}\times\mathbf{C}_l)$ is indeed $[T,s]$ one can find a basic subquiver as follows. The tree $T$ is given by the extended Dynkin diagram
$\tilde{D}_{q+2}$, where the vertices at the ends correspond to the one-dimensional representations and the other ones to the two-dimensional irreducible representations. Choose alternating signs for its vertices, for instance
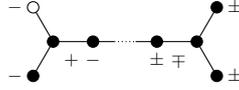
\begin{equation*}
\dynkin[scale=1.5,extended,labels={-,-,+,-,\pm,\mp,\pm,\pm}]D{}
\end{equation*}
From Subsection~\ref{sssec:MQ.Dnq} when $l=1$ and also from Proposition~\ref{prop:MQ.GxCm} when $l>1$, there are two-dimensional irreducible representations to which arrive two arrows coming from one-dimensional representations. Choose one of such two-dimensional representations and denote it by $\rho_1$ and denote by $\alpha_0$ and $\alpha_1$ the two one-dimensional representations that have arrows to $\rho_1$. There is only one two-dimensional representation $\rho_2$ with an arrow to $\rho_1$, and $\rho_2$ has also another arrow to a two-dimensional representation $\rho_3$. In turn, there is only one two-dimensional representation $\rho_4$ with an arrow to $\rho_3$, and it has also another arrow to a two-dimensional representation $\rho_5$. We follow this path of arrows between two-dimensional representations until we arrive to a two-dimensional representation $\rho_{q-1}$ which has arrows from or to two one-dimensional representations, which we denote by $\alpha_2$ and $\alpha_3$. The vertices $\alpha_0$, $\alpha_1$, $\alpha_2$, $\alpha_3$, and $\rho_1$ to $\rho_{q-1}$ and the arrows which connect them form the $0$-th basic subquiver of $Q_{\rho_{\mathrm{NAT}}}(D_{2^{k+1}\cdot q}\times\mathbf{C}_l)$, the other $m-2$ successive basic subquivers are obtained by clockwise rotation. Finding all the basic subquivers, it is easy to see that the remaining arrows which connect them are the arrows that define the quiver $[T,m]$.

In the examples of the McKay quivers of the group $D_{4(3)}$ in Figure~\ref{fig:MQ:D4.1} and of the group $D_{8(3)}\times\mathbf{C}_5$ in Figure~\ref{fig:MQ.D2k2r.Cm}, the red arrows show a basic subquiver $T^0$, where $T$ is the extended Dynkin diagram $\tilde{D}_{5}$.
In the example of the McKay quiver of the group $D_{4(5)}$ given in Figure~\ref{fig:MQ:D4.2} the red arrows show a basic subquiver $T^0$, where $T$ is the extended Dynkin diagram $\tilde{D}_{7}$.
\end{remark}

\begin{remark}
In Theorem~\ref{thm:sfs} the tetrahedral groups $\mathbb{T}_m$ with $\gcd(m,6)=3$ are of the form
\begin{equation*}
(\mathbf{C}_{6m},\mathbf{C}_{2m};\mathrm{B}\mathbf{T},\mathrm{B}\mathbf{D}_{2})_\phi,
\end{equation*}
that is, we have $L=\mathbf{C}_{6m}$, $L_K=\mathbf{C}_{2m}$, $R=\mathrm{B}\mathbf{T}$, $R_K=\mathrm{B}\mathbf{D}_{2}$ and $\phi\colon \mathbf{C}_{6m}/\mathbf{C}_{2m}\to \mathrm{B}\mathbf{T}/\mathrm{B}\mathbf{D}_{2}$ is an
isomorphism between cyclic groups of order three. Hence, the subgroup $H$ of $\mathbb{S}^1\times SU(2)$ is
$H=\{(\lambda,\gamma)\in \mathbf{C}_{6m}\times\mathrm{B}\mathbf{T}\mid \phi(\lambda+\mathbf{C}_{2m})=(\gamma+\mathrm{B}\mathbf{D}_{2})\}$ and the corresponding subgroup $\mathbb{T}_m=\Phi(H)$ of $U(2)$ under the $2:1$ homomorphism \eqref{eq:PHI} is
\begin{equation*}
 \mathbb{T}_m=\Phi(H)=\{\lambda\gamma\mid (\lambda,\gamma)\in H\}.
\end{equation*}
By \cite[Proposition~7-(g)]{Auslander-Reiten:MQEDD} the McKay quiver $Q_{\rho_{\mathrm{Nat}}}(\mathbb{T}_m)$ of the group $\mathbb{T}_m$ with respect to the natural representation given by the inclusion $\mathbb{T}_m\hookrightarrow U(2)$ is given by the quiver defined as follows: The vertices are
\begin{equation*}
\{(u,i)\mid 0\leq i<m\},\quad \{(v,i)\mid 0\leq i<3m\},\quad  \{(w,i)\mid 0\leq i<3m\},
\end{equation*}
and the arrows are given by
\begin{equation}\label{eq:AR.MQ.Tm}
\begin{tikzcd}
(w,i) \arrow[r]  & (v,i)\\
\end{tikzcd}\quad
\begin{tikzcd}
   & (w,i-1) \\
(v,i)\arrow[rd] \arrow[ru] &  \\
& (u,i)\\
\end{tikzcd}\quad
\begin{tikzcd}
   & (v,i-1)\\
(u,i)\arrow[rd] \arrow[r] \arrow[ru] & (v,m+i-1) \\
& (v,2m+i-1)
\end{tikzcd}
\end{equation}
where addition is modulo $m$ for $(u,i)$ and modulo $3m$ for $(v,i)$ and $(w,i)$.
On the other hand, by Theorem~\ref{thm:sfs} (see \cite[Theorem~2.7]{Arciniega-etal:CRMQSAPS}) we have that
\begin{equation*}
\mathbb{T}_m=(\mathbf{C}_{6m},\mathbf{C}_{2m};\mathrm{B}\mathbf{T},\mathrm{B}\mathbf{D}_{2})_\phi\cong P'_{8\cdot 3^k}\times\mathbf{C}_l,\quad\text{with $m=3^{k-1}l$}.
\end{equation*}
The McKay quiver $Q_{\rho_{\mathrm{NAT}}}(P'_{8\cdot 3^k}\times\mathbf{C}_l)$ is given in \eqref{mq:P83k.Cm}.
In order to prove that these two McKay quivers are isomorphic we need the following lemma.
\begin{lemma}\label{lem:cong}
Let $m\in\mathbb{N}$ with $\gcd(m,6)=3$, that is, $m=3r$ with $r$ odd. Then
\begin{equation*}
\frac{m(m-1)}{2}\equiv m \mod 3m.
\end{equation*}
\end{lemma}

\begin{proof}
Since $r$ is odd we have $(r-1)\equiv 0 \mod 2$, thus
\begin{align*}
9r^2-9r&\equiv 0 \mod 18r,\\
9r^2-3r&\equiv 6r \mod 18r,\\
3r(3r-1)&\equiv 2(3r) \mod 6(3r),\\
m(m-1)&\equiv 2m \mod 6m\\
\frac{m(m-1)}{2}&\equiv m \mod 3m.\qedhere
\end{align*}
\end{proof}

\begin{proposition}
Let $m\in\mathbb{N}$ with $\gcd(m,6)=3$ and $m=3^{k-1}l$ with $l$ odd. The correspondence
\begin{equation*}
\tilde{\alpha}_q\mapsto (w,\frac{3m-1}{2}q),\quad \tilde{\varrho}_q\mapsto (v,\frac{3m-1}{2}q+\frac{3m+1}{2}),\quad \tilde{\varsigma}_q\mapsto (u,\frac{m-1}{2}q+1),
\end{equation*}
gives an isomorphism between the McKay quiver $Q_{\rho_{\mathrm{NAT}}}(P'_{8\cdot 3^k}\times\mathbf{C}_l)$ given by \eqref{mq:P83k.Cm} and the McKay quiver $Q_{\rho_{\mathrm{Nat}}}(\mathbb{T}_m)$ given by \eqref{eq:AR.MQ.Tm}.
\end{proposition}

\begin{proof}
First we prove that the following diagrams commute:
\begin{equation*}
\begin{tikzcd}
\tilde{\alpha}_{q} \arrow[r]\arrow[d]  & \tilde{\varrho}_{q+1}\arrow[d]\\
(w,i)\arrow[r] & (v,i)
\end{tikzcd}\quad\quad
\begin{tikzcd}
\tilde{\varrho}_{q} \arrow[r]\arrow[d]  & \tilde{\alpha}_{q+1}\arrow[d]\\
(v,i)\arrow[r] & (w,i-1)
\end{tikzcd}\quad\quad
\begin{tikzcd}
\tilde{\varrho}_{q} \arrow[r]\arrow[d]  & \tilde{\varsigma}_{q+1}\arrow[d]\\
(v,i)\arrow[r] & (u,i)
\end{tikzcd}
\end{equation*}
For the first one we have
\begin{align*}
\tilde{\alpha}_{q}&\mapsto (w,\frac{3m-1}{2}q)\to(v,\frac{3m-1}{2}q),\\
\tilde{\alpha}_{q}&\to\tilde{\varrho}_{q+1}\mapsto (v,\frac{3m-1}{2}(q+1)+\frac{3m+1}{2})=(v,\frac{3m-1}{2}q\mod 3m).
\end{align*}
For the second one we have
\begin{align*}
\tilde{\varrho}_{q}&\mapsto(v,\frac{3m-1}{2}q+\frac{3m+1}{2})\to(w,\frac{3m-1}{2}q+\frac{3m+1}{2}-1)=(w,\frac{3m-1}{2}(q+1)),\\
\tilde{\varrho}_{q}&\to \tilde{\alpha}_{q+1}\mapsto (w,\frac{3m-1}{2}(q+1)).
\end{align*}
For the third one we have
\begin{align*}
\tilde{\varrho}_{q}&\mapsto(v,\frac{3m-1}{2}q+\frac{3m+1}{2})\to(u,\frac{3m-1}{2}q+\frac{3m+1}{2}),\\
\tilde{\varrho}_{q}&\to\tilde{\varsigma}_{q+1}\mapsto(u,\frac{m-1}{2}(q+1)+1)=(u,\frac{3m-1}{2}q+\frac{3m+1}{2}).
\end{align*}
For the arrows
\begin{equation*}
\tilde{\varsigma}_{q}\Rrightarrow\{\tilde{\varrho}_{q+1},\tilde{\varrho}_{m+q+1},\tilde{\varrho}_{2m+q+1}\}
\end{equation*}
and
\begin{equation*}
(u,i)\Rrightarrow\{(v,i-1),(v,m+i-1),(v,2m+i-1)\}.
\end{equation*}
the diagrams not always commute one-by-one (only when $q\equiv 0\mod 3$), but we shall prove that one set of three vertices is sent to the other one.
We have
\begin{equation*}
\tilde{\varsigma}_{q}\mapsto (u,\frac{m-1}{2}q+1)\Rrightarrow\Bigl\{(v,\frac{m-1}{2}q),(v,\frac{m-1}{2}q+m),(v,\frac{m-1}{2}q+2m)\Bigr\}.
\end{equation*}
\begin{multline*}
\tilde{\varsigma}_{q}\Rrightarrow\{\tilde{\varrho}_{q+1},\tilde{\varrho}_{m+q+1},\tilde{\varrho}_{2m+q+1}\}\mapsto\\
\Bigl\{(v,\tfrac{3m-1}{2}(q+1)+\tfrac{3m+1}{2}),(v,\tfrac{3m-1}{2}(m+q+1)+\tfrac{3m+1}{2}),(v,\tfrac{3m-1}{2}(2m+q+1)+\tfrac{3m+1}{2})\Bigr\}
\end{multline*}
Thus, we need to prove that the set of vertices $\{(v,\frac{m-1}{2}q+jm)\}_{j=0}^{2}$ coincides with the set of vertices
$\{(v,\tfrac{3m-1}{2}(jm+q+1)+\tfrac{3m+1}{2})\}_{j=0}^{2}$.

We have $m=3^{k-1}l$, let $r=3^{k-2}l$, thus $m=3r$ with $r$ odd. Write $q=3s+t$ with $t=0,1,2$. Then we have
\begin{align*}
&\tfrac{3m-1}{2}(jm+3s+t+1)+\tfrac{3m+1}{2}=\tfrac{3m-1}{2}(jm+3s+t)+\tfrac{3m-1+3m+1}{2}\\
&=\tfrac{3m-1+m-m}{2}(jm+3s+t)+3m=\Bigl(\tfrac{m-1}{2}+m\Bigr)(jm+3s+t)+3m\\
&=\tfrac{m-1}{2}(3s+t)+3ms+tm+\tfrac{m-1}{2}jm+jm^2+3m\\
\intertext{by Lemma~\ref{lem:cong} $\tfrac{m-1}{2}jm\equiv jm\mod 3m$ and since $m=3r$ we have $m^2=3mr$}
&\equiv\tfrac{m-1}{2}(3s+t)+(t+j)m\mod 3m,\quad j=0,1,2,\ t=0,1,2.
\end{align*}
Therefore the two sets of vertices coincide.
\end{proof}
\end{remark}

\paragraph{\textbf{Acknowledgments:}} The first author thanks Galatasaray Universitesi in Istanbul, for their kind hospitality during part of the writing of this article, and to José Antonio Arciniega-Nevarez and Agustín Romano-Velázquez for enlightening conversations.
The second author gratefully acknowledges the hospitality of the Institut des Hautes Études Scientifiques (IHES) in Paris during this work.

\end{document}